\newcommand\Modified{Sept. 18, 2010}
\documentclass{amsart}

\usepackage{amsfonts, amsmath,amscd, amssymb, latexsym, mathrsfs, stmaryrd, verbatim, wasysym }
\usepackage{graphicx}
\usepackage[all]{xy}

\newtheorem{theorem}{Theorem}
\newtheorem{definition}{Definition}

\newtheorem{corollary}[theorem]{Corollary}

\newtheorem{lemma}[theorem]{Lemma}
\newtheorem{proposition}[theorem]{Proposition}
\newtheorem{remark}[theorem]{Remark}

\numberwithin{equation}{section}
\numberwithin{theorem}{section}

 \newcommand\datver[1]{\def\datverp%
 {\par\boxed{\boxed{\text{Version: #1; Run: \today}}}}}
 
 \usepackage{color}
\definecolor{darkgreen}{cmyk}{1,0,1,.2}
\definecolor{m}{rgb}{1,0.1,1}

\renewcommand{\bar}{\overline}

\newcommand{\df}[1]{\mathfrak{#1}}
\renewcommand{\hat}[1]{\widehat{#1}}

\newcommand{\rest}[1]{\big\rvert_{#1}} 
\newcommand{\script}[1]{\textsc{#1}}
\renewcommand{\tilde}{\widetilde}
\newcommand{\wt}[1]{\widetilde{#1}}
\newcommand{\wh}{\widehat}
\newcommand\Iielaga{{}^{\mathrm{iie}}\Lambda^*_{\Gamma}} 

\newcommand{\dR}{\mathrm{dR}}
\newcommand{\Image}{\operatorname{Image}}

\newcommand\eps\varepsilon

\newcommand\pa{\partial}

\newcommand\spec{\mathrm{spec}}
\newcommand\Spec{\mathrm{Spec}}

\newcommand\ie{\operatorname{ie}}
\newcommand\iie{\operatorname{iie}}
\newcommand\pie{\operatorname{pie}}

\newcommand\Ie{{}^{\ie}}
\newcommand\Iie{{}^{\iie}}

\newcommand\ch{\operatorname{ch}}
\newcommand\id{\operatorname{id}}

\newcommand\CI{{\mathcal{C}}^{\infty}}
\newcommand\CIc{{\mathcal{C}}^{\infty}_c}
\newcommand\CmI{{\mathcal{C}}^{-\infty}}

\newcommand{\lrpar}[1]{\left( #1 \right)}

\newcommand\ang[1]{\left\langle #1 \right\rangle}
\newcommand{\lrbrac}[1]{\left\lbrace #1 \right\rbrace}
\newcommand{\norm}[1]{\lVert #1 \rVert}
\newcommand{\abs}[1]{\left\lvert #1 \right\rvert}

\DeclareMathOperator*{\btimes}{\times} 

\newcommand{\Ch}{\operatorname{Ch}}

\newcommand\Diff{\operatorname{Diff}}

\newcommand{\Dom}{\operatorname{Dom}}
\newcommand\dvol{\operatorname{dvol}}

\newcommand{\Hom}{\operatorname{Hom}}
\newcommand\Id{\operatorname{Id}}
\newcommand{\ind}{\operatorname{ind}}
\newcommand{\Ind}{\operatorname{Ind}}

\newcommand{\loc}{\operatorname{loc}}

\newcommand{\sign}{\operatorname{sign}}

\newcommand{\supp}{\operatorname{supp}}

\newcommand\Mand{\text{ and }}

\newcommand\Mforevery{\text{ for every }}

\newcommand\Mif{\text{ if }}

\newcommand\Mor{\text{ or }}

\newcommand\Mst{\text{ s.t. }}

\newcommand\Mthen{\text{ then }}
\newcommand\Mwhere{\text{ where }}

\DeclareMathAlphabet{\mathpzc}{OT1}{pzc}{m}{it}
\newcommand{\cl}[1]{\mathpzc{cl}\left( #1 \right)}

\newcommand\paperintro%
        {%
         }
\newcommand\paperbody%
        {%
         }

\newcommand\bbC{\mathbb{C}}
\newcommand\bbD{\mathbb{D}}

\newcommand\bbK{\mathbb{K}}

\newcommand\bbN{\mathbb{N}}

\newcommand\bbQ{\mathbb{Q}}
\newcommand\bbR{\mathbb{R}}
\newcommand\bbS{\mathbb{S}}

\newcommand\bbZ{\mathbb{Z}}

\newcommand\cB{\mathcal{B}}

\newcommand\cD{\mathcal{D}}
\newcommand\cE{\mathcal{E}}
\newcommand\cF{\mathcal{F}}

\newcommand\cH{\mathcal{H}}
\newcommand\cI{\mathcal{I}}

\newcommand\cK{\mathcal{K}}

\newcommand\cM{\mathcal{M}}
\newcommand\cN{\mathcal{N}}
\newcommand\cO{\mathcal{O}}
\newcommand\cP{\mathcal{P}}

\newcommand\cS{\mathcal{S}}
\newcommand\cT{\mathcal{T}}
\newcommand\cU{\mathcal{U}}
\newcommand\cV{\mathcal{V}}

\newcommand\bN{\mathbf{N}}

\newcommand{\RR}{\mathbb{R}}

\newcommand{\CC}{\mathbb{C}}

\newcommand{\e}{\epsilon}
\newcommand{\del}{\partial}

\newcommand{\calC}{{\mathcal C}}

\newcommand{\calH}{{\mathcal H}}

\newcommand{\calU}{{\mathcal U}}
\newcommand{\calV}{{\mathcal V}}
\newcommand{\calW}{{\mathcal W}}

\newcommand{\frakS}{{\mathfrak S}}

\newcommand{\ovl}{\overline}

\datver{\Modified}

\begin{document}
\title[The signature package on Witt space]{The signature package on Witt spaces.}

\author{Pierre Albin}
\address{Courant Institute and Institute for Advanced Study \newline
Current address: Department of Mathematics, University of Illinois at Urbana-Champaign}
\email{palbin@illinois.edu}
\author{Eric Leichtnam}
\address{CNRS Institut de Math\'ematiques de Jussieu}
\author{Rafe Mazzeo}
\address{Department of Mathematics, Stanford University}
\email{mazzeo@math.stanford.edu}
\author{Paolo Piazza}
\address{Dipartimento di Matematica, Sapienza Universit\`a di Roma}
\email{piazza@mat.uniroma1.it}


\begin{abstract}
In this paper we prove a variety of results about the signature operator
on Witt spaces. First, we give a parametrix construction for the signature 
operator on any compact, oriented, stratified pseudomanifold $X$ which satisfies 
the Witt condition. This construction, which is inductive over the `depth' of the 
singularity, is then used to show that the signature operator
is essentially self-adjoint and has discrete spectrum of finite multiplicity,
so that its index -- the analytic signature of $X$ -- is well-defined. This provides
an alternate approach to some well-known results due to Cheeger. We then prove some
new results. By coupling this parametrix construction to a $C^*_r\Gamma$ Mishchenko bundle 
associated to any Galois covering of $X$ with covering group $\Gamma$, we prove  
analogues of the same analytic results, from which it follows that one may define 
an analytic signature index class as an element of the
$K$-theory of $C^*_r\Gamma$.  We go on to establish in this setting and for this class the full 
range of conclusions which sometimes goes by the name of the signature package. 
In particular, we prove a new and purely topological theorem, asserting the stratified homotopy 
invariance of the higher signatures of $X$,  defined through the homology $L$-class 
of $X$, whenever the rational assembly map
$K_* (B\Gamma)\otimes\bbQ \to K_*(C^*_r \Gamma)\otimes \bbQ$ is injective.
\end{abstract}

\maketitle
\section{Introduction}
Let $X$ be an orientable closed compact Riemannian manifold with fundamental group $\Gamma$. Let $X^\prime$ 
be a Galois $\Gamma$-covering and $r: X\to B\Gamma$ a classifying map for $X^\prime$.  The {\it signature 
package} for the pair $(X, \, r\!:\!X\to B\Gamma)$ refers to the following collection of results: 
\begin{enumerate}
\item the signature operator with values in the Mishchenko bundle $r^* E\Gamma \times_\Gamma C^*_r\Gamma$
defines  a signature index class  $\Ind (\wt \eth_{\sign})\in K_* (C^*_r \Gamma)$, $* \equiv \dim X \;({\rm mod}\; 2)$; 
\item the signature index class is a bordism invariant; more precisely  it defines a group homomorphism 
$\Omega^{{\rm SO}}_* (B\Gamma) \to K_* (C^*_r \Gamma)$;
\item the signature 
index class is a  homotopy invariant; 
\item  there is a  K-homology signature class $[\eth_{\sign}]\in K_* (X)$ whose Chern
 character is, rationally, the Poincar\'e dual of the L-Class;
\item the assembly map $\beta: K_* (B\Gamma)\to K_* (C^*_r\Gamma)$ sends the
class $ r_* [\eth_{\sign}]$ into $\Ind (\wt \eth_{\sign})$;
\item  if the assembly map is rationally injective, one can deduce from (1) - (5) that the Novikov higher signatures 
$$\{\langle L(X)\cup r^* \alpha,[X] \rangle, \ \alpha\in H^* (B\Gamma,\bbQ)\}$$
are homotopy invariant.
\end{enumerate}

We call this list of results, together with the following item, the {\em full} signature package:

\smallskip
 \begin{list}
 {(7)} \item there is a ($C^*$-algebraic) symmetric signature $\sigma_{C^*_r\Gamma} (X,r)\in K_* (C^*_r \Gamma)$,
which is topologically defined, a bordism invariant $\sigma_{C^*_r\Gamma}: \Omega^{{\rm SO}}_* (B\Gamma) \to 
K_* (C^*_r \Gamma)$ and, in addition, is equal to the signature index class.
\end{list}

\medskip
For history and background see \cite{oberwolfach} \cite{rosenberg-anft} and for a survey we refer to \cite{Kasparov-contemporary}.

\medskip
{\it The main goal of this paper is to formulate and establish the signature package for a class of stratified pseudomanifolds known as Witt spaces. In particular, we prove by analytic methods
a new and purely topological result concerning the stratified homotopy invariance of suitably defined 
higher signatures under an injectivity assumption on the assembly map for the group
$\Gamma$.}

\medskip
The origins of the signature package on a closed oriented manifold $X$ can be traced back
to the Atiyah-Singer proof of the signature formula of Hirzebruch, 
$\sigma_{\mathrm{top}}(X) = \mathcal{L} (X):=\langle L(X),[X] \rangle $ .
In this proof the central
object is the Fredholm index of the signature operator which is proved to be simultaneously equal
to the topological signature of the manifold $\sigma_{\mathrm{top}}(X)$ and to its L-genus $\mathcal{L} (X)$:
$$\sigma_{\mathrm{top}}(X)=\ind  (\eth_{\sign})=\mathcal{L} (X)\,.$$
The idea of using index theory to investigate topological properties of $X$ received new impetus through 
the seminal work of Lusztig, who used the family index theorem of Atiyah-Singer in order to establish the Novikov 
conjecture on the homotopy invariance of the higher signatures of $X$ when $\pi_1 (X) = \bbZ^k$.
Most of the signature package as formulated here can be seen as a noncommutative version of the results of Lusztig. 
Crucial in the formulation and proof of the signature package are the following issues:
\begin{itemize}
\item the Poincar\'e duality property for the (co)homology of $X$ and more generally,
the Algebraic Poincar\'e Complex structure of its (co)chain complex;
\item the possibility of defining bordism groups $\Omega^{{\rm SO}} ( T)$, $T$ a topological space,
with cycles given by closed oriented manifolds endowed with a reference map to $T$;
\item an elliptic theory which allows one to establish the analytic properties
of $\eth_{\sign} $ and then connect them to the topological properties of $X$;
\item the possibility of extending this elliptic theory to signature operators 
twisted by a bundle of finitely generated projective $A$-modules, where $A$ is a $C^*$-algebra.
The prototype is the signature operator $\widetilde{\eth}_{\sign}$ 
twisted by the Mishchenko bundle $r^* E\Gamma\times_\Gamma C^*_r \Gamma$. 

\end{itemize}

\medskip

Once one moves from closed oriented manifold  to stratified pseudomanifolds, many of these issues need 
careful reformulation and substantially more care. First, it is well-known that Poincar\'e duality fails on a 
general stratified pseudomanifold $\wh{X}$. Next, the bordism group $\Omega^{{\rm pseudo}} (T)$, the cycles of which
are {\it arbitrary} stratified pseudomanifolds endowed with a reference map to $T$, is not the right one; indeed,
as explained in \cite{Banagl}, the coefficients of such a theory, $\Omega^{{\rm pseudo}} ({\rm point})$, are trivial. 
Finally, the analytic properties of the signature operator on the regular part of a stratified pseudomanifold 
endowed with an `incomplete iterated edge metric' (which is a particularly simple and natural type of metric that
can be constructed on such a space) are much more delicate than in the closed case. In particular, this operator may not even be
essentially self-adjoint, and the possibility of numerous distinct self-adjoint extensions complicates the 
possible connections to topology. 

The first problem has been tackled by Goresky and MacPherson in the topological setting \cite{GM} \cite{GM2}
and by Cheeger in the analytic setting \cite{Ch-PNAS} \cite{Cheeger-symp} (at least for the particular subclass of 
stratified pseudomanifolds we consider below). The search for a cohomology theory on such spaces with
some vestiges of Poincar\'e duality led Goresky and MacPherson to their discovery of intersection (co)homology 
groups, $I\! H^*_p ( \wh{X},\bbQ)$, where $p$ is a `perversity function', and to the existence of a perfect pairing 
$$I\! H^*_p (\wh{X},\bbQ)\times I\! H^*_q ( \wh{X},\bbQ)\to\bbQ$$
where $p$ and $q$ are complementary perversities. Notice that we still do not obtain a signature 
unless the perversities can be chosen the same, i.e.\ unless there is a perfect pairing 
$$I\! H^*_m (\wh{X},\bbQ)\times I\! H^*_m (\wh{X},\bbQ)\to\bbQ$$ for some perversity function $m$.  
Witt spaces constitute a subclass of stratified pseudomanifolds for which all of these difficulties can be overcome.

\smallskip
A stratified pseudomanifold $\wh{X}$ is a Witt space if any  even-dimensional link $L$ satisfies $I\! H_m ^{\dim L/2} (L,\bbQ)=0$, 
where $m$ is the upper-middle perversity function. Examples of Witt spaces include any singular projective variety over
${\mathbb C}$. We list some particularly interesting properties of Witt spaces:

\begin{itemize}
\item the upper-middle and lower-middle perversity functions define the same intersection cohomology groups, 
which are then denoted $I\! H^*_m (\wh{X})$;
\item there is a perfect pairing $$I\! H^*_m (\wh{X},\bbQ)\times I\! H^*_m (\wh{X},\bbQ)\to\bbQ\;;$$
in particular, there is a well defined intersection cohomology signature;
\item there are well-defined and nontrivial Witt bordism groups $\Omega^{{\rm Witt}} (T)$
(for example, these are rationally isomorphic to the connected version of $KO$-homology, $ko (T)\otimes_{\bbZ} \bbQ$);
\item there is a class of Riemannian metrics on the regular part of $\wh{X}$ for which 
\begin{itemize}
\item the signature operator is essentially self-adjoint
\item its unique self-adjoint extension has discrete spectrum of finite multiplicity
\item there is a de Rham-Hodge theorem, connecting the Hodge cohomology,
the $L^2$-cohomology and the intersection cohomology $I\! H^*_m (\wh{X},\bbC)$.
\end{itemize}
\end{itemize}
The topological results here are due to Goresky-MacPherson and Siegel. The analytic results are due initially to Cheeger,
though there is much further work in this area, see, for example, \cite{BS}, \cite{L}, \cite{Hunsicker-Mazzeo}, \cite{Sch-MSRI}.  Cheeger's 
results on the signature operator are based on a careful analysis of the heat kernel of the associated Laplacian. 

\smallskip
We have a number of goals in this article:
\begin{itemize}
\item we give a new treatment of Cheeger's result on the signature operator based on the methods
of geometric microlocal analysis;
\item this approach is then adapted to the signature operator $\widetilde{\eth}_{\sign}$
with value in the Mishchenko bundle $r^* E\Gamma\times_\Gamma C^*_r \Gamma$;
\item we carefully analyze the resulting index class, with particular emphasis on its stability property;
\item we collect this analytic information and establish the whole range
of results encompassed by the signature package on Witt spaces. In particular, we 
prove a {\it Novikov conjecture on Witt spaces} whenever the assembly map
for the fundamental group is rationally injective. We note again that this is a new and purely topological result.
\end{itemize}

This article is divided into three parts. In the first, we give a detailed account of the
resolution, through a series of blowups, of an arbitrary stratified pseudomanifolds (not necessarily 
satisfying the Witt condition) to a manifold with corners. This has been studied in the past, most 
notably by Verona \cite{Ver}; the novelty in our treatment is the introduction of iterated fibration structures, 
a notion due to Melrose, as an extra structure on the boundary faces of the resolved manifold with
corners. We also show that a manifold with corners with an iterated fibration structure can be 
blown down to a stratified pseudomanifold. In other words, the classes of stratified pseudomanifolds and 
of manifold with corners with iterated fibration structure are equivalent. Much of this material is based
on unpublished work by Richard Melrose, and we are grateful to him for letting us use and develop these
ideas here. We then describe the {\it (incomplete) iterated edge metrics}, which are the simplest type of incomplete
metrics adapted to this class of singular space. We show in particular that the space of such metrics is 
nonempty and path-connected.  We also consider, for any such metric, certain conformally related complete,
and `partially complete' metrics used in the ensuing analysis.

The second part of this article focuses on the analysis of natural elliptic operators, specifically, the de Rham
and signature operators, associated to incomplete iterated edge metrics.  Our methods are drawn from
geometric microlocal analysis. Indeed, in the case of simply stratified spaces, with only one singular stratum, 
there is a very detailed pseudodifferential theory \cite{Mazzeo:edge} which can be used for problems of this type,
and in the even simpler case of manifolds with isolated conic singularities, one may use the somewhat simpler 
$b$-calculus of Melrose, see \cite{APS Book}. In either of these cases, a crucial step is to consider the de Rham 
or signature operator associated to an incomplete edge or conic metric as a singular factor multiplying an elliptic 
operator in the edge or $b$-calculus, and then to study this latter, auxiliary, operator using methods adapted to 
the geometry of an associated complete metric $\tilde{g}$ on the interior of the resolved space $\wt{X}$. 

This idea was employed by Gil and Mendoza \cite{Gil-Mendoza} in the conic case, where $\widetilde{X}$ is a manifold 
with boundary and $\tilde g$ is a $b$- (or asymptotically cylindrical) metric, and also by Hunsicker and Mazzeo 
\cite{Hunsicker-Mazzeo}, for Witt spaces with simple edge singularities. We sketch this transformation briefly in these two cases.

First suppose that $(\wh{X},g)$ is a space with isolated conic singularity. Then we can write $\eth_{\mathrm{sign}} = r^{-1}D$, 
where $D$ is an elliptic differential $b$-operator of order $1$; in local coordinates $r \geq 0$ and $z$ on $F$ 
(so $F = \partial \wt{X}$), 
\begin{equation}\label{intr-sign-b}
D = A(r,z) \left(r\del_r + \eth_{\mathrm{sign}, F}\right).
\end{equation}
The second term on the right is the signature operator on the link $F$. Thus $D$ defines a $b$-operator on $\wt{X}$.
Mapping properties of the signature operator and regularity properties for solutions of $\eth_{\mathrm{sign}}u = 0$ are
consequences of the corresponding properties for $D$, which can be studied using the calculus of pseudodifferential 
$b$-operators.
 
Next, suppose that $\wh{X}$ has a simple edge singularity; then $\widetilde{X}$ is a manifold with fibered
 boundary and $\tilde g = r^{-2}g$ is a complete edge metric, where $r$ is the distance to the singular
stratum in $(\wh{X},g)$. Here too, $\eth_{\mathrm{sign}} = r^{-1}D$ where $D$ is an elliptic edge operator.
Locally, using coordinates $(r,y,z)$, where $r$ is as above (hence is the radial variable in the cone fibres), 
and $z \in F$ and $y$ are coordinates on the edge, we have
\begin{equation}\label{intr-sign-edge}
D = A(r,y,z) \left(r\del_r + \sum B_i(r,y,z) r\del_{y_i} + \eth_{\mathrm{sign},F}\right).
\end{equation}
Thus $D$ is an elliptic differential edge operator on $\widetilde{X}$ in the sense of \cite{Mazzeo:edge}, and
the pseudodifferential edge calculus from that paper can be used to obtain all necessary properties of $\eth_{\mathrm{sign}}$. 

One of the main elements in the $b$- and edge calculi is the use of model operators associated to an operator
such as $D$. In the $b$-calculus, $D$ is modelled near the cone point by its indicial operator; in the edge 
calculus, $D$ has two models: its indicial operator and its normal operator. The latter
captures the tangential behaviour of $D$ along the edge, as well as its asymptotic behaviour in the $r$
and $z$ directions.  Their mapping properties, as determined by the construction of
inverses for them, are key in understanding the analytic properties of $D$ and hence of $\eth_{\mathrm{sign}}$.

For iterated edge spaces, we proceed in a fairly similar way, using an inductive procedure. Let $(\wh{X},g)$ 
be an iterated edge space and $Y$ a stratum of maximal depth, so that $Y$ is a compact smooth manifold without boundary
and some neighbourhood of $Y$ in $\wh{X}$ is a cone bundle over $Y$ with each fibre a cone over a compact space $F$.
If this maximal depth is greater than one, then $F$ is an iterated edge space with depth one less than 
that of $\wh{X}$. If $r$ is the radial coordinate in this cone bundle, then $\eth_{\mathrm{sign}} = r^{-1}D$ where 
$D = A(r,y,z) \left(r\del_r + \sum B_i(r,y,z) r\del_{y_i} + \eth_{\mathrm{sign},F}\right)$.  Here $\eth_{\mathrm{sign},F}$ is the
signature operator on $F$, and is an iterated edge operator. The gain is that since $F$ is one step `simpler' 
than $\wh{X}$, by induction we can assume that the analytic properties of $\eth_{\mathrm{sign},F}$ are already known,
and from these we deduce the corresponding properties for $\eth_{\mathrm{sign}}$ on $\wh{X}$. Notice that we are 
conformally rescaling in only the `final' radial variable and appealing to the geometry of 
the {\it partially complete} metric $r^{-2}g$ on the complement of $Y$ in $\wh{X}$. 

Ideally, at this stage we could appeal to a complete pseudodifferential calculus adapted to this iterated edge
geometry. Such a calculus does not yet exist, but we can take a shorter route for the problems at hand. 
Rather than developing all aspects of this pseudodifferential theory at each step of this induction, we develop 
only certain parts of the Fredholm and regularity theory for the signature operator, and phrase these in terms
of a priori estimates rather than the sharp structure of the Schwartz kernel of a parametrix for it.
By establishing the correct set of estimates at each stage of the induction, we can prove the corresponding estimates 
for spaces of one greater depth. This involves analyzing the normal and indicial operators of the partial 
completion of $\eth_{\mathrm{sign}}$, and uses the Witt hypothesis in a crucial way. 

As noted earlier, an important feature of this approach is that it carries over directly when $\eth_{\mathrm{sign}}$ 
is coupled to a $C^*$ bundle. Hence the main theorem in the higher setting can be deduced with little extra effort from 
the techniques used for the ordinary case.  This is a key motivations for developing a geometric microlocal 
approach to replace the earlier successful methods of Cheeger. The fact that such techniques are well suited
to this higher setting has already played a role, for example, on manifolds with boundary, cf.\ the work of Leichtnam, Lott and 
Piazza \cite{LLP}   on the Novikov conjecture on manifolds with boundary and the survey \cite{LPFOURIER}.

This leads eventually to our main analytic and topological Theorems:
  
\begin{theorem}
Let $\wh{X}$ be any smoothly stratified pseudomanifold satisfying the Witt hypothesis. Let $g$ be any
adapted Riemannian metric on the regular part of $\wh{X}$. Denote by $\eth = d + \delta$ either the 
Hodge-de Rham operator $\eth_{\dR}$  or the signature operator $\eth_{\sign}$ associated to $g$. Then:
\begin{itemize}
\item[1)]
As an unbounded operator on $C_c^\infty (X,  \Iie\Lambda^*(X))\subset L^2_{\iie}(X;  \Iie\Lambda^*(X) )$, 
$\eth$ has a unique closed  extension, hence is essentially self-adjoint. 
\item[2)]
For any $\eps >0$, the domain of this unique closed extension, still denoted $\eth$, is contained in 
\begin{equation*}
	 \rho^{1-\eps}L^2_{\iie}(X; \Iie\Lambda^*(X)) \cap H^1_{\loc}(X;  \Iie\Lambda^*(X))
\end{equation*}
which is compactly included in $L^2_{\iie}(X;  \Iie\Lambda^*(X))$. 
\item[3)]
As an operator on its maximal domain endowed with the graph norm, $\eth$ is Fredholm.
\item[4)]
$\eth$ has discrete spectrum of finite multiplicity.
\end{itemize}
\label{MT1}
\end{theorem}
Items 1), 3) and 4)  have been proved by  Cheeger  \cite{Ch} (using the heat-kernel) for metrics quasi-isometric to a piecewise flat ones.



\begin{theorem}
There is a well defined signature class $[ \eth_{\sign}]\in K_* (\wh{X})$, $*=\dim \wh{X}$ $({\rm mod}\; 2)$, which is 
independent of the choice  of the adapted metric on the regular part of $\wh{X}$. When $\dim X$ is even, the index 
of the signature operator is well-defined.

If $\wh{X}' \to \wh{X}$ is a Galois covering with group $\Gamma$ and $r: \wh{X} \to B\Gamma$ is the
classifying map, then the signature operator $\wt \eth_{\sign}$ with coefficients in the Mishchenko bundle,
together with the $C^*_r\Gamma$-Hilbert module $L^2_{\iie,\Gamma}(X;\Iie\Lambda_\Gamma^*X)$ 
define an unbounded Kasparov $(\bbC,C^*_r\Gamma)$-bimodule and hence a class in $KK_* (\bbC, C^*_r \Gamma)$
$=$$K_* (C^*_r\Gamma)$,
which we call the index class associated to $\wt \eth_{\sign}$ and denote by $\Ind (\wt \eth_{\sign})\in K_* (C^*_r\Gamma)$.
If $[[\eth_{\sign}]]\in KK_*(C(\widehat{X})\otimes C^*_r\Gamma, C^*_r\Gamma)$ is the class obtained from $[\eth_{\sign}]\in 
KK_*(C(\widehat{X}),\bbC)$ by tensoring with  $C^*_r\Gamma$, then $\Ind (\wt \eth_{\sign})$ is equal to the 
Kasparov product of the class defined by the Mishchenko bundle $[\widetilde{C^*_r}\Gamma]\in  KK_0(\bbC,C(\widehat{X})\otimes 
C^*_r\Gamma)$ with  $[[\eth_{\sign}]]$:
\begin{equation}\label{tensor1}
\Ind (\wt \eth_{\sign})= [\widetilde{C^*_r}\Gamma)]\otimes [[\eth_{\sign}]]\, .
\end{equation}
In particular, the index class $\Ind (\wt \eth_{\sign})$ does not depend on the choice of the adapted metric on 
the regular part of $\wh{X}$. Finally, if $\beta:  K_* (B\Gamma)\to K_* (C^*_r\Gamma)$ denotes the assembly 
map in K-theory, then 
\begin{equation}
\beta(r_*  [\eth_{\sign}])=\Ind (\wt \eth_{\sign})\text{ in }  K_* (C^*_r\Gamma)
\end{equation}
\end{theorem}

These theorems establish property 1), the first part of property 4) and property 5) of the signature package on 
Witt spaces. The rest of the signature package is proved in the third part of this paper. 

The Witt bordism invariance of the signature index class $\Ind(\wt \eth_{\sign})$ in $  K_* (C^*_r\Gamma)$
is proved using $KK$-techniques, just as in the closed case. 

The proof that $\Ind(\wt \eth_{\sign})\in  K_* (C^*_r\Gamma)$ is a stratified homotopy invariant is more difficult.
In Section \ref{sect:shi} we follow the strategy of Hilsum and Skandalis, but encounter extra 
complications caused by the singular structure of $\wh{X}$.  To deal with these we use the interplay between 
the compact singular space $\wh{X}$ with its incomplete metric and its resolution $\widetilde{X}$ with the
conformally related complete metric.

The equality of the Chern character of the signature K-homology class $[\eth_{\sign}]\in K_* (\widehat{X})$
with the homology L-class $L_* (\widehat{X})$ had already been proved by Moscovici and Wu using Cheeger's 
methods, and we simply quote their result. The stratified homotopy invariance of the higher signatures, defined 
as the collection of numbers
$$
\{ \langle \alpha,r_* (L_* (\widehat{X})) \rangle\,,\;\; \alpha\in H^* (B\Gamma,\bbQ) \} ,
$$ 
is proved in Section \ref{section:novikov} under the hypothesis that the assembly map $\beta$ is rationally injective. 
Finally, in Section \ref{section:symmetric} we prove  the (rational) equality of our index class $ \Ind (\wt \eth_{\sign})$ 
in $K_* (C^*_r\Gamma)$ with the $C^*_r \Gamma$-symmetric signature  $\sigma^{{\rm Witt}}_{C^*_r\Gamma} (\widehat{X})$ 
obtained from the one recently defined by Banagl in $L^* (\bbQ\Gamma)$. The Witt-bordism invariance of 
$\Ind (\wt \eth_{\sign})$ and $\sigma^{{\rm Witt}}_{C^*_r\Gamma} (\widehat{X})$ plays a fundamental role in the proof of this 
last item in the signature package. 

In the brief final section, we explain where the proof of each item in the signature package may be found
in this paper. 

\medskip
\noindent
{\bf Acknowledgements.} The authors are happy to thank the referee for very helpful comments. 
P.A., R.M. and P.P. are grateful to MSRI for hospitality and financial support during the Fall Semester of 2008 when 
much of the analytic part of this paper was completed. P.A. was partly supported by NSF grant DMS-0635607002 and 
by an NSF Postdoctoral Fellowship, and thanks Stanford for support during visits; R.M. was supported by 
NSF grants DMS-0505709 and DMS-0805529, and enjoyed the hospitality and financial support of MIT, Sapienza 
Universit\`a di Roma and the Beijing International Center for Mathematical Research; P.P. wishes to thank the CNRS and 
Universit\'e Paris 6 for financial support during  visits to Institut de Math\'ematiques de Jussieu in Paris. E.L was partially 
supported during visits to Sapienza Universit\`a di Roma by CNRS-INDAM (through the bilateral agreement GENCO 
(Non commutative Geometry)) and the Italian {\it Ministero dell' Universit\`a  e della  Ricerca Scientifica} (through the 
project "Spazi di moduli e teoria di Lie").

\smallskip

The authors are particularly grateful to Richard Melrose for his help and encouragement at many stages of this project and for 
allowing them to include some of his unpublished ideas in section 2. They also thank: Markus Banagl for many interesting discussions 
on the notion of symmetric signature on Witt spaces; Shmuel Weinberger for suggesting the proof of Proposition \ref{prop:magic}; 
and Michel Hilsum for useful remarks on the original manuscript. Finally, they thank Thomas Krainer,  Eugenie Hunsicker and 
Gerardo Mendoza for helpful correspondence.

\section{{Stratified spaces and resolution of singularities}} \label{Resolution}
This section describes the class of smoothly stratified pseudomanifolds. We first
recall the notion of a stratified space with `control data'; this is a topological space which
decomposes into a union of smooth strata, each with a specified tubular neighbourhood with fixed 
product decomposition, all satisfying several basic axioms. This material is taken from the paper of
Brasselet-Hector-Saralegi \cite{BHS}, but see Verona \cite{Ver} and Pflaum \cite{Pfl}
for more detailed expositions. We also refer the reader to \cite{Mather}, \cite{Hughes-Weinberger}, 
\cite{Banagl} and \cite{Kirwan-Woolf}. 
Definitions are not entirely consistent across those sources, so one purpose of 
reviewing this material is to specify the precise definitions used here. A second goal here
is to prove the equivalence of this class of smoothly stratified pseudomanifolds
and of the class of manifolds with corners with iterated fibration structures, as introduced by Melrose. 
The correspondence between elements in these two classes is by blowup (resolution) and blowdown, respectively 
We introduce the latter class in \S 2.2 and show that any manifold with corners with 
iterated fibration structure can be blown down to a smoothly stratified pseudomanifold. The converse, 
that any smoothly stratified pseudomanifold can be blown up, or resolved, to obtain a manifold with 
corners with iterated fibration structure, is proved in \S 2.3; this resolution was already defined
by Brasselet et al.\ \cite{BHS}, cf.\ also Verona \cite{Ver}, though those authors did not phrase it 
in terms of the fibration structures on the boundaries of the resolution.  The proper definition of 
isomorphism between these spaces is subtle; we discuss this and propose a suitable definition, 
phrased in terms of this resolution, in \S 2.4. This alternate description of smoothly stratified pseudomanifolds
also helps to elucidate certain notions such as the natural classes of structure vector fields, metrics, etc.

\subsection{Smoothly stratified spaces}
\begin{definition}
A stratified space $X$ is a metrizable, locally compact, second countable space which admits a locally finite 
decomposition into a union of locally closed {\rm strata} $\frakS = \{Y_\alpha\}$, 
where each $Y_\alpha$ is a smooth (usually open) manifold, with dimension depending on the index $\alpha$. 
We assume the following:

\begin{itemize}
\item[i)] If $Y_\alpha, Y_\beta \in \frakS$ and $Y_\alpha \cap \overline{Y_\beta} \neq \emptyset$,
then $Y_\alpha \subset \overline{Y_\beta}$. 
\item[ii)] Each stratum $Y$ is endowed with a set of `control data' $T_Y$, $\pi_Y$ and $\rho_Y$;
here $T_Y$ is a neighbourhood of $Y$ in $X$ which retracts onto $Y$, $\pi_Y: T_Y \longrightarrow Y$ is 
a fixed continuous retraction and $\rho_Y: T_Y \to [0,2)$ is a proper `radial function' in this 
tubular neighbourhood such that $\rho_Y^{-1}(0) = Y$. Furthermore, we require that if $Z \in \frakS$ and 
$Z \cap T_Y \neq \emptyset$, then 
\[
(\pi_Y,\rho_Y): T_Y\cap Z \longrightarrow Y \times [0,2)
\]
is a proper differentiable submersion. 
\item[iii)] If $W,Y,Z \in \frakS$, and if $p \in T_Y \cap T_Z \cap W$ and $\pi_Z(p) \in T_Y \cap Z$,
then $\pi_Y(\pi_Z(p)) = \pi_Y(p)$ and $\rho_Y(\pi_Z(p)) = \rho_Y(p)$.
\item[iv)] If $Y,Z \in \frakS$, then
\begin{eqnarray*}
Y \cap \ovl{Z} \neq \emptyset & \Leftrightarrow & T_Y \cap Z \neq \emptyset, \\
T_Y \cap T_Z \neq \emptyset & \Leftrightarrow & Y \subset \ovl{Z}, \ Y=Z\ \  \mbox{or}\ Z \subset \ovl{Y}.
\end{eqnarray*}
\item[v)] For each $Y \in \frakS$, the restriction $\pi_Y: T_Y \to Y$ is a locally trivial fibration with
fibre the cone $C(L_Y)$ over some other stratified space $L_Y$ (called the link over $Y$), with atlas $\calU_Y = 
\{(\phi,\calU)\}$ where each $\phi$ is a trivialization $\pi_Y^{-1}(\calU) \to \calU \times C(L_Y)$, and the 
transition functions are stratified isomorphisms (in the sense of Definition \ref{correct-smooth} below) of $C(L_Y)$ which preserve the rays of 
each conic fibre as well as the radial variable $\rho_Y$ itself, hence are suspensions of isomorphisms of 
each link $L_Y$ which vary smoothly with the variable $y \in \calU$. 
\end{itemize}

If in addition we let $X_j$ be the union of all strata of dimensions less than or equal to $j$, and
require that 
\begin{itemize}
\item[vi)] $X = X_n \supseteq X_{n-1} = X_{n-2} \supseteq X_{n-3} \supseteq \ldots \supseteq X_0$ and
$X \setminus X_{n-2}$ is dense in $X$
\end{itemize}
then we say that $X$ is a stratified pseudomanifold. 
\end{definition}

Some of these conditions require elaboration: 

\medskip

$\bullet$ The depth of a stratum $Y$ is the largest integer $k$ such that there is a chain of strata
$Y = Y_k, \ldots, Y_0$ with $Y_j \subset \overline{Y_{j-1}}$ for $1 \leq j \leq k$. A stratum of maximal 
depth is always a closed manifold. The maximal depth of any stratum in $X$ is called the depth of $X$ as 
a stratified space. (Note that this is the opposite convention of depth from that in \cite{BHS}.)  

We refer to the dense open stratum of a stratified pseudomanifold $\wh{X}$ as its regular set,
and the union of all other strata as the singular set,
\[
\mathrm{reg}(\wh{X}) := \wh{X}\setminus \mathrm{sing}(\wh{X}), \qquad \mathrm{where}\qquad
\mathrm{sing}(\wh{X}) = \bigcup_{{Y\in \mathfrak S}\atop{\mathrm{depth}\, Y > 0}} Y.
\]

\smallskip

$\bullet$ If $X$ and $X'$ are two stratified spaces, a stratified isomorphism between them is 
a homeomorphism $F: X \to X'$ which carries the open strata of $X$ to the open strata of $X'$ 
diffeomorphically, and such that $\pi_{F(Y))}^\prime \circ F = F \circ \pi_Y$, $\rho_Y^\prime 
= \rho_{F(Y)} \circ F$ for all $Y \in \frakS(X)$. (We shall discuss this in more detail below.)

\smallskip

$\bullet$ If $Z$ is any stratified space, then the cone over $Z$, denoted $C(Z)$, is the space $Z \times \RR^+$ 
with $Z \times \{0\}$ collapsed to a point. This is a new stratified space, with depth one greater than $Z$
itself. The vertex $0 := Z \times \{0\} / \sim $ is the only maximal depth stratum; $\pi_0$ is the natural
retraction onto the vertex and $\rho_0$ is the radial function of the cone. 

\smallskip

$\bullet$ There is a small generalization of the coning construction. For any $Y \in \mathfrak{S}$, 
let $S_Y = \rho_Y^{-1}(1)$. This is the total space of a fibration $\pi_Y: S_Y \to Y$ with fibre $L_Y$. 
Define the mapping cylinder over $S_Y$ by $\mathrm{Cyl}\,(S_Y,\pi_Y) = S_Y \times [0,2)\, /\sim$ where 
$(c,0) \sim (c',0)$ if $\pi_Y(c) = \pi_Y(c')$. The equivalence class of a point $(c,t)$ is sometimes 
denoted $[c,t]$, though we often just write $(c,t)$ for simplicity. Then there is a stratified isomorphism 
\[
F_Y: \mathrm{Cyl}\,(S_Y,\pi_Y)  \longrightarrow T_Y;
\]
this is defined in the canonical way on each local trivialization $\calU \times C(L_Y)$ and
since the transition maps in axiom v) respect this definition, $F_Y$ is well-defined.

\smallskip

$\bullet$ Finally, suppose that $Z$ is any other stratum of $X$ with $T_Y \cap Z \neq \emptyset$,
so by axiom iv), $Y \subset \bar{Z}$. Then $S_Y \cap Z$ is a stratum of $S_Y$.

\medskip

We have been brief here since these axioms are described more carefully in the references cited above. 
Axiom v) is sometimes considered to be a consequence of the other axioms. In the topological
category (where the local trivializations of the tubular neighbourhoods are only required to
be homeomorphisms) this is true, but the situation is less clear for smoothly stratified spaces,
so we prefer to leave this axiom explicit. Let us direct the reader to \cite{Mather} and \cite{Hughes-Weinberger}
for more on this. 

We elaborate further on the definition of stratified isomorphism. This definition is strictly determined by 
the control data on the domain and range, i.e.\ by the condition that $F$ preserve the product decomposition 
of each tubular neighbourhood. It is nontrivial to prove that the same space $X$ endowed with two 
different sets of control data are isomorphic in this sense. There are other even more rigid definitions 
of isomorphism in the literature. The one in 
\cite{Pfl} requires that the spaces $X$ and $X'$ are differentiably embedded into some ambient Euclidean 
space, and that the map $F$ locally extends to a diffeomorphism of these ambient spaces. 
For example, let $X$ be a union 
of three copies of the half-plane $\RR \times \RR^+$, as follows. The first and second are embedded as 
$\{(x,y,z): z=0, y \geq 0 \}$ and $\{(x,y,z): y=0, z \geq 0\}$, while the third is given by
$\{(x,y,z): y = r \cos \alpha(x), z = r \sin \alpha(x), r \geq 0\}$ where $\alpha: \RR \to (0,\pi/2)$
is smooth. In other words, this last sheet is the union of a smoothly varying family of 
rays orthogonal to the $x$-axis, with slope $\alpha(x)$ at each slice. Requiring a stratified 
isomorphism to extend to a diffeomorphism of the ambient $\RR^3$ would make these spaces 
for different functions $\alpha(x)$ inequivalent. 
We propose a different definition below which has various advantages over either of the ones above.

\subsection{Iterated fibration structures} $\;$\newline
The definition of an iterated fibration structure was proposed by Melrose in the late '90's as the boundary 
fibration structure in the sense of \cite{Melrose:Kyoto} associated to the resolution of an iterated edge space 
(what we are calling a smoothly stratified space). It has not appeared in the literature previously (though 
we can now refer to \cite{AlbMel}, which was finished after the present paper), and we are grateful to him for allowing us to present
it here. The passage to this resolution is necessary in order to apply the methods of geometric microlocal analysis. 
A calculus of pseudodifferential iterated edge operators, when it is eventually written down fully, will yield direct proofs
of most of the analytic facts in later sections of this paper.  

Let $\wt{X}$ be a manifold of dimension $n$ with corners up to codimension $k$. This means that any point $p \in \wt{X}$ 
has a neighbourhood $\calU \ni p$ which is diffeomorphic to a neighbourhood of the origin $\calV$ 
in the orthant $(\RR^+)^\ell \times \RR^{n-\ell}$ for some $\ell \leq k$, with $p$ mapped to the origin. 
There are induced local coordinates $(x_1,\ldots, x_\ell, y_1, \ldots, y_{n-\ell})$, where each 
$x_i \geq 0$ and $y_j \in (-\e,\e)$. There is an obvious decomposition of $X$ into its interior and
the union of its boundary faces of various codimensions. We make the additional global assumption that 
each face is itself an embedded manifold with corners in $\wt{X}$, or in other words, that no boundary face
intersects itself.

We shall frequently encounter fibrations $f: \wt{X} \to \wt{X}'$ between manifolds with corners. 
By definition, a map $f$ is a fibration in this setting if it satisfies the following three properties:
$f$ is a `$b$-map', which means that if $\rho'$ is any boundary defining function in $\wt{X}'$, then $f^*(\rho')$ 
is a product of boundary defining functions of $\wt{X}$ multiplied by a smooth nonvanishing function;
next, each $q \in \wt{X}'$ has a neighbourhood $\calU$ such that $f^{-1}(\calU)$ is diffeomorphic to 
$\calU \times F$ where the fibre $F$ is again a manifold with corners; finally, we require that each fibre
$F$ be a `$p$-submanifold' in $\wt{X}$, which means that in terms of an appropriate adapted corner coordinate 
system $(x,y) \in (\RR^+)^\ell \times \RR^{n-\ell}$, as above, each $F$ is defined by setting some subset of 
these coordinates equal to $0$. 

The collection of boundary faces of codimension one play a special role, and is denoted $\calH = 
\{H_\alpha\}_{\alpha \in A}$ for some index set $A$. Each boundary face $G$ is the intersection of 
some collection of boundary hypersurfaces, $G = H_{\alpha_1} \cap \ldots \cap H_{\alpha_\ell}$,
which we often write as $H_{A'}$ where $A' = \{\alpha_1, \ldots, \alpha_\ell\} \subset A$. 

\begin{definition}[Melrose]
An iterated fibration structure on the manifold with corners $\wt{X}$ consists of the following data: 
\begin{itemize}
\item[a)] Each $H_\alpha$ is the total space of a fibration $f_\alpha: H_\alpha \to B_\alpha$,
where both the fibre $F_\alpha$ and base $B_\alpha$ are themselves manifolds with corners.
\item[b)] If two boundary hypersurfaces meet, i.e.\ $H_{\alpha \beta} := H_\alpha \cap H_\beta \neq 
\emptyset$, then $\dim F_\alpha \neq \dim F_\beta$.
\item[c)] If $H_{\alpha \beta} \neq \emptyset$ as in b), and $\dim F_\alpha < \dim F_\beta$, then the 
fibration of $H_\alpha$ restricts naturally to $H_{\alpha\beta}$ (i.e.\ the leaves of the fibration of 
$H_\alpha$ which intersect the corner lie entirely within the corner) to give a fibration of 
$H_{\alpha \beta}$ with fibres $F_\alpha$, whereas the larger fibres $F_\beta$ must be transverse to 
$H_\alpha$ at $H_{\alpha\beta}$. Writing $\del_\alpha F_\beta$ for the boundaries of these fibres at the 
corner, i.e.\ $\del_\alpha F_\beta := F_\beta \cap H_{\alpha\beta}$, then $H_{\alpha \beta}$ is also the total 
space of a fibration with  fibres $\del_\alpha F_\beta$. Finally, we assume that the fibres
$F_\alpha$ at this corner are all contained in the fibres $\del_\alpha F_\beta$, and in fact that
each fibre $\del_\alpha F_\beta$ is the total space of a fibration with fibres $F_\alpha$. 
\end{itemize}
Two spaces $\wt{X}$ and $\wt{X}'$ with iterated fibration structures are isomorphic precisely when there 
exists a diffeomorphism $\Phi$ between these manifolds with corners which preserves all of the fibration 
structures at all boundary faces. 
\end{definition}

The index set $A$ has a partial ordering: the ordered chains $\alpha_1 < \ldots < \alpha_r$ in $A$ are in 
bijective correspondence with the corners $H_{A'} := H_{\alpha_1} \cap \ldots \cap H_{\alpha_r}$, 
$A' = \{\alpha_1, \ldots, \alpha_r\}$, where by definition $\alpha_i  < \alpha_j$ 
if $\dim F_{\alpha_i} < \dim F_{\alpha_j}$.  In particular, $\alpha < \beta$ implies $H_\alpha \cap H_\beta \neq \emptyset$. 
We say that $H_\alpha$ has depth $r$ if the longest chain $\beta_1 < \beta_2 < \ldots < \beta_r$ in $A$ 
with maximal element $\beta_r = \alpha$ has length $r$. The depth of a manifold with corners with iteration fibration structure 
is the maximal depth of any of its boundary hypersurfaces, equivalently, the maximal codimension of any of its corners. 
The precise relationships between the induced fibrations on each corner are not easy to describe in general, but these 
do not play a role here. 

\begin{lemma}
If $\alpha < \beta$, then the boundary of each fibre $F_\alpha \subset H_\alpha$ is disjoint from the interior of 
$H_{\alpha \beta}$.  Furthermore, the restriction of $f_\alpha$ to $H_{\alpha \beta}$  has image lying within $\del B_\alpha$, 
whereas the restriction of $f_\beta$ to $H_{\alpha \beta}$ has image intersecting the interior of $B_\beta$. In particular, 
if $\alpha$ and $\beta$ are, respectively, minimal and maximal elements in $A$, then the fibres $F_\alpha$ and the 
base $B_\beta$ are closed manifolds without boundary.
\label{le:ifs}
\end{lemma}
\begin{proof} Choose adapted local coordinates $(x_\alpha,x_\beta,y_1, \ldots, y_{n-2})$ in $H_{\alpha\beta}$ which simultaneously 
straighten out these fibrations. Thus $(x_\beta, y)$ are coordinates on $H_\alpha = \{x_\alpha = 0\}$, and there is a splitting
$y = (y', y'')$ so that $(x_\beta, y', y'') \mapsto (x_\beta, y'')$ represents the fibration $H_\alpha \to B_\alpha$.  
By part c) of the definition, since $\dim F_\alpha < \dim F_\beta$, there is a further decomposition $y'' = (y_1'', y_2'')$ so that 
the fibration of $H_{\alpha \beta}$ with fibres $\del_\alpha F_\beta$ is represented by $y \mapsto y_2''$. 
Thus $(x_\beta,y'')$ and $y'$ are local coordinates on $B_\alpha$ and each $F_\alpha$, and $y_2''$ and $(x_\alpha, y', y_1'')$
are local coordinates on $B_\beta$ and each $F_\beta$, respectively. All the assertions are direct consequences of this.
\end{proof}

Unlike for smoothly stratified spaces, the structure of control data has not been incorporated into this definition of 
iterated fibration structures, because its existence and uniqueness can be inferred from standard
facts in differential topology. Nonetheless, these data are still useful, and we discuss them now.
\begin{definition}
Let $\wt{X}$ be a manifold with corners with an iterated fibration structure. Then a control data set for $\wt{X}$
consists of a collection of triples $\{\tilde{T}_H,, \tilde{\pi}_H, \tilde{\rho}_H\}$, one for each $H \in \calH$, where
$\tilde{T}_H$ is a collar neighbourhood of the hypersurface $H$, $\tilde{\rho}_H$ is a defining 
function for $H$ and $\tilde{\pi}_H$ is a diffeomorphism from each slice $\tilde{\rho}_H = \mbox{const.}$
to $H$. Thus the pair $(\tilde{\pi}_H,\tilde{\rho}_H)$ gives a diffeomorphism $\tilde{T}_H \to 
H \times [0,2)$, and hence an extension of the fibration of $H$ to all of $\tilde{T}_H$. 
These data are required to satisfy the following additional properties: for any hypersurface $H'$ which
intersects $H$ with $H' < H$, the restriction of $\tilde{\rho}_H$ to $H' \cap \tilde{T}_H$ is constant
on the fibres of $H'$; finally, near any corner $H_{A'}$, $A' = \{\alpha_1, \ldots, \alpha_r\}$, the extension
of the set of fibrations of $H_{A'}$ induced by the product decomposition 
\[
\left.(\tilde{\pi}_{H_{\alpha_j}}, \tilde{\rho}_{H_{\alpha_j}})\right|_{\alpha_j \in A'}: 
\bigcap_{j=1}^r \tilde{T}_{H_{\alpha_j}} \cong H_{A'} \times [0,2)^r
\]
preserves all incidence and inclusion relationships between the various fibres. 
\label{de:controldata}
\end{definition}

The existence of control data for an iterated fibration structure on a manifold with corners $\wt{X}$
is discussed in \cite[Proposition 3.7]{AlbMel}, so we make only a few remarks here.  We can find some set of control
data by successively choosing the maps $\tilde{\pi}_H$ and defining functions $\tilde{\rho}_H$ in order 
of increasing depth, at each step making sure to respect the compatibility relationships with all previous 
hypersurfaces. The uniqueness up to diffeomorphism can be established in much the same way, based 
on the fact that there is a unique product decomposition of a collar neighbourhood of any $H$ up to diffeomorphism. 
\begin{proposition}
Let $\wt{X}$ be a manifold with corners with iterated fibration structure, and suppose that
$\{\tilde{\pi}_H,\tilde{\rho}_H\}$ and $\{\tilde{\pi}_H',\tilde{\rho}_H'\}$ are two sets
of control data on it. Then there is a diffeomorphism $\tilde{f}$ of $\wt{X}$ which preserves the 
iterated fibration structure, and which intertwines the two sets of control data.
\label{pr:eqcd}
\end{proposition}
The key idea in the proof is that we can pull back any set of control data on 
$\wt{X}$ to a `universal' set of control data defined on the union of the inward pointing normal 
bundles to each boundary hypersurface which satisfies the obvious set of compatibility conditions.
The fact that any two such sets of `pre-control data' are equivalent can then be deduced inductively
using standard results about uniqueness up to diffeomorphism of collar neighbourhoods of
these boundary hypersurfaces.

Finally, note that if $\wt{X}$ has an iterated fibration structure, then any corner $H_{A'}$ inherits such a 
structure too (we forget about the fibration of its interior), with depth equal to $k - \mathrm{codim}\, H_{A'}$. 

\begin{proposition} If $\wt{X}$ is a compact manifold with corners with an iterated fibration structure, 
then there is a smoothly stratified space $\wh{X}$ obtained from $\wt{X}$ by a process of successively 
blowing down the connected components of the fibres of each hypersurface boundary of $\wt{X}$ in order 
of increasing fibre dimension (or equivalently, of increasing depth). The corresponding blowdown map will be
denoted $\beta: \wt{X} \to \wh{X}$. 
\label{pr:blowdown}
\end{proposition}
\begin{proof}
We warm up to the general case by first considering what happens when $\wt{X}$ is a manifold with boundary, 
so $\del \wt{X}$ is the total space of a fibration with fibre $F$ and base space $Y$ and both $F$ and $Y$ are closed 
manifolds. Choose a boundary defining function $\rho$ and fix a product decomposition $\del \wt{X} \times 
[0,2)$ of the collar neighbourhood $\calU = \{\rho < 2\}$. This defines a retraction $\tilde{\pi}: \calU \to \del X$, 
as well as a fibration of $\calU$ over $\del \tilde{X}$ with fibre $\tilde{\pi}^{-1}(F) = F \times [0,2)$. 
Now collapse each fibre $F$ at $x=0$ to a point. This commutes with the restriction to 
each $F \times [0,2)$,  
so we obtain a bundle of cones $C(F)$ over $Y$. We call this space the blowdown of $\tilde{X}$ along the fibration, 
and write it as $X/F$. Denote by $T_Y$ the image of $\calU$ under this blowdown. 
The map $\tilde{\pi}$ induces a retraction map $\pi(\calU) = T_Y \to Y$, and $\rho$ also descends to $T_Y$. 
Thus $\{T_Y, \pi,\rho\}$ are the control data for the singular stratum $Y$, and it is easy to check that 
these satisfy all of the axioms in \S 2.1, hence $X/F$ is a smoothly stratified space.

Now turn to the general case, which is proved by induction on the depth. As in the next subsection, 
where we we follow an argument from \cite{BHS} and show how to blow up a smoothly stratified space, 
we use a `doubling construction' to stay within the class of stratified pseudomanifolds while
applying the inductive hypothesis to reduce the complexity of the problem.  To set this up, beginning with
$\wt{X}$, a manifold with corners with iterated fibration structure of depth $k$, form a new manifold with 
corners and iterated fibration structure of depth $k-1$ by simultaneously 
doubling $\wt{X}$ across all of its maximal depth hypersurfaces. In other words, consider
\[
\wt{X}' = \left((\wt{X} \times {-1}) \sqcup (\wt{X} \times {+1}) \right)/\sim
\]
where $(p,-1) \sim (q,+1)$ if and only if $p = q \in H \in \calH$ where $\mbox{depth}\,(H) = k$. 
By standard arguments in differential topology, one can give $\wt{X}'$ the structure 
of a manifold with corners up to codimension $k-1$. If $H_j \in \calH$ is any face with depth $j < k$
which intersects a face $H_k$ of depth $k$, then as in Lemma \ref{le:ifs}, the boundaries of the fibres 
$F_j \subset H_j$ only meet the corners $H_{ij}$ for $i < j$, and do not meet the interior of $H_{jk}$.
In terms of the local coordinates $(x_j, x_k, y)$ in that Lemma, we simply let $x_k$ vary in $(-\e,\e)$ rather 
than just $[0,\e)$, and it is clear how to extend the fibrations accordingly.

The dimensional comparisons and inclusion relations at all other corners remain unchanged.
Therefore, $\wt{X}'$ has an iterated fibration structure. This new space also carries a smooth involution
which has fixed point set the union of all depth $k$ faces, where the two copies of $\wt{X}$ are joined, as 
well as a function $\rho_k$ which is positive on one copy of $\wt{X}$, negative on the other, and which 
vanishes simply on the interface between the two copies of $\wt{X}$. For simplicity, 
assume that there is only one depth $k$ face, $H_k$. We can also choose $\tilde{\rho}_k$ so that 
it is constant on the fibres of all other boundary faces, and a retraction $\tilde{\pi}_k$ defined on
the set $|\tilde{\rho}_k| < 2$ onto $H_k$. 

Now apply the inductive hypothesis to blow down the boundary hypersurfaces of $\wt{X}'$ in order of 
increasing fibre dimension to obtain a smoothly stratified space $\wh{X}'$ of depth $k-1$. The 
function $\rho_k$ descends to a function (which we give the same name) on this space. Consider
the open set $\wh{X}^+ := \wh{X}' \cap \{\rho_k > 0\}$, and also $\del_k \wh{X} := 
\wh{X}' \cap \{\rho_k = 0\}$. Both of these are smoothly stratified spaces; for the former this
is because (in the language of \cite{BHS}) we are restricting to a `saturated' open set of $\wh{X}'$,
though we do not need to appeal to this terminology since the assertion is clear, whereas for the latter 
it follows by induction since it is the blowdown of $H_k$, which has depth less than $k$. This space
$\del_k \wh{X}$, which we denote by $\wh{H_k}$ is the total space of a fibration induced
from the fibration of the face $H_k$ in $\wt{X}$. By Lemma \ref{le:ifs}, since the $H_k$ are
maximal, the base $B_k$ has no boundary, and the fibres $\wt{F}_k$ are manifolds with corners with 
iterated fibration structures of depth less than $k$. 

Hence after the blowdown, the base of the fibration of $\del_k \wh{X}$ is 
still $B_k$ while the fibres are the blowdowns $\wh{F}_k$ of the spaces $\wt{F}_k$, which are again well 
defined by induction. Finally, using the product decomposition of a neighbourhood of $H_k$ in $\wt{X}$, 
collapsing the fibres of $H_k$ identifies the blowdown of this neighbourhood with the mapping cylinder
for the fibration of $\del_k \wh{X}$. This produces the final space $\wh{X}$. 

It suffices to check that the stratification of $\wh{X}$ satisfies the axioms of a smoothly stratified space 
only near where this final blowdown takes place, since the inductive hypothesis guarantees that they hold
elsewhere. These axioms are not difficult to verify from the local description of $\wt{X}$ in a product 
neighbourhood of $H_k$.
\end{proof}
\begin{remark}
There is a subtlety in this result since there is typically more than one smoothly stratified space $\wh{X}$ which
may be obtained by blowing down a manifold with corners $\wt{X}$ with iterated fibration structure. More specifically,
there is a minimal blowdown, which associates to each connected hypersurface boundary of $\wt{X}$ a 
stratum of the blowdown $\wh{X}$. However, it may occur that two strata of $\wh{X}$ of highest depth,
for example, are diffeomorphic, and after identifying these strata we obtain a new smoothly stratified
space. It may not be easy to quantify the full extent of nonuniqueness, but we do not attempt (nor need) this here.  
\label{nonuniqueness}
\end{remark}

\subsection{The resolution of a smoothly stratified space}$\;$\newline
The other part of this description of the differential topology of smoothly stratified spaces is the 
resolution process: namely, conversely to the blowdown construction above, if $\widehat{X}$
is any smoothly stratified space, one may resolve its singularities by successively blowing up 
its strata in order of decreasing depth to obtain a manifold with corners $ \wt{X} $ with iterated fibration 
structure.  Following Remark~\ref{nonuniqueness}, two different smoothly stratified spaces $\widehat{X}_1, 
\widehat{X}_2$ may resolve to the same manifold with corners $\widetilde{X}$. 

\begin{proposition}\label{prop:BlowUp}
Let $\wh{X}$ be a smoothly stratified pseudomanifold. Then there exists a manifold with corners $\wt{X}$ 
with an iterated fibration structure, and a blowdown map $\beta: \widetilde{X} \to \widehat{X}$
which has the following properties:
\begin{itemize}
\item there is a bijective correspondence $Y \leftrightarrow \wt{X}_Y$ between the strata $Y \in \mathfrak S$ 
of $\wh{X}$ and the (possibly disconnected) boundary hypersurfaces of $\wt{X}$ which blow down to these strata;
\item $\beta$ is a diffeomorphism between the interior of $\wt{X}$ and the regular set of $\wh{X}$;
we denote by $X$ this open set, which is dense in either $\wt{X}$ or $\wh{X}$; 
\item $\beta$ is also a smooth fibration of the interior of each boundary hypersurface $\wt{X}_Y$ 
with base the corresponding stratum $Y$ and fibre the regular part of the link of $Y$ in $\wh{X}$;
moreover, there is a compactification of $Y$ as a manifold with corners $\wt{Y}$ such
that the extension of $\beta$ to all of $\wt{X}_Y$ is a fibration with base $\wt{Y}$ and fibre
$\wt{L_Y}$; finally, each fibre $\wt{L_Y} \subset \wt{X}_Y$ is a manifold with corners with iterated 
fibration structure and the restriction of $\beta$ to it is the blowdown onto the smoothly stratified space
$\bar{Y}$. 
\end{itemize}
\end{proposition}

We sketch the proof, adapting the construction from \cite{BHS}, to which we refer for
further details. The proof is inductive: if $\wh{X}$ has depth $k$ and we simultaneously blow up 
the union of the depth $k$ strata to obtain a space $\wh{X}_1$, then all the control data of the 
stratification on $\wh{X}$ lifts to give $\wh{X}_1$ the structure of a smoothly 
stratified space of depth $k-1$. Iterating this $k$ times completes the proof. 
However, in order to stay within the category of smoothly stratified pseudomanifolds, 
which by definition have no codimension one boundaries, we proceed as in the proof of Proposition \ref{pr:blowdown} 
(and as in \cite{BHS}) and construct a space $\wh{X}_1'$ which is the double across the boundary hypersurface of 
the blowup of $\wh{X}$ along its depth $k$ strata, and show that $\wh{X}'_1$ is a smoothly stratified 
pseudomanifold of depth $k-1$. This space $\wh{X}'_1$ is equipped with an involution $\tau_1$ 
which interchanges the two copies of the double; the actual blowup is the closure of one component 
of the complement of the fixed point set of this involution. Iterating this $k$ times, we obtain 
a smooth compact manifold $\wh{X}'_k$ equipped with $k$ commuting involutions $\{\tau_j\}_{j=1}^k$;
the manifold with corners we seek is any one of the $2^k$ fundamental domains for this action.
\begin{proof}
To begin, fix a stratum $Y$ which has maximal depth $k$; this is a smooth closed manifold. Recall 
the notation from \S 2.1, and in particular the stratified isomorphism $F_Y$ from the mapping cylinder of 
$(S_Y,\pi_Y)$ to $T_Y$ and the family of local trivializations $\phi:\pi_Y^{-1}(\calU) \to \calU \times C(L_Y)$ 
for suitable $\calU \subset Y$. If $u \in T_Y \cap \pi_Y^{-1}(\calU)$, we write $\phi(u) = (y,z,t)$ where 
$y \in \calU$, $z \in L_Y$ and $t = \rho_Y(u)$; by axiom v), there is a retraction $R_Y: 
T_Y \setminus Y \to S_Y$, given on any local trivialization by $(y,z,t) \to (y,z,1)$ 
(which is well defined since $t \neq 0$). 

To construct the first blowup, assume for simplicity that there is only one stratum $Y$ of 
maximal depth $k$. Define
\begin{equation}
\wt{X}_1' := \left((\wh{X} \setminus Y) \times \{-1\}\right) \sqcup \left((\wh{X} \setminus Y) \times \{+1\}\right)
\sqcup \big(S_Y \times (-2,2) \big) / \sim
\label{first-blow-up}
\end{equation}
where (if $\e = \pm 1$), 
\begin{equation}
(p, \e) \sim (R_Y(p),\rho_Y(p)) \quad\text{if}\quad p \in T_Y\setminus Y \  \text{and}\ \e t  >0.
\label{eq-relation}
 \end{equation}
Let $\wh{X}' = (\wh{X} \times \{-1\}) \sqcup (\wh{X} \times \{+1\}) / \sim$ where
$(u,\e) \sim (u',\e')$ if and only if $u=u' \in Y$.  Note that $\wt{X}_1' \setminus S_Y 
\times\{0\}$ is naturally identified with $\wh{X}' \setminus Y$, so this construction replaces $Y$ with $S_Y$. 

There is a blowdown map $\beta_1: \wt{X}_1' \to \wh{X}'$ given by
\[
\beta_1(u,\e) = (u,\e) \quad \text{if}\quad u \notin Y, \qquad
\beta_1(u,0) = \pi_Y(u). 
\]
Clearly $\beta_1: \wt{X}_1'  \setminus S_Y \times \{0\} \to \wh{X}' \setminus Y$ is an isomorphism
of smoothly stratified spaces and $(S_Y \times (-2,2))$ is a tubular neighbourhood of $(\beta_1)^{-1}Y= 
S_Y\times \{0\}$ in $\wt{X}_1'$.

We shall prove that $\wt{X}_1'$ is a smoothly stratified space of depth $k-1$ equipped with an involution
$\tau_1$ which fixes $S_Y \times \{0\}$ and interchanges the two components of the complement of this
set in $\wt{X}_1'$, and which fixes all the control data of $\wt{X}_1'$. 
To do all of this, we must fix a stratification ${\mathfrak S}_1$ of $\wt{X}_1'$ and define all of
the corresponding control data and show that these satisfy properties i) - vi). 

\medskip

$\bullet$ Fix any stratum $Z\in \mathfrak{S}$ of $\wh{X}$ with $\text{depth}\,(Z) < k$, and define
\begin{equation}
\wt{Z}_1' := (Z \times \{\pm 1\}) \sqcup \left( (S_Y \cap Z)\times (-2,2)\right)/ \sim, 
\label{new-stratum}
\end{equation}
where $\sim$ is the same equivalence relation as in \eqref{eq-relation}. The easiest way to see that this is 
well-defined is to note that $S_Y \cap Z$ is a stratum of the smoothly stratified space $S_Y$ and that 
the restriction 
\begin{equation}
F_Y : \mathrm{Cyl}\,(S_Y \cap \bar{Z}, \pi_Y) \longrightarrow \bar{Z} \cap T_Y
\label{eq:cylstrat}
\end{equation}
is an isomorphism. (This latter assertion follows from axiom ii).)

As above, let $Z'$ be the union of two copies of $\bar{Z}$ joined along $\bar{Z} \cap Y$. 

\smallskip

$\bullet$ Now define the stratification ${\mathfrak S}_1$ of $\wt{X}_1'$
\begin{equation}\label{stratification-w-b}
{\mathfrak S}_1 := \{\wt{Z}_1': Z \in \mathfrak{S}\setminus Y\}. 
\end{equation}
We must now define the control data $\{ T_{\wt{Z}_1'}, \pi_{\wt{Z}_1'}, \rho_{\wt{Z}_1'}\}_{\wt{Z}_1'
\in \mathfrak{S}_1}$ associated to this stratification. 

\smallskip

$\bullet$ Following \eqref{new-stratum}, set
\begin{equation}\label{tube-for-z}
T_{\wt{Z}'_1}:= T_Z \times \{\pm 1\} \sqcup \left((S_Y\cap T_Z)\times (-2,2)\right)/\sim,
\end{equation}
where $(p,\e) \sim (c,t)$ if $t\e >0$ and $p=F_Y (c,|t|)$. Extending (or `thickening') (\ref{eq:cylstrat}),
by axiom iii) we also have that $F_Y$ restricts to an isomorphism between $\mbox{Cyl}\,(T_Z \cap S_Y ,\pi_Y) $ 
and $T_Z \cap T_Y$. In turn, using axiom ii) again, within the smoothly stratified space $S_Y$, $F_{T_Y \cap Z}$ 
is an isomorphism from $\mbox{Cyl}\,(S_Y \cap S_Z, \pi_Z)$ to the tubular neighbourhood of $Z \cap S_Y$ in $S_Y$, 
which is the same as $T_{S_Y \cap Z}$. Using these representations, the fact that (\ref{tube-for-z}) is well-defined 
follows just as before. 

Note that $Y$ has been stretched out into $S_Y \times \{0\}$, and $T_{\wt{Z}_1'} \cap (S_Y \times \{0\})$ is
isomorphic to $T_{\wt{Z}_1'} \cap (S_Y \times \{t\})$ for any $t \in (-2,2)$. 

\smallskip

$\bullet$ The projection $\pi_{\wt{Z}_1'}$ is determined by $\pi_Z$ on each slice $(S_Y \cap T_Z) 
\times \{t\}$, at least when $t \neq 0$, and extends uniquely by continuity to the slice at $t=0$
in $\wt{X}_1'$. A similar consideration yields the function $\rho_{\wt{Z}_1'}$. 

\smallskip

$\bullet$ One must check that the space $\wt{X}_1'$ and this control data for its stratification
satisfies axioms i) - vi). This is somewhat lengthy but straightforward, so details are left
to the reader. 

\smallskip

$\bullet$ Finally, this whole construction is symmetric with respect to the reflection $\tau_1$
defined by $t \mapsto -t$ in $T_Y$ and which extends outside of $T_Y$ as the interchange of the 
two components of $X' \setminus Y$. The fixed point set of $\tau_1$ is the slice $S_Y \times \{0\}$.

\medskip

This establishes that the space $\wt{X}_1'$ obtained by resolving the depth $k$ smoothly stratified
space $\wh{X}$ along its maximal depth strata via this doubling-blowup construction is a smoothly
stratified space of depth $k-1$, equipped with one extra piece of data, the involution $\tau_1$. 

This process can now be iterated. After $j$ iterations we obtain a smoothly stratified
space $\wt{X}_j'$ of depth $k-j$ which is equipped with $j$ commuting involutions $\tau_i$,
$1 \leq i \leq j$. In particular, the space $\wh{X}_k'$ is a compact closed manifold. 

It is easy to check, e.g.\ using the local coordinate descriptions, that these involutions are `independent' 
in the sense that for any point $p$ which lies in the fixed point set of more than one of the $\tau_i$, the 
$-1$ eigenspaces of the $d\tau_i$ are independent. 

The complement of the union of fixed point sets of the involutions $\tau_i$ has $2^k$ components, 
and $\wt{X}$ is the closure of any one of these components. 

The construction is finished if we show that $\wt{X}$ carries the structure of a manifold
with corners with iterated fibration structure. We proved already that $\wt{X}$ has
the local structure of a manifold with corners, but we must check that the boundary
faces are embedded. For this, first note that all faces of the resolution of $\wh{X}_1'$ are embedded, 
and by its description in the resolution construction, $H_k$ is as well; finally, all corners of $\wt{X}$ 
which lie in $H_k$ are embedded since they are faces of the resolution of $S_Y$ where $Y$ is the maximal 
depth stratum and we may apply the inductive hypothesis. This proves that $\wt{X}$ is a manifold with corners. 

Now examine the structure on the boundary faces inductively. The case $k=1$ is obvious since then 
$\wt{X}$ is a manifold with boundary; $\del\wt{X}$ is the total space 
of a fibration and there are no compatibility conditions with other faces. 
Suppose we have proved the assertion for all spaces of depth less
than $k$, and that $X$ is a smoothly stratified space of depth $k$. Let $Y$ be the union
of all strata of depth $k$ and consider the doubled-blowup space $\wt{X}_1'$. This is
a stratified space of depth $k-1$, so its resolution is a manifold with corners up to
codimension $k-1$ with iterated fibration structure. Since $S_Y$ is again a
smoothly stratified space of depth $k-1$, its resolution $\wt{S_Y}$ is also a manifold with 
corners with iterated fibration structure. The blowdown of $\wt{S_Y}$ along the fibres
of all of its boundary hypersurfaces is a smoothly stratified space $\wh{S_Y}$ and
this is the boundary $H_k$ of $\wt{X}_1$, the `upper half' of $\wt{X}_1'$. 

Once we have performed all other blowups, we know that the compatibility conditions are satisfied 
at every corner except those which lie in $\wt{S_Y}$. The images of the other boundaries of 
$\wt{X}_1$ by blowdown into $\wh{X}_1$ are the singular strata of this space. 
Furthermore, there is a neighbourhood of $H_k'$ in $\wt{X}_1$ of the form $S_Y \times [0,2)$ (using 
the variable $t$ in this initial blowup as the defining function $\rho_k$), so that near $H_k$, 
$\wt{X}$ has the product decomposition $\wt{S_Y} \times [0,2)$. From this it follows that each
fibre $F_j$ of $H_j$, $j < k$, lies in the corresponding corners $H_k \cap H_j$; it also follows
that each fibre $F_k$ of $H_k$ is transverse to this corner, and has boundary $\del_j F_k$ equal 
to a union of the fibres $F_j$. This proves that conditions a) - c) of the iterated fibration structure are satisfied. 
\end{proof}

\subsection{Smoothly stratified isomorphisms} $\;$ \newline
We now return to a closer discussion of a good definition of isomorphism between smoothly stratified spaces. 
Following Melrose, these isomorphisms are better understood through their lifts to the resolutions. 

To begin, we state a result which is a straightforward consequence of the resolution and blowdown
constructions above. 
\begin{proposition} Let $\wh{X}$ and $\wh{X}'$ be two smoothly stratified spaces and $\wt{X}$, $\wt{X}'$
their resolutions, with blowdown maps $\beta: \wt{X} \to \wh{X}$ and $\beta': \wt{X}' \to \wh{X}'$.
Suppose that $\hat f : \wh{X}\to \wh{X}'$ is a stratified isomorphism as in \cite{BHS}, \S 2. 
Then there is a unique 
diffeomorphism of manifolds with corners $\tilde{f}: \wt{X}\to  \wt{X}'$ which preserves the iterated 
fibration structures and which satisfies $\hat{f}\circ \beta=\beta' \circ \tilde{f}$.
\end{proposition}
\begin{proof}
If such a lift exists at all, it must be unique simply because it is defined by continuous extension from a map defined 
between the interiors of $\wt{X}$ and $\wt{X}'$. Because of this uniqueness, it suffices to prove the existence of the 
lift in local coordinates, and this is done in \cite{BHS}, \S 2 Prop. 3.2 and Remark 4.2. Of course, since those authors 
are not using the notion of iterated 
fibration structures, they do not consider the issue of whether the lift preserves the fibrations at the boundaries;
however, a cursory inspection of their proof shows that the map they construct does have this property.
\end{proof}

The converse result is also true, up to a technical point concerning connectedness of the links.
\begin{proposition} Given $\wt{X}$, $\wt{X}'$, $\wh{X}$ and $\wh{X}'$, as above, suppose that $\tilde{f}: 
\wt{X} \to \wt{X}'$ is a diffeomorphism of manifolds with corners which preserves the fibration
structures at the boundaries. Suppose furthermore that $\wh{X}$ and $\wh{X}'$ are the minimal
blowdowns of $\wt{X}$, $\wt{X}'$ in the sense of Remark~\ref{nonuniqueness}. 
Then there exists some choice of control data on the blown down
spaces and a stratified isomorphism $\hat{f}: \wh{X} \to \wh{X}'$ such that $\hat{f}\circ \beta=\beta' \circ \tilde{f}$.
\label{pr:diffeosss}
\end{proposition}
\begin{proof}
As above, $\hat{f}$ is uniquely determined over the principal dense open stratum
of $\wh{X}$. The fact that $\tilde{f}$ preserves the fibration structures means that
$\hat{f}$ extends to a continuous map $\wh{X} \to \wh{X}'$. However, this extension is not
a stratified isomorphism unless we use the correct choices of control data on all these spaces.
Thus fix control data on $\wt{X}$; this may be pushed forward to control data on $\wt{X}'$ via 
$\tilde{f}$.  Any set of control data for a manifold with corners with iterated fibration
structure  can be pushed down to a set of control data on its blowdown.
Therefore we have now induced control data on $\wh{X}$ and $\wh{X}'$, and it follows
from this construction that the induced map $\hat{f}$ intertwines these sets of control data,
as required.
\end{proof}

Combined with Proposition~\ref{pr:eqcd}, this gives another proof of the result from \cite{BHS} that 
any two sets of control data on a smoothly stratified space $\wh{X}$ are equivalent by a smoothly stratified
isomorphism.

This discussion motivates the following
\begin{definition}\label{correct-smooth}
A smoothly stratified map $\hat{f}$ between smoothly stratified spaces $\widehat{X}$
and $\widehat{X}^\prime$ is a continuous map $f:\widehat{X}\to \widehat{X}^\prime$ 
sending the open strata of $\widehat{X}$ smoothly into the open strata of $\widehat{X}^\prime$
and for which there exists a lift $\tilde{f}: \widetilde{X}\to \widetilde{X}^\prime$,  
$\hat{f}\circ \beta=\beta' \circ \tilde{f}$, which is a $b$-map of manifolds with corners 
preserving the iterated fibration structures.
\end{definition}
This definition has the advantage that it is not inductive (even though many of the arguments 
behind it are), and it provides a clear notion of the regularity of these isomorphisms on approach 
to the singular set. 

\section{Iterated edge metrics. Witt spaces.} \label{section:iterated}

In this section we first introduce the class of Riemannian metric on smoothly stratified spaces with which we shall
work throughout this paper. These metrics are only defined on $\mbox{reg}\,(\wh{X})$,
but the main point is their behaviour near the singular strata. These metrics were also considered
by Brasselet-Legrand \cite{BL}; closely related metrics had been considered by Cheeger \cite{Ch}; they are most easily
described using adapted coordinate charts (see pp.\, 224-5 of \cite{BL}) or equivalently, on
the resolution $\wt{X}$. In the second part of the section we introduce the Witt condition and
recall the fundamental theorem of Cheeger, asserting the isomorphism between intersection cohomology
and Hodge cohomology on these spaces.
In the following, we freely use notation from the last section.

\subsection{Existence of iterated edge metrics.} $\,$\newline
We begin by constructing an open covering of $\mbox{reg}\,(\wh{X})$ by sets with an iterated conic structure. 
Let $Y_1$ be any stratum. By definition, for each $q_1 \in Y_1$ there exists a neighbourhood $\calU_1$ and a 
trivialization $\pi_{Y_1}^{-1}(\calU_1) \cong \calU_1 \times C(L_{Y_1})$. Now fix any stratum $Y_2 \subset L_{Y_1}$,
and a point $q_2 \in Y_2$. As before, there is a neighbourhood $\calU_2 \subset Y_2$ and a trivialization 
$\pi_{Y_2}^{-1}(\calU_2) \cong \calU_2 \times C(L_{Y_2})$. Continuing on in this way, the process must stop in
no more than $d = \mbox{depth}\,(Y_1)$ steps when $q_s$ lies in a stratum $Y_s$ of depth $0$ in $L_{Y_{s-1}}$ (
which must, in particular, occur when $L_{Y_{s-1}}$ itself has depth $0$). We obtain in this way an open set of the form 
\begin{equation}\label{eq:chart}
	\calU_1 \times C\big( \calU_{2} \times C( \calU_{3} \times \ldots \times C(\calU_s)\,) \cdots \big),
\end{equation}
where $s \leq d$, which we denote  by $\calW = \calW_{q_1, \ldots, q_s}$. Choose a local coordinate system 
$y^{(j)}$ on  $\calU_j$, and let $r_j$ be the radial coordinate in the cone $C(L_{Y_j})$. Thus $(y^{(1)},r_1, y^{(2)},r_{2}, \ldots, y^{(s)})$ 
is a full set of coordinates in $\calW$. Clearly we may cover all of $\wh{X}$ by a finite number of sets of this form.
We next describe the class of admissible Riemannian metrics on $\mbox{reg}\,(\wh{X})$ by giving their
structure on each set of this type. 

\begin{definition}\label{def:metric}
We say that a Riemannian metric $g$ defined on $\mbox{reg}\,(\wh{X})$ is an iterated edge metric if there is a covering 
by the interiors of sets of the form $\calW_{q_1, \ldots, q_s}$ so that in each such set, 
\[
g= h_{1} +  d r_1^2 + r_{1}^2 ( h_{2} + d r_{2}^2 +  r_{2}^2 ( h_{3} +  d r_{3}^2 +  r_{3}^2 (  h_{4} +\ldots + r_{s-1}^2 h_s) ) ),
\]
with $ 0< r_j < \epsilon$ for some $\epsilon >0$ and every $j$, and where $h_j$ is a metric on $\calU_j$.  We also assume that 
for every $j = 1, \ldots, s$,  $h_j$ depends only on $y^{(1)}, r_1, y^{(2)} , r_{2}, \ldots, y^{(j)}, r_{j}$. 

If each $h_j$ is independent of the radial coordinates $r_{1}, \ldots, r_j$,  then we call $g$ a rigid iterated edge metric.
Note that this requires the choice of a horizontal lift of the tangent space of each stratum $Y$ as a subbundle of
the cone bundle $T_Y$ which is invariant under the scaling action of the radial variable on each conic fibre.
\end{definition}

\begin{proposition} \label{thm:metric} 
Let $\wh{X}$ be a smoothly stratified pseudomanifold. Then there exists a rigid iterated edge metric $g$ on 
$\mbox{reg}\,(\wh{X})$. 
\end{proposition}
\begin{proof} 
We prove this by induction. For spaces of depth $0$, there is nothing to prove, so suppose that $\wh{X}$ is a 
smoothly stratified space of depth $k \geq 1$, and assume that the result is true for all spaces with depth less than $k$.

Let $Y$ be the union of  strata  of depth $k$, each component of which is necessarily a closed manifold; for convenience 
we assume that $Y$ is connected. Consider the space $\wt{X}_1'$ obtained in the first step of the resolution
process in \S 2.3 by adjoining two copies of $\wh{X}$ along $Y$ and replacing the double of the neighbourhood
$T_Y$ by a cylinder $S_Y \times (-2,2)$. This is a space of depth $k-1$, and hence possesses a rigid
iterated edge metric $g_1$. We may in fact assume that in the cylindrical region $S_Y \times (-2,2)$,
$g_1$ has the form $dt^2 + g_{S_Y}$, where $g_{S_Y}$ is a (rigid) iterated edge metric on $S_Y$ which is
independent of $t$. Recalling that $S_Y$ is the total space of a fibration with fibre $L_Y$, we can define
a family of metrics $g_{S_Y}^r$ on $S_Y$ by scaling the metric on each fibre by the factor $r^2$. This leads
to a rigid iterated edge metric $g_{T_Y} := dr^2 + g_{S_Y}^r$ on the tubular neighborhood $T_Y\subset \wh{X}$ around $Y$,
which by construction is also rigid. Now use the induction hypothesis to choose a rigid iterated edge metric $g_C$ on
the complement $C$ of the region $\{r < 1/2\} \subset T_Y$.  Finally, choose a smooth partition of unity $\{\phi(r), \psi(r)\}$
relative to the open cover $[0,2/3) \cup (1/3,\infty)$ of $\RR^+$; the metric $\phi g_{T_Y} + \psi g_{C}$ on $T_Y$ extends
to $g_{C}$ outside $T_Y$, and satisfies our requirement.
\end{proof}
\begin{proposition}\label{prop:homotopymet} 1) Any two  iterated edge metrics on $\wh{X}$ are homotopic 
within the class of  iterated edge metrics. 2) Any two rigid iterated edge metrics on $\wh{X}$ are homotopic 
within the class of rigid iterated edge metrics. 
\end{proposition}
\begin{proof}
We proceed by induction.  The result is obvious when the depth is $0$, so assume
it holds for all spaces of depth strictly less than $k$ and consider a pseudomanifold  of depth $k$ with
two  iterated edge metrics $g$ and $g^\prime$.  

To begin, then, fix a stratum $Y$ which has maximal depth $k$. Then $Y$ is a smooth closed manifold. Recall 
the notation from \S 2.1, and in particular the stratified isomorphism $F_Y$ from the mapping cylinder of 
$(S_Y,\pi_Y)$ to $T_Y$ and the family of local trivializations $\phi:\pi_Y^{-1}(\calU) \to \calU \times C(L_Y)$ 
for suitable $\calU \subset Y$. If $u \in T_Y \cap \pi_Y^{-1}(\calU)$, we write $\phi(u) = (y,z,t)$ where 
$y \in \calU$, $z \in L_Y$ and $t = \rho_Y(u)$; in particular, by axiom v), there is a retraction $R_Y: 
T_Y \setminus Y \to S_Y$, given on any local trivialization by $(y,z,t) \to (y,z,1)$ 
(which is well defined since $t \neq 0$). 

In any of these trivializations, the metric $g$ has the form
\begin{equation*}
	(\phi^{-1})^*g
	= g_{\calU}(y,t) + dt^2 + t^2 g_{L_Y}(t,y,z)
\end{equation*}
and the homotopy 
\begin{equation*}
	s \mapsto 
	g_{\calU}(y,s + (1-s)t) + dt^2 + t^2 g_{L_Y}(s+(1-s)t,y,z)
\end{equation*}
removes the dependence of $g_{\calU}$ and $g_{L_Y}$ on $t$ while remaining in the class of iterated edge metrics.
Since the coordinate $t=\rho_Y(u)$ is part of the control data, this homotopy can be performed consistently across all of the local trivializations $\phi$.

Without loss of generality we may assume that
\begin{equation*}
	(\phi^{-1})^*g
	= g_{\calU}(y) + dt^2 + t^2 g_{L_Y}(y,z), \Mand
	(\phi^{-1})^*g'
	= g_{\calU}'(y) + dt^2 + t^2 g_{L_Y}'(y,z).
\end{equation*}
The metrics $g_{\calU}$ and $g_{\calU}'$ are homotopic and, by inductive hypothesis, so are the metrics $g_{L_Y}$ and $g_{L_Y}'$.
Thus the metrics $(\phi^{-1})^*g$ and $(\phi^{-1})^*g'$ are homotopic within the class of iterated edge metrics on $\calU \times C(L_Y)$.
Using consistency of the trivializations $\phi$ we can patch these homotopies together and see that $g$ and $g'$ are homotopic in a neighborhood of $Y$.

We can thus assume that $g$ and $g'$ coincide in a neighborhood of $Y$ and, in this neighborhood, are independent of $\rho_Y$.
As in the proof of Proposition \ref{prop:BlowUp} we consider the space 
\begin{equation*}
\wt{X}_1' := \left((\wh{X} \setminus Y) \times \{-1\}\right) \sqcup \left((\wh{X} \setminus Y) \times \{+1\}\right)
\sqcup \big(S_Y \times (-2,2) \big) / \sim
\end{equation*}
Define the lift $\wt g$ of $g$ to $\wt X_1'$ by $g$ on each copy of $\wh X \setminus Y$ and 
\begin{equation*}
	g_{\calU}(y)  +  g_{L_Y}(y,z) + dt^2
\end{equation*}
on each neighborhood of $S_Y \times (-2,2)$ corresponding to the trivialization $\phi$ as above, and define $\wt g'$ similarly.
Then $\wt g$ and $\wt g'$ are iterated edge metric on a space of depth $k-1$ so by inductive hypothesis are homotopic.
Moreover since they coincide in $S_Y \times (-2,2)$, the homotopy can be taken to be constant in a neighborhood of $S_Y$, and hence the homotopy descends to a homotopy of $g$ and $g'$.

If $g$ and $g'$ are rigid, the homotopies above preserve this.
\end{proof}

Cheeger also defines \cite{Cheeger-symp} (p.\ 127) a class of admissible metrics $g$ on the regular part of 
a smoothly stratified pseudomanifold $\wh{X}$. He uses a slightly different decomposition of $\wh{X}$ and 
assumes that on each `handle' of the form  $(0,1)^{n-i} \times C(N^{i-1})$, $g$ induces a metric  quasi-isometric
to one of the form
\[
(d y_1)^2 + \ldots + (d y_{n-i})^2 + (d r)^2 + r^2 g_{N^{i-1}};
\]
see \cite{Cheeger-symp} for the details. Using the proof of Proposition \ref{thm:metric} as well as 
\cite{Cheeger-symp} (page 127), we obtain the following 

\begin{proposition} \label{prop:adm}
{\item 1)} Any rigid iterated edge  metric as in Definition \ref{def:metric}) is admissible in the sense of Cheeger.
{\item 2)} Any two admissible metrics are quasi-isometric.
\end{proposition}

Recall the manifold with corners with iterated fibration structure $\wt{X}$, which is the resolution of $\wh{X}$.
Its interior is canonically identified with $\mbox{reg}\,(\wh{X})$, and we identify these without comment.
Let $x_\alpha$ be a global defining function for the boundary hypersurface $H_\alpha$ of $\wt{X}$ (so $H_\alpha 
= \{x_\alpha = 0\}$);  the total boundary defining function of $\wt{X}$ is, by definition,
\[
\rho = \prod_{\alpha \in A} x_\alpha.
\]
If $g$ is an iterated edge metric on $\mbox{reg}\,(\wh{X})$, then set
\begin{equation}
\wt{g} = \rho^{-2} g.
\label{def:completemet}
\end{equation}
It is not hard to check that this metric is complete. 

\subsection{The Witt condition. Cheeger's Hodge theorem on Witt spaces.} $\;$\newline
In this paper, we consider only orientable Witt spaces, which are defined as follows.
\begin{definition} A pseudomanifold $\wh{X}$ is a Witt space if, for some (and hence any) stratification, all links of even dimension have vanishing lower middle perversity intersection homology in middle degree, i.e.,
\begin{equation*}
	Y \in {\mathfrak S}, \quad
	\dim L_Y = f_Y \text{ even}
	\implies I\! H_{\underline{m}}^{f_Y/2} ( L_Y)= 0.
\end{equation*}
\end{definition}

It is a theorem that on a Witt space $\widehat{X}$ the lower and upper middle perversity 
intersection homology
groups are equal up to isomorphism: $I\! H_{\underline{m}}^{*} ( \widehat{X}) = I\! H_{\overline{m}}^{*} ( \widehat{X})$.

A famous result concerning the $L^2$ cohomology of Witt spaces is due to Cheeger. 
\begin{theorem}\label{theo:cheeger}(Cheeger) Let $\widehat{X}$ be a Witt space endowed with
an iterated edge metric $g$. 
Denote  by $H_{(2)}^{*} (\widehat{X}) $ 
the cohomology of the $L^2$ de Rham complex with maximal domain; 
denote by $\calH_{(2)}^{*} ( \widehat{X})$ the $L^2$ maximal  Hodge cohomology.
Then 
\begin{equation}\label{eq:cheeger-0}
H_{(2)}^{*} (\widehat{X} )= \calH_{(2)}^{*} ( \widehat{X})=I\! H_m^* (\widehat{X},\bbC).
\end{equation}
with $m$ denoting either the upper or lower middle perversity.\\
In particular, if $Y$ is a stratum with link $L_Y$, and $f_Y = \dim L_Y$ is even, then 
\begin{equation}\label{eq:cheeger}
H_{(2)}^{f_Y/2} ( L_Y)=\calH_{(2)}^{f_Y/2} ( L_Y) = 0.
\end{equation}
\end{theorem}

\section{Iterated edge vector fields and operators} \label{Operators}

On a closed manifold, $L^2$ and Sobolev spaces are defined using a Riemannian metric but the spaces
themselves are metric-independent. A differential operator induces a bounded map between suitable ones of
these spaces, and ellipticity guarantees that this map is Fredholm. All of this fails when the manifold is not 
closed, and in this section we describe some of what is true for iterated edge metrics.

The space $X := \mbox{reg}\,(\wh{X})$ with complete metric $\wt{g}$ is an example of what is called a Riemannian manifold 
with bounded geometry. There are natural classes of $L^2$ and Sobolev spaces on any such space, as well as
a class of `uniform' differential operators, which induce bounded maps between these function spaces. There is also
a calculus of uniform pseudo-differential operators which contains parametrices of uniform elliptic operators, and which 
can be used to prove certain uniform elliptic regularity results. Using that $X$ compactifies to $\wh{X}$, we can also
define weighted $L^2$ and Sobolev spaces in this setting, and the uniform calculus gives some results for operators
mapping between these as well. This uniform calculus does not establish that these mappings are Fredholm, and indeed,
that requires more delicate arguments. 

In this section we describe these ideas and explain how they can be applied to the de Rham operator of the edge iterated  metric $g$.
The uniform pseudodifferential calculus can also be used to obtain a parametrix even after twisting by a bundle of  
projective finitely generated modules over a $C^*$-algebra. 

\subsection{Edge vector fields on $X$}\label{complete-metric} $ $\newline
Associated to the complete metric $\wt g$ on $X$ is the space of `iterated edge' vector fields
\begin{equation}\label{FirstDescVie}
	\cV_{\ie}
	= \{ V \in C^\infty (\wt X, T\wt X): X \ni q \mapsto \wt g_q(V,V) \in \bbR^+ \text{ is bounded} \}.
\end{equation}
In the notation of \S 3, on a neighbourhood of the form $\calW_{q_1, \ldots, q_s}$, this is locally spanned 
over $\calC^\infty(\wt X)$ by vector fields of the form 
\[
r_1 \ldots r_{s-1} \del_{r_1}, r_1 \ldots r_{s-1} \del_{y^{(1)}}, r_1 \ldots r_{s-2} \del_{r_2}, r_1 \ldots r_{s-2} \del_{y^{(2)}}, 
\ldots, \del_{y^{(s)}}.
\]
It is easy to see that $\cV_{\ie}$ forms a locally finitely generated, locally free Lie algebra with respect to the usual bracket
on vector fields; furthermore, Swan's theorem shows that there is a vector bundle ${}^{\ie}TX$ over $\wt X$ whose space of sections is $\cV_{\ie}$,
\begin{equation}\label{DefIeTX}
	\CI (\wt X, ^{\ie}T X ) = \cV_{\ie}.
\end{equation}
This bundle ${}^{\ie}TX$ coincides with the usual tangent bundle $TX$ over the interior of $\wt X$ and is isomorphic to $T\wt X$, 
though there is no canonical isomorphism. It is easy to see that $\wt g$ defines a metric on ${}^{\ie}TX$. 

\begin{proposition} $(X, \widetilde{g})$ is a complete Riemannian manifold of bounded geometry.
\end{proposition} 
\begin{proof}
Recall the theorem of Gordon-de Rham-Borel, which states that a manifold is complete if and only if  it admits a 
nonnegative, smooth, proper function with bounded gradient. For this metric $\wt g$, such a function is $-\log (\rho)$,
where $\rho$ is the total boundary defining function. To prove that $g$ has bounded geometry one must check that 
the curvature tensor of $\wt g$, and its covariant derivatives, are bounded and that the injectivity radius of $\wt g$ 
has a positive lower bound. The former follows from the compactness of $\wt X$, and the latter can be shown as in \cite{ALN}.
 \end{proof}

The set of $\ie$-differential operators is the enveloping algebra of $\cV_{\ie}$; i.e., it consists of linear combinations (over $\CI(\widetilde{X})$) 
of finite products of elements of $\cV_{\ie}$. We denote by $\Diff_{\ie}^k(X)$ the subset of differential operators that have local descriptions involving products of at most $k$ elements of $\cV_{\ie}$.
If $E$ and $F$ are vector bundles over $\wt X$, then the space of $\ie$-differential operators acting between sections of $E$ and sections of $F$ is defined similarly, by taking linear combinations over $\CI(\wt X, \Hom(E,F) )$.

We define Sobolev spaces for $\ie$ metrics by
\begin{gather*}
	H^0_{\ie}(X) = L^2_{\ie}(X) = L^2(X, \dvol(\wt g)) \\
	H^k_{\ie}(X)
	= \{ u \in L^2_{\ie}(X) : Au \in L^2_{\ie}(X), \Mforevery A \in \Diff_{\ie}^k(X) \}, \; k \in \bbN
\end{gather*}
then define $H^t_{\ie}(X)$ using Calder\'on interpolation for $t \in \bbR^+$ and duality for $t \in \bbR^-$.
Sobolev spaces for sections of bundles over $ \widetilde{X}$ are defined similarly.

We will also allow for operators to act between sections of certain bundles of projective finitely generated modules over a $C^*$-algebra; see \cite{ST} for the basic definitions.
We assume that we have a continuous map $r_0: X \to B\Gamma$ which extends continuously to
\begin{equation*}
 r: \widehat{X}\to B\Gamma
	\end{equation*}
 where $\Gamma$ is a countable, finitely generated, finitely presented group. This determines a $\Gamma$-covering, $ \widehat{X}^\prime \to 
\widehat{X}$;  and we will denote by $\wt{C^*_r}\Gamma$ the corresponding bundle, over $\widehat{X}$, of free left 
$C^*_r\Gamma$-modules of rank one: 
\begin{equation}\label{cEDef}
	\wt{C^*_r}\Gamma: = C^*_r\Gamma \times_\Gamma \widehat{X}^\prime.
\end{equation} 
Observe that this bundle induces, after pull back by the blowdown map $\wt X \rightarrow \widehat{X}$,
a bundle on $\wt X$ (for which we keep the same notation).
Given vector bundles $E$ and $F$ over $\wt X$ of rank $k$ and $\ell$, we define  bundles $\cE$ and $\cF$ over $\wt{X}$ 
by tensoring $E$ and $F$ by $\wt{C^*_r}\Gamma$; we obtain in this way 
bundles of projective finitely generated $C^*_r \Gamma$-modules   of rank $k$ and $\ell$ . 
We shall briefly refer to  $\cE$ and $\cF$
as $C^*_r\Gamma$-bundles.
An iterated edge differential operator acting between sections of $\cE$ and $\cF$ is defined as above, but allowing the coefficients to be 
$C^*_r\Gamma$-linear. The space of such operators will be denoted
\begin{equation*}
	\Diff_{\ie,\Gamma}^*(X; \cE, \cF).
\end{equation*}
Finally, we denote by  $H^t_{\ie,\Gamma}(X;\cE) $  the corresponding Sobolev $C^*_r \Gamma$-module, see \cite{Mich-Fomenko}.

\subsection{Uniform pseudodifferential operators} \label{sec:uniform} $ $\newline
We showed above that $\ie$ metrics have bounded geometry. This allows us to use the calculus of 
uniform pseudo-differential operators as described in the work of Meladze-Shubin (see
\cite{Meladze-Shubin} and \cite{Kordyukov}).

We single out the space $\cB C^\infty (X)$ of functions which are uniformly bounded with uniformly bounded derivatives of
all orders. Smooth functions on $\wt X$ are in $\cB \CI (X)$, but the latter space is larger since general elements are not
smooth at the boundary faces of $\wt X$. A vector bundle over $X$  
is said to be a {\em bundle of bounded geometry} if it has trivializations whose transition functions are (matrices with entries) in $\cB\CI(X)$. 
Vector bundles that extend smoothly to $\wt X$ have bounded geometry. 

The spaces of operators $\Diff^*_{\cB} (X;E,F)$ and, more generally, $\Diff^*_{\cB,\Gamma} (X;\cE,\cF)$, are defined by
requiring the coefficients to be in $\cB C^\infty$. These spaces contain $\Diff_{\ie}^*(X;E,F)$ and $\Diff_{\ie,\Gamma}^*(X;\cE, \cF)$,
respectively. 

Next, the bounded geometry of $(X,\wt g)$ implies that it is possible to find a countable cover of $X$ by open sets,
each of which are normal coordinate charts for the \emph{complete} metric $\wt{g}$ and which all have fixed
radius $\eps > 0$.  Calling these charts $\cU_{\eps}(\zeta_i)$, then it is also possible to arrange that $\cU_{2\eps}(\zeta_i)$ 
has uniformly bounded, finite multiplicity as a cover of $X$. We can then choose partitions of unity $\wt \phi_i$, $\phi_i$ 
subordinate to $\{ \cU_{2\eps}(\zeta_i) \}$ and $\{ \cU_{\eps}(\zeta_i) \}$ respectively such that 
$\wt \phi_i$, $\phi_i$ have bounded derivatives uniformly in $i$, and such that $\wt{\phi}_i = 1$ on $\mathrm{supp}\, \phi_i$. 
These functions can be used to transplant constructions from $\bbR^n$ to $X$. 

We next recall how to transfer pseudodifferential operators from $\bbR^n$. Let $E$ and $F$ be vector bundles over $\wt X$,
and denote by $d = d_{\wt{g}}$ the distance function associated to the complete metric $\wt g.$ 
An operator $A:\CIc(X;E) \to \CIc(X;F)$ is called a {\bf uniform pseudodifferential operator} of order $s \in \bbR$, 
\begin{equation*}
	A \in \Psi^s_{\cB}(X;E,F), 
\end{equation*}
if its Schwartz kernel $\cK_A \in \CmI(X^2; \Hom(E,F))$ satisfies the following properties. 
\begin{itemize}
\item[i)] For some $C_A > 0$, 
\begin{equation*}
	\cK_A(\zeta, \zeta')=0 \Mif d(\zeta, \zeta')>C_A,\\
\end{equation*}
\item[ii)] For every $\delta>0$, and any multi-indices $\alpha$, $\beta$ there is a constant $C_{\alpha\beta\delta}>0$ such that
\begin{equation*}
	|D^\alpha_\zeta D^\beta_{\zeta'} \cK_A(\zeta, \zeta')| \leq C_{\alpha\beta\delta}, 
	\text{ whenever } d(\zeta, \zeta') >\delta. 
\end{equation*} 
\item[iii)] For any $i$, and using the normal coordinate chart to identify $\cU_{2\eps}(\zeta_i)$ with $B_{2\eps}(0)$ in $\RR^n$, 
$\wt \phi_i A \phi_i$ is a pseudodifferential operator of order $s$ in $B_{2\eps}(0)$, whose full symbol $\sigma$ satisfies the 
usual symbol estimates uniformly in $i$,
\begin{equation*}
	\abs{ D_\zeta^\alpha D_\xi^\gamma \sigma( \wt \phi_i A \phi_i )(\zeta, \xi) }
	\leq C_{\alpha\beta\gamma} (1+|\xi|_{{\wt g}}^2)^{\frac12(s-|\gamma|)};
\end{equation*}
here $ |\xi|_{{\wt g}}$ is the norm of $\xi \in T^*_\zeta X$ with respect to $\wt g.$
\end{itemize}

We always assume that the symbols are (one-step) polyhomogeneous. Uniform pseudo-differential operators form an algebra. 
There is a well defined principal  symbol map, with values in $\cB C^\infty (S^* X, \hom (\pi^* E,\pi^* F))$. Ellipticity is defined
in a natural way (one requires the principal symbol to be uniformly invertible, i.e. invertible with inverse in $\cB C^\infty $).  
The principal symbol $\sigma(P)$ of a uniform pseudodifferential operator $P$ is a section of $\Ie T^*X$ (the bundle 
dual to $\Ie T X$) restricted to $X$. In general, $\sigma(P)$ does not extend to be a smooth section of $\Ie T^*X \rightarrow \wt X.$

For a bundle of bounded geometry $E$ and $s \in \RR$, define the $\cB$-Sobolev space
\begin{multline}\label{SobSpaces}
	H^s_{\cB}(X;E) \\
	= \{ u \in \CmI(X;E) : \phi_i u \in H^s(\bbR^n;E) \text{ with norm bounded uniformly in } i \}.
\end{multline}
The same definition holds for $C^*_r\Gamma$-bundles and we denote by $H^s_{\cB,\Gamma} (X;\cE)$, the
corresponding $C^*_r \Gamma$-module. Uniform pseudodifferential operators extend to bounded operators 
between $\cB$-Sobolev spaces. 

If a map $r:\widehat{X}\to B\Gamma$ is given, then we can define uniform pseudo-differential operators between 
sections of $\cE$ and sections of $\cF$ by combining the above definition and the classic construction 
of Mishchenko and Fomenko; we denote  by  $\Psi^*_{\cB, \Gamma}(X;\cE, \cF)$ the corresponding algebra. 
Notice that the principal symbol is in this case a $C^*_r\Gamma$-linear map between the lifts of $\cE$ and $\cF$ to the cotangent bundle.

The intersection over $s\in \bbR$ of the $\Psi_{\cB,\Gamma}^s(X;\cE,\cF)$ is denoted $\Psi_{\cB,\Gamma}^{-\infty}(X;\cE,\cF)$ and 
consists of smoothing operators whose integral kernel in $X\times X$ is in $\cB C^\infty $.

Elements of the uniform calculus also define bounded maps between {\it weighted} $C^*_r\Gamma$-Sobolev spaces. Let $\rho$
be the total boundary defining function for $\wt{X}$. 
\begin{lemma}
If $A \in \Psi_{\cB,\Gamma}^s(X; \cE, \cF)$, then for any $a, t \in \bbR$, $A$ induces a bounded operator
\begin{equation*}
	A: \rho^a H^t_{\ie,\Gamma}(X;\cE) \to \rho^a H^{t-s}_{\ie,\Gamma}(X;\cF).
\end{equation*}
\end{lemma}

\begin{proof}
It is enough to check that $\rho^{-a} A \rho^{a} \in A \in \Psi_{\cB,\Gamma}^s(X; \cE, \cF)$ for any $a$. 
The integral kernel of $\rho^{-a} A \rho^{a}$ is 
\begin{equation*}
	\lrpar{\frac{\rho(\zeta)}{\rho(\zeta')}}^a \cK_A(\zeta, \zeta')
\end{equation*}
and the lemma follows by noting that if $\lrpar{\frac{\rho(\zeta)}{\rho(\zeta')}}^a$ is a bounded smooth function on the support of $\cK_A$.
\end{proof}

An important property of the uniform pseudodifferential calculus is that it has a symbolic calculus.
By standard constructions, this implies that any elliptic element in $\Diff_{\cB,\Gamma}^k(X; \cE, \cF)$ has a symbolic parametrix, 
i.e.\ an inverse modulo smoothing operators. In particular, using the above Proposition, we see that an elliptic $\ie$ operator 
$A \in \Diff_{\ie,\Gamma}^k(X; \cE, \cF)$ has a {\em symbolic parametrix}
\begin{equation*}
	Q \in \Psi_{\cB,\Gamma}^{-k}(X; \cF, \cE) \Mst \\
	\Id_{\cE} - QP \in \Psi_{\cB,\Gamma}^{-\infty}(X;\cE), \quad
	\Id_{\cF} - PQ \in \Psi_{\cB,\Gamma}^{-\infty}(X;\cF). 
\end{equation*}

The symbolic calculus also yields the standard characterization of Sobolev spaces. 
For instance, if $N \in \bbN$, then 
\begin{multline*}
	H^N_{\cB}(X)
	= \{ u \in \CmI(X) : Au \in L^2(X) \text{ for all } A \in \Diff_{\cB}^N(X) \} \\
	= \{ u \in \CmI(X) : Au \in L^2(X) \text{ for some uniformly elliptic } A \in \Diff_{\cB}^N(X) \};
\end{multline*}
in fact, if $A \in \Diff_{\cB}^N(X)$ is uniformly elliptic, then $H^N_{\cB}(X)$ equals the maximal domain of $A$ as an unbounded 
operator on $L^2(X)$. This characterization, applied to an elliptic operator $A \in \Diff_{\ie}^N(X)$, shows that 
$H^N_{\ie}(X) = H^N_{\cB}(X)$. Using Calderon interpolation and duality, we see that $H^t_{\ie}(X) = H^t_{\cB}(X)$ for all $t \in \bbR$, 
and the same is true for sections of bundles of bounded geometry and the corresponding $C^*_r\Gamma$-bundles.

\subsection{Incomplete iterated edge operators} $ $\newline
The set of incomplete iterated edge differential operators, $\Diff_{\iie,\Gamma}^*(X;\cE,\cF)$ is defined in terms of $\Diff_{\ie,\Gamma}^*(X;\cE,\cF)$ by
\begin{equation*}
	\Diff_{\iie,\Gamma}^k(X;\cE,\cF) = \rho^{-k} \Diff_{\ie,\Gamma}^k(X;\cE,\cF),
\end{equation*}
where $\rho = x_0\cdots x_{m-1}$. 
As an operator between weighted $L^2$ spaces  with appropriate {\em different} weights, an operator $A \in \Diff_{\iie,\Gamma}^k(X;\cE,\cF)$ is unitarily equivalent to an iterated edge operator.
Thus, for instance, for any $a \in \bbR$, $A$ defines an unbounded operator
\begin{equation*}
	A: \rho^a L^2_{\ie,\Gamma}(X;\cE) \to \rho^{a-k} L^2_{\ie,\Gamma}(X;\cF)
\end{equation*}
which has a unique closed extension whose domain is $\rho^{\alpha}H^k_{\ie,\Gamma}(X;\cE)$; moreover, $A$ defines bounded operators
\begin{equation*}
	\rho^{a} H^t_{\ie,\Gamma}(X;\cE) \to \rho^{a-k} H^{t-k}_{\ie,\Gamma}(X;\cF)
\end{equation*}
for every $a$ and $t \in \bbR$.
However, it is the more complicated behavior of $A$ as an unbounded operator
\begin{equation}\label{Aop}
	A: \rho^a L^2_{\ie,\Gamma}(X;\cE) \to \rho^{a} L^2_{\ie,\Gamma}(X;\cF)
\end{equation}
that we will be concerned with.
We point out that the operator \eqref{Aop} is unitarily equivalent to the unbounded operator
\begin{equation*}
	\wt A = \rho^{k/2} A \rho^{k/2}: 
	\rho^{a-k/2} L^2_{\ie,\Gamma}(X;\cE) \to \rho^{a+k/2} L^2_{\ie,\Gamma}(X;\cF),
\end{equation*}
Since $\wt A\in \Diff_{\ie,\Gamma}^*(X;\cE,\cF)$, this  shows that
the study of incomplete iterated edge operators acting on a fixed Hilbert space is the same as the study of complete $\ie$-operators acting between {\em different} Hilbert spaces.

We point out that the $L^2$ spaces of the incomplete iterated edge metric $g$ and the associated complete $\ie$ metric $\wt g = \rho^{-2 }g$ are related by
\begin{equation*}
	L^2_{\ie,\Gamma}(X, \cE) = \rho^{n/2} L^2_{\iie,\Gamma}(X, \cE)
\end{equation*}
with $n$ equal to the dimension of $X$, 
so switching between them only involves a shift of the weight. 
Similarly, we introduce the spaces $H^t_{\iie, \Gamma}(X; \cE)$ for $t \in \bbR$ by
\begin{equation*}
	H^t_{\ie, \Gamma}(X; \cE) = \rho^{n/2} H^t_{\iie, \Gamma}(X; \cE).
\end{equation*}
Thus, for instance, if $N \in \bbN$ then $H^N_{\iie,\Gamma}(X, \cE)$ is the set of elements $ u \in L^2_{\iie,\Gamma}(X, \cE)$ such 
 that for any vector fields $V_1, \ldots, V_p \in \mathcal{V}_{\ie}$ where $p \leq N,$ we have 
 $V_1 \ldots V_p u \in L^2_{\iie,\Gamma}(X, \cE).$

We say that $A\in \Diff_{\iie,\Gamma}^k(X; \cE, \cF)$ is elliptic if $\wt A = \rho^k A$ is an elliptic $\ie$ operator.
Elliptic $\ie$ operators always have a symbolic parametrix (see \S\ref{sec:uniform}).
A symbolic parametrix $\wt Q$ for $\wt A$ yields a symbolic parametrix $Q = \rho^{k/2}\wt Q \rho^{k/2}$ for $A$. 
Recall that a continuous adjointable $C^*_r\Gamma$-linear operator $K$ is called $C^*_r\Gamma$-compact 
if both $K$ and $K^*$ are uniform limits of sequences of $C^*_r\Gamma$-linear operators whose ranges are 
finitely generated $C^*_r \Gamma-$modules.
As is well-known, since smoothing operators are not necessarily $C^*_r\Gamma$-compact, a symbolic parametrix 
is generally not enough to determine when an operator is $C^*_r\Gamma$-Fredholm, so one also needs to know about the behavior at 
the boundary.
  
However, the uniform calculus does establish elliptic regularity in the sense that, whenever $B \in \Diff_{\ie,\Gamma}^k(X; \cE, \cF)$ 
is elliptic and $a \in \bbR,$ we have 
 \begin{equation}\label{EllReg}
	u \in \rho^a L^2_{\iie,\Gamma}(X, \cE), \quad
	Bu \in \rho^a L^2_{\iie,\Gamma}(X, \cF)
	\implies u \in \rho^a H^N_{\iie,\Gamma}(X, \cE).
\end{equation}

\subsection{The de Rham operator} $ $\newline
We are interested in analyzing the de Rham operator of an $\iie$ metric,
\begin{equation*}
	\eth_{\dR} = d + \delta: \Omega^*X \to \Omega^*X.
\end{equation*}
As with the tangent bundle, it is convenient to replace the bundle of forms $\Omega^*(X)= \CI (X,\Lambda^* (T^*X))$ with the bundle of $\iie$-forms, 
\begin{equation*}
	\Iie \Omega^*(X) = \CI (X, \Lambda^* (\Iie T^*X) ),
\end{equation*}
where $\Iie T^*X \rightarrow \wt X $ is the rescaled bundle (cf. \cite[Chapter 8]{APS Book}) defined by  
\begin{equation*}
	\CI (\wt X,\Iie T^*X) = \rho \CI (\wt X,\Ie T^*X).
\end{equation*}
We set $\Iie \Lambda^*X = \Lambda^* (\Iie T^*X),$ and  we have
\begin{equation*}
	\eth_{\dR} \in \Diff_{\iie}^1(X; \Iie\Lambda^*(X), \Iie\Lambda^*(X))
\end{equation*}
as we now explain.

First note that whether or not $\eth_{\dR}$ is an element of $\Diff_{\iie}^1(X; \Iie\Lambda^*(X), \Iie\Lambda^*(X))$ can be checked 
locally in coordinate charts. There is nothing to check in the interior of the manifold. Then, with the notations of 
\S 3, we consider a distinguished neighborhood  $W$ of a point of a stratum $Y.$ Thus 
$W$ is diffeomorphic to $B \times C(Z)$ where $B$ is an open subset of $Y$ which is diffeomorphic to a vector space and 
$C(Z)$ is the cone whose base $Z$ is a stratified space. 
The `radial' coordinate of the cone will be denoted by $x$. 

As in \S\ref{section:iterated}, the fibration over $B$ extends to $W$, 
\begin{equation*}
	Z \times [0,1)_x - W \xrightarrow{\wt\phi} B,
\end{equation*}
and using $x$ and a choice of connection for this fibration we can write
\begin{equation*}
	T^*X\rest{W} = \ang{dx} \oplus T^*Y \oplus T^*Z.
\end{equation*}
With respect to this splitting the metric $g$ restricted to $W$ has the form
\begin{equation*}
	g = dx^2 + \wt\phi^*g_Y + x^2 g_Z
\end{equation*}
and the differential forms on $X$ can be decomposed as
\begin{equation}\label{SplitForms}
\begin{gathered}
	\Lambda^*X = 
	(\Lambda^*Y \wedge \Lambda^* Z)
	\oplus dx \wedge
	(\Lambda^*Y \wedge \Lambda^* Z) \\
	\Iie \Lambda^*X = 
	(\Lambda^*Y \wedge x^{\bN} \Lambda^* Z)
	\oplus dx \wedge
	(\Lambda^*Y \wedge x^{\bN} \Lambda^* Z)
\end{gathered}
\end{equation}
where $\bN$ is the `vertical number operator', i.e., the map given by multiplication by $k$ when restricted to forms of vertical degree $k$.
This allows us to split the exterior derivative into
\begin{equation*}
	d = 
	\df e_{dx} \pa_x
	\oplus d^Y
	\oplus d^Z
\end{equation*}
where $\df e_{dx}$ denotes the exterior product by $ d x$ and correspondingly
\begin{equation*}
	\delta =
	\star^{-1} \df e_{dx} \pa_x \star
	\oplus 
	\star^{-1} d_Y \star
	\oplus 
	\star^{-1} d_Z \star
	= 
	\star^{-1} \df e_{dx} \pa_x \star
	\oplus \delta^Y_x
	\oplus \delta^Z_x
\end{equation*}
where the $x$-dependence in $\delta^Y_x$ and $\delta^Z_x$ comes from the $x$-dependence of the Hodge star operator, $\star$.
A straightforward computation shows that with respect to the splitting \eqref{SplitForms} of $\Iie\Lambda^*X$, (and with $f = \dim Z$), 
\begin{equation}\label{PreDNearBdy}
	\eth_{\dR}
	= \begin{pmatrix}
	\frac1x( d^Z + \delta^Z_x) + d^Y + \delta^Y_x & - \star^{-1}\pa_x \star - \tfrac1x(f-\bN) \\
	\pa_x + \tfrac1x \bN & -\frac1x( d^Z + \delta^Z_x) - d^Y - \delta^Y_x 
	\end{pmatrix}.
\end{equation}
As in \cite[(19)]{Hunsicker-Mazzeo} one can write this in terms of operators related to the fibration, 
however for our purposes it is more important to point out that the leading order term with respect to $x$ (as an $\iie$ operator) is given by
\begin{equation}\label{DNearBdy}
	\eth_{\dR} \sim
	 \begin{pmatrix}
	\tfrac1x \eth^Z_{\dR} + \eth^Y_{\dR} & - \pa_x - \tfrac1x(f-\bN) \\
	\pa_x + \tfrac1x \bN & -\tfrac1x \eth^Z_{\dR}  - \eth^Y_{\dR} 
	\end{pmatrix}.
\end{equation}
where $f$ denotes the dimension of $Z$, $\eth^Y_{\dR}$ and $\eth^Z_{\dR}$ are the de Rham operators of $\wt\phi^*g_Y\rest{x=0}$ and $g^Z\rest{x=0}$, respectively. 
In effect, because of the weighting of the vertical forms, the Hodge star operator is asymptotically acting like the Hodge star operator of the product metric at $\{x=0\}$.

By induction on the depth of the stratification and using \eqref{DNearBdy} one proves without difficulties the following:

\begin{lemma}  \label{lem:iie} The operator $\eth_{\dR} $ is  in $\Diff_{\iie}^1$, i.e.,
 $\rho \eth_{\dR} $ is  in $\Diff_{\ie}^1$
\end{lemma} 
 
We are also interested in the behaviour of $\eth_{\dR}$ after twisting to get $C^*$-algebra coefficients. Thus we assume, as before, that we have a continuous  map
\begin{equation*}
	r:  \widehat{X} \to B\Gamma
\end{equation*}
We compose $r$ with the blow-down map $\beta$ and we pull-back the universal bundle $E\Gamma$
to $\widetilde{X}$ using $f\circ \beta$. We obtain a Galois $\Gamma$-covering $\widetilde{X}^\prime$
over   $\widetilde{X}$ and 
the  associated bundle 
$\wt C^*_r\Gamma \to \widetilde{X}$, with 
$$\wt C^*_r\Gamma:= C^*_r \Gamma \times_\Gamma \widetilde{X}^\prime\,.$$
 We restrict $\wt C^*_r\Gamma$ to $X$.
Endowing $C^*_r\Gamma\times \widetilde{X}^\prime$, as a trivial bundle over $\widetilde{X}^\prime$, with the trivial connection induces a (non-trivial) flat connection on the bundle $\wt C^*_r\Gamma \rightarrow \widetilde{X}$;
we also obtain a flat connection on the restriction of   $\wt C^*_r\Gamma$ to $X$
(and it is obvious that this connection will automatically extend to $\widetilde{X}$). 
Using the latter connection we can define directly $\wt \eth_{\dR}$, the twisted de Rham operator on the sections of the vector bundle 
\begin{equation*}
	\Iie\Lambda_\Gamma^*(X)  = \Iie\Lambda^*X \otimes \wt C^*_r\Gamma.
\end{equation*}
By construction  $\wt \eth_{\dR}\in \Diff^*_{\iie,\Gamma}$, i.e. $\rho \wt \eth_{\dR}$ is an element in $\Diff^*_{\ie,\Gamma}$.

\section{Inductive analysis of the signature operator}

In this section we analyze the behavior of the de Rham operator near the singular part of $\widehat{X}$. This is
done inductively. The base case is that of a closed manifold, which is classical. Stratifications of depth one are 
analyzed in the work of Hunsicker and the third author \cite{Hunsicker-Mazzeo}, where the relationship between 
intersection cohomology and Hodge cohomology is treated in detail. Our results for depth one stratifications 
is implicitly contained in \cite{Hunsicker-Mazzeo}. However, the treatment in  \cite{Hunsicker-Mazzeo} relies 
heavily on the edge calculus \cite{Mazzeo:edge} which allows refined results, such as finding conormal representatives 
of cohomology classes. Though we cannot use the edge calculus directly, we proceed by adapting certain arguments from 
\cite{Mazzeo:edge} to our context. More precisely, we define a model for this operator at each point of a singular stratum
and then we establish that these model operators are invertible when acting on the appropriate Sobolev spaces. 
Taken together, ellipticity and this asymptotic invertibility are enough to establish the Fredholm properties we seek.

The main advantage of the de Rham operator over an arbitrary $\iie$ operator lies in \eqref{DNearBdy}. Indeed this 
shows that, at a given point $q$ on the boundary, the leading order behavior of $\eth_{\dR}$ involves the fibre $Z$ over 
$q$ only through its de Rham operator $\eth_{\dR}^Z$.  To take advantage of this structure we multiply this operator
by a (symmetrically distributed) power of the radial distance $x$ to the highest depth stratum $Y$. Since this is closely
related to the de Rham operator for the metric $x^{-2}g$, which we regard as a `partial completion' of $g$ (i.e.\ we
have made it complete near $Y$, but the link $Z$ of the associated cone bundle with its induced metric remains incomplete.
This allows us to set up an inductive scheme. 

\subsection{The partial completion of the de Rham operator}$\;$\newline
Recall that  \eqref{DNearBdy} was written in a  distinguished neighborhood  $W$ of a point of a stratum $Y$  and that
$W$ is diffeomorphic to $B \times C(Z)$ where $B$ is an open subset of $Y$ diffeomorphic to a vector space and 
$C(Z)$ is the cone with smoothly stratified link $Z$. The `radial' coordinate of the cone is still 
denoted $x$, which we identify with one the boundary defining functions $x_j$ and thereby extend globally 
to $\wt X.$ To take advantage of the structure of the de Rham operator in $W$, as it 
appears in \eqref{DNearBdy}, we define the `partial conformal completion' of the signature operator
\begin{equation*}
	D_0 = x^{1/2} \eth_{\dR} x^{1/2}.
\end{equation*}

The advantage of using $x^{1/2} \eth_{\dR} x^{1/2}$ over, say, $x\eth_{\dR}$ is that the former
is symmetric as an operator
\begin{equation*}
	x^{-1/2} L^2_{\iie,\Gamma}(X, \Iie\Lambda_\Gamma^*(X) ) 
	\to x^{1/2} L^2_{\iie,\Gamma}(X, \Iie\Lambda_\Gamma^*(X) )
\end{equation*}
(with respect to the natural pairing between the spaces on the right and left here), since $\eth_{\dR}$ is a symmetric 
operator on $L^2_{\iie,\Gamma}(X; \Iie\Lambda_\Gamma^*(X) )$ with core domain  $C^\infty_c$.

To analyze $\eth_{\dR}$ it is useful to consider the operator it induces on various weighted $L^2$ spaces.
For later use we point out first that $\eth_{\dR}$ satisfies  
\begin{equation}\label{DavsD0}
	\eth_{\dR}(x^a v) 
	= [\eth_{\dR}, x^a]v 
	+ x^a \eth_{\dR} v
	= x^a [ a \df e(\tfrac{dx}{x}) - a \df i(\tfrac{dx}{x}) + \eth_{\dR}] v,
\end{equation}
where $\df e$ and $\df i$ denote exterior and interior product respectively, and, second, that we have a unitary equivalence of 
unbounded operators \footnote{Note that in \cite{Hunsicker-Mazzeo}, for a stratification of depth one, $D_a$ denotes the 
de Rham operator of the complex $(x^aL^2_{\iie}, d)$ while here $D_a$ denotes the de Rham operator of the complex 
$(L^2_{\iie}, d)$ as an operator on $x^aL^2_{\iie}$.}
\begin{equation*}
\begin{gathered}
	\eth_{\dR} : x^a L^2_{\iie,\Gamma}(X, \Iie\Lambda_\Gamma^*(X) ) \to x^a L^2_{\iie,\Gamma}(X, \Iie\Lambda_\Gamma^*(X) ) \\
	\leftrightarrow
	D_a= x^{1/2-a} \eth_{\dR} x_0^{1/2+a} : 
	x^{-1/2} L^2_{\iie,\Gamma}(X, \Iie\Lambda_\Gamma^*(X) ) 
	\to x^{1/2} L^2_{\iie,\Gamma}(X, \Iie\Lambda_\Gamma^*(X) ).
\end{gathered}
\end{equation*}
 
In order to adapt arguments from \cite{Mazzeo:edge} it is more natural to work with the operator $x^{1/2-a} \eth_{\dR} x_0^{1/2+a}$ as an unbounded operator from the space 
$x^{-1/2} L^2_{\iie,\Gamma}(X, \Iie\Lambda_\Gamma^*(X) )$ {\em to itself.} Thought of in this way, we denote it as $P_a$,
\begin{equation}\label{PMap}
	P_a:
	x^{-1/2} L^2_{\iie,\Gamma}(X, \Iie\Lambda_\Gamma^*(X) ) 
	\to x^{-1/2} L^2_{\iie,\Gamma}(X, \Iie\Lambda_\Gamma^*(X) ).
\end{equation}

Our analysis of $\eth_{\dR}$ will proceed in two steps: in the first step we will analyze the behavior of $P_a$  
by adapting two model operators from \cite{Mazzeo:edge} -- the normal operator and the indicial family.
Then, in the second step, we will use the information gleaned about $P_a$ to analyze $\eth_{\dR}$.

\medskip
 \noindent
{\bf Remark.} 
{\em These two steps can be thought of in the following way. We  first analyze $x^{1/2} \eth_{\dR} x^{1/2}$ as a {\em partially complete} edge operator on $W$; complete in the
$(x,y)$ variables with values in $\iie$-operators on $Z$. Then, as a second step, we analyze it as an {\em incomplete} edge operator in the  $(x,y)$ variables with values, again, in $\iie$-operators on $Z$.}

\subsection{The normal operator of $P_a$} $ $\newline 
Recall that every point $q \in Y$ has a neighborhood $W$ which we identify with the product of $\cU\times C(Z)$, where
$\cU$ is a neighborhood of the origin in $\bbR^b \cong T_qY$. If this neighborhood is small enough that 
$\Iie\Lambda^*(X)\rest{W}$ can be identified with the pull-back of some vector bundle over $Z$ and similarly 
for $\Iielaga(X)|_{W}$, then we call $W$  a {\bf basic neighborhood}. In such a $W$, let us fix smooth nonnegative 
cutoff functions $\chi$ and $\wt\chi$, both independent of the $Z$ variables, with supports in $W$ and equaling
one in a neighborhood of $q$, and such that $\wt\chi\chi = \chi$. We  refer to $W$, $\psi$, $\chi$, $\wt\chi$ as a 
{\bf basic setup} at $q \in Y$.

We can identify a basic neighborhood $W$ with a subset of the product of $Z$ with $T_qY^+ \cong \bbR^{+}_s \times  \bbR^{b}_u$ 
and use this identification to model the operator $P_a$ near $q$ by an operator on $Z \times T_qY^+$, the {\bf normal operator} 
of $P_a$ at $q \in Y$. 
Notice that the bundles $\Iie\Lambda^*(X)\rest{W}$, $\Iielaga(X)|_{W}$ as pull-backs of bundles over $Z$, extend naturally to $Z \times T_qY^+$, and that the dilation maps $R_t: T_qY^+ \to T_qY^+$ for any $t >0$ preserve the space of sections of these bundles.

\begin{definition}
The {\bf normal operator} 
$N_q(P_a)$ is the operator whose action on any $u \in \CIc( Z \times T_qY^+, \Iielaga (Z \times T_qY^+))$ is given by 
\begin{equation*}
	N_q(P_a)u = \lim_{r\to 0} R_r^* \; (\psi^{-1})^* \; \wt\chi \; P_a \; \psi^* \; \chi R^*_{1/r} u.
\end{equation*}
\end{definition}

Thus in local coordinates $(s,y,z)$ the action of the normal operator of $P_a$ on a section $u$ is obtained by evaluating $u$ at $(s/r, y/r, z)$, applying $P_a$, dilating back by a factor of $r$, and then letting $r \to 0$. 
It is easy to see that this procedure takes a vector field of the form $a(s, y, z) (s\pa_s) + b(s, y, z)(s\pa_y)$
to the vector field $a(0,0,z)(s\pa_s) + b(0,0,z)(s\pa_y)$, while for a vertical vector field $V$, this procedure returns $V\rest{s=0, y=0}$.
In fact, it is easy to see that this procedure replaces the metric 
\begin{equation*}
	g\rest{W} = g_{\cU}(x,y) + dx^2 + x^2 g_Z(x,y,z)
\end{equation*}
which is a submersion metric with respect to the projection $\cU \times C(Z) \to \cU$, with the product of an $\iie$ metric on $C(Z)$ and the flat metric on $\cU$,
\begin{equation*}
	g_{Z \times T_qY^+} = g_{\cU}(0,0) + ds^2 + s^2 g_Z(0,0,z).
\end{equation*}
It follows that any natural operator associated to $g_{\iie}$ is taken by this procedure to the corresponding natural operator of $g_{Z \times T_qY^+}$ --in particular this is true for $\eth_{\dR}$.

\begin{lemma}
The normal operator of $P_a$ at $q \in Y$ is equal to $s^{1/2-a}\eth_{\dR}s^{1/2+a}$ where $\eth_{\dR}$ is the de Rham operator of the metric $g_{Z \times T_qY^+}$. Thus in local coordinates,
\begin{equation}\label{NormalPa}
	N_q(P_a) 
	= \begin{pmatrix}
		\eth_{\dR}^{Z} + s\eth_{\dR}^{\bbR^{b}}  & -s\pa_s - (f_0-\bN+a+1/2) \\
		s\pa_s + \bN + a +1/2 & - \eth_{\dR}^{Z} - s\eth_{\dR}^{\bbR^{b}}.
	\end{pmatrix}
\end{equation}
\end{lemma}

{\bf Remark.} {\em As explained above, this expression follows by naturality of the de Rham operator. Alternately, one can compute \eqref{NormalPa} directly from \eqref{DNearBdy}.}

\subsection{Localizing the maximal domain} $ $\newline 
The following lemma will allow us  to  ``localize the maximal domain" of $\eth_{\dR}$ near the singular locus.
\begin{lemma}\label{localize}
Let $W$, $\psi$, $\chi$, $\wt\chi$ be a basic setup at $q \in Y$.\\
Let  $u \in x^{-1/2} L^2_{\iie,\Gamma}(X;\Iielaga (X))$ be such that
 $P_a u  \in x^{1/2} L^2_{\iie,\Gamma}(X;\Iielaga (X)).$ Then  $\chi u \in x^{-1/2} L^2_{\iie,\Gamma}(X;\Iielaga (X))$ and $P_a(\chi u) 
  \in x^{1/2} L^2_{\iie,\Gamma}(X;\Iielaga (X))$.
\end{lemma}
\begin{proof}
Clearly $P_a(\chi u) = \chi (P_a u) + [P_a, \chi]u,$ and, since $\chi$ is independent of the $Z$-variables, 
\eqref{DNearBdy} allows us to see that $[P_a, \chi] = \sigma(P_a)(d\chi)= x H$ where $H$ is a multiplication operator 
by smooth bounded functions. Since $u \in x^{-1/2} L^2_{\iie,\Gamma}(X;\Iielaga (X))$ we see that 
$[P_a, \chi]u \in x^{1/2} L^2_{\iie,\Gamma}(X;\Iielaga (X))$, which establishes the lemma.
\end{proof}
\begin{proposition}\label{prop:upshot} Let $u \in x^{-1/2} L^2_{\iie}(X;\Iielaga(X))$ with compact support included in $W$  and such that $\chi =1$ 
on supp $u.$ Then  $P_a u \in x^{1/2} L^2_{\iie}(X;\Iielaga(X))$
 if and only if $N_q (P_a) (u \circ \psi^{-1}) \in s^{1/2} L^2_{\iie}(Z \times T_qY^+, \Iielaga (Z \times T_qY^+))$.
\end{proposition} 

\begin{proof} 
We prove only one implication, the other one is similar. Since we work in the distinguished chart $W$, we may identify $u$ with $u \circ  \psi^{-1}.$ 

Let $\rho$ denote a total boundary defining function. The operator $\frac{\rho}{x} P_a$ is an elliptic $\ie$ differential operator, so elliptic regularity \eqref{EllReg} yields $u \in x^{-1/2} H^1_{\iie}(X;\Iielaga(X)).$

We observe that, in the expression \eqref{PreDNearBdy},  $x(d^Y + \delta^Y_x)$ sends 
$x^{-1/2} H^1_{\iie}(X;\Iielaga(X))$ into $x^{1/2} L^2_{\iie}(X;\Iielaga(X)) $
and 
a similar observation is 
true for $s \eth_{\dR}^{\mathbb{R}^{b_0}}$, so using formulas \eqref{PreDNearBdy} and \eqref{NormalPa}, 
we get $P_a u - N_q (P_a) (u \circ \psi^{-1}) \in s^{1/2} L^2_{\iie},$ which proves the lemma.
\end{proof}

\subsection{Injectivity of $N_q(P_a)$} \label{sec:MapProp} $ $\newline 
We take as an inductive hypothesis that the signature operator on $Z$ is self adjoint with discrete spectrum.
We make two further assumptions: 
\begin{equation}\label{Ass2}
	\boxed{ 
	\begin{array}{rcl}
	a)& \ &\Spec(\eth_{\dR}^{Z}) \cap (-1,1) \subseteq \{0\},  \\
	b)& \ &\mbox{If}\  k = \frac{f_0}2 \Mthen \mathcal{H}^k_{(2)}(Z)=0. 
	\end{array} }
\end{equation} 
By Theorem \ref{theo:cheeger}, b) is a topological condition on $Z$.  
\begin{proposition} \label{prop:rescaling} 
{\item 1)} There exists a  (rigid) iterated edge metric (cf Theorem \ref{thm:metric}) such that 
condition a) is satisfied on all links in $\wh{X}$. Such a metric will be called adapted (rigid) iterated edge.
{\item 2)} Any two adapted (rigid) iterated edge metrics are homotopic within the class of adapted  (rigid) iterated edge metrics.
\end{proposition} 
\begin{proof} 1). 
Observe that condition $a)$ can be arranged to hold along a given stratum by scaling the metric on $Z$. To check that this can be done 
coherently for all links in the Witt space $\wh{X}$, one must retrace the proof of Theorem \ref{thm:metric} concerning the existence of 
rigid iterated edge metrics. Following the inductive step there, we see that we can scale the metric on the link of the highest depth stratum 
so that a) is satisfied without disturbing the corresponding property for all the links of lower depth strata. 

\noindent 2). Retrace the proof of Proposition \ref{prop:homotopymet} along the lines of the previous proof.
\end{proof}
\begin{lemma}\label{lem:Bessel}
Let $a \in (0,1)$ and assume the conditions \eqref{Ass2} and that Theorem \ref{MT1} has been proven for $Z$. Then $N(P_a)$ acting  on
$$
s^{-1/2} L^2_{\iie}(Z \times T_qY^+, \Iielaga (Z \times T_qY^+) )
$$ 
is injective on its maximal domain.
\end{lemma}

\begin{proof}
Define $R = s^{-1/2} N_q(P_0) s^{-1/2}$ (this is effectively $N_q(\eth_{\dR})$), so that
\begin{equation*}
	R =
	 \begin{pmatrix}
	\tfrac1s \eth^Z_{\dR} + \eth^{\bbR^b}_{\dR} & - \pa_s - \tfrac1s(f_0-\bN) \\
	\pa_s + \tfrac1s \bN & -\tfrac1s \eth^Z_{\dR}  - \eth^{\bbR^b}_{\dR} 
	\end{pmatrix}
\end{equation*}
Since $N_q(P_a) = s^{1/2-a} R s^{1/2 +a},$ if
$u$ solves $N_q(P_a)u =0$ then 
\begin{equation*}
	v = s^{\frac {f_0}2 + a}u
\end{equation*}
solves $R s^{-\tfrac{f_0-1}2}v=0.$ 
Clearly  $u \in s^{-1/2}L^2_{\iie}(T_q)$ precisely when we have $v \in s^{\frac{f_0-1}2 +a}L^2_{\iie}(T_q),$ so it suffices to solve
\begin{equation*}
	R s^{-\tfrac{f_0-1}2}v=0, \quad 
	v \in 
	s^{\frac{f_0-1}2 +a}L^2_{\iie}(s^f \; ds dy dz, \Iie\Omega)
	= s^{-\frac{1}2 +a}L^2_{\iie}(ds dy dz, \Iie\Omega).
\end{equation*}
The advantage of this formulation is that $v$ is also in the null space of
\begin{equation}\label{Nov.6.2010.1}
\begin{gathered}
	s^{\tfrac{f_0-1}2} (sR) s^{-\tfrac{f_0-1}2}
	= \begin{pmatrix}
	\eth_{\dR}^Z + s \eth_{\dR}^{\bbR^b} & -s\pa_s + \bN - \tfrac {f_0}2 -\tfrac12 \\
	s\pa_s + \bN - \tfrac {f_0}2 + \tfrac12 & - \eth_{\dR}^Z - s \eth_{\dR}^{\bbR^b}
	\end{pmatrix} \\
	\text{ and }
	s^{\tfrac{f_0-1}2} (s^2R^2) s^{-\tfrac{f_0-1}2}
	= \begin{pmatrix}
	\cK_1 & -2d^Z \\ -2\delta^Z & \cK_{-1}
	\end{pmatrix}, \\
	\Mwhere \cK_\ell =  \Delta^Z + s^2\Delta^{\bbR^b} - (s\pa_s)^2 
		+ (\bN -\tfrac {f_0}2 + \tfrac{\ell}2)^2.
\end{gathered}
\end{equation}

To analyze these systems, we point out that $L^2$ forms on $Z$ satisfy a strong Kodaira decomposition, i.e., every $L^2$ form on $Z$ can be written in a unique way as a sum of a form in the image of $d^{Z}$, a form in the image of $\delta^{Z}$ and a form in the joint kernel of $d^{Z}$ and $\delta^{Z}$. 
As explained in \cite[\S2]{Hunsicker-Mazzeo} {\em weak} Kodaira decompositions are a general feature of Hilbert complexes.
Inductively, we are assuming that $d+\delta$ is essentially self-adjoint and that its closed extension
has  closed range;  this implies, see \cite[Proposition 4.6]{Hunsicker-Mazzeo}, that $d$ has a unique  
closed extension and that this extension
 has closed range (for instance, because $d$ coincides with $d+\delta$ on $(\ker d)^\perp$).
Hence the weak Kodaira decomposition is a strong Kodaira decomposition.

The upshot is that if $v = (\alpha, \beta),$ then we can write 
\begin{equation*}
	\alpha = d^Z\alpha_1 + \delta^Z \alpha_2 + \alpha_3,
	\ \alpha_1 \in (\ker d^Z)^\perp, \alpha_2 \in (\ker \delta^Z)^\perp, \ \alpha_3 \in \ker d^Z \cap \ker \delta^Z
\end{equation*}
and similarly for $\beta.$ 

Inserting this decomposition into 
$s^{\frac{f_0-1}2}(sR)s^{-\frac{f_0-1}2}v=0$ 
and using
\begin{equation*}
	d^Z \bN = (\bN - 1) d^Z, \quad
	\delta^Z \bN = (\bN + 1) \delta^Z
\end{equation*}
yields
\begin{equation*}
\begin{gathered}
	d^Z( \delta^Z \alpha_2 + s \eth_{\dR}^{\bbR^b}\alpha_1 
	- s\pa_s\beta_1 + (\bN-\tfrac {f_0}2 + \tfrac12) \beta_1) \\
	+ \delta^Z( d^Z \alpha_1 + s\eth_{\dR}^{\bbR^b}\alpha_2
	-s\pa_s\beta_2 + (\bN-\tfrac {f_0}2 - \tfrac32)\beta_2) \\
	+s\eth_{\dR}^{\bbR^b}\alpha_3 - s\pa_s \beta_3 
	+ (\bN -\tfrac{f_0}2-\tfrac12)\beta_3 =0
	\\
	d^Z(-\delta^Z\beta_2 - s\pa_{\dR}^{\bbR^b}\beta_1 
	+ s\pa_s \alpha_1 + (\bN - \tfrac {f_0}2 +\tfrac32)\alpha_1) \\
	+ \delta^Z( - d^Z\beta_1 - s\eth_{\dR}^{\bbR^b}\beta_2
	+s\pa_s \alpha_2 + (\bN - \tfrac {f_0}2 - \tfrac12)\alpha_2) \\
	-s\eth_{\dR}^{\bbR^b}\beta_3 + s\pa_s\alpha_3 
	+ (\bN - \tfrac {f_0}2 + \tfrac12)\alpha_3 =0
\end{gathered}
\end{equation*}
and hence another application of the Kodaira decomposition shows that
\begin{gather}
\label{FromR1}
\begin{cases}
	\delta^Z \alpha_2 + s \eth_{\dR}^{\bbR^b}\alpha_1 
	- s\pa_s\beta_1 + (\bN-\tfrac {f_0}2 + \tfrac12) \beta_1=0 \\
	 d^Z \alpha_1 + s\eth_{\dR}^{\bbR^b}\alpha_2
	-s\pa_s\beta_2 + (\bN-\tfrac {f_0}2 - \tfrac32)\beta_2 = 0\\
	-\delta^Z\beta_2 - s\pa_{\dR}^{\bbR^b}\beta_1 
	+ s\pa_s \alpha_1 + (\bN - \tfrac {f_0}2 +\tfrac32)\alpha_1=0 \\
	- d^Z\beta_1 - s\eth_{\dR}^{\bbR^b}\beta_2
	+s\pa_s \alpha_2 + (\bN - \tfrac {f_0}2 - \tfrac12)\alpha_2 = 0
\end{cases} \\
\label{FromR1'}
\begin{cases}
	s\eth_{\dR}^{\bbR^b}\alpha_3 - s\pa_s \beta_3 
	+ (\bN - \tfrac {f_0}2 - \tfrac12) \beta_3 =0 \\
	-s\eth_{\dR}^{\bbR^b}\beta_3 + s\pa_s\alpha_3 
	+ (\bN - \tfrac {f_0}2 + \tfrac12)\alpha_3 =0.
\end{cases}
\end{gather}

We also insert the Kodaira decomposition of $v$ into 
$s^{\frac{f_0-1}2}(s^2R^2)s^{-\frac{f_0-1}2}v,$ and since
$d^Z \cK_{\ell}  = \cK_{\ell-2} d^Z,$ $\delta^Z \cK_{\ell} = \cK_{\ell+ 2} \delta^Z,$ this yields
\begin{gather*}
	d^Z(\cK_3\alpha_1 - 2 \delta^Z\beta_2) + \delta^Z(\cK_{-1}\alpha_2) 
		+ \cK_1\alpha_3 =0, \\
	d^Z(\cK_1\beta_1) + \delta^Z(\cK_{-3}\beta_2 - 2d^Z\alpha_1)
		\cK_{-1}\beta_3 = 0.
\end{gather*}
Once again another application of the Kodaira decomposition shows that
\begin{gather}
\label{FromR2a}
	\cK_3 \alpha_1 = 2 \delta^Z\beta_2 \\
\label{FromR2b}
	2d^Z \alpha_1 = \cK_{-3}\beta_2 \\
	\cK_{-1}\alpha_2 
	= 
	\cK_1\alpha_3 
	=
	\cK_1\beta_1 
	=
	\cK_{-1}\beta_3 = 0.
\label{FromR2}\end{gather}
We are looking for solutions of \eqref{FromR1}-\eqref{FromR2} in 
$s^{-\frac{1}2 +a}L^2_{\iie}(ds dy dz, \Iie\Omega).$

Let us find the null space of $\cK_{\ell}.$
Conjugating by the Fourier transform in $\bbR^b$ (with dual variable $\eta$ to $y$) and introducing the new variables $t = s|\eta|,$ $\hat \eta = \frac{\eta}{|\eta|},$ takes $\cK_{\ell}$ to
\begin{equation*}
	\hat\cK_{\ell} = \Delta^Z + t^2 - (t\pa_t)^2 +(\bN-\tfrac {f_0}2 + \tfrac \ell2)^2. 
\end{equation*}
By assumption $\Delta^Z$ has discrete spectrum
and, since $\Delta^Z$ commutes with $\hat \cK_{\ell},$ we can restrict to the $\lambda$ eigenspace of $\Delta^Z,$
\begin{equation*}
	\hat\cK_{\ell, \lambda} = \lambda + t^2 - (t\pa_t)^2 + (\bN - \tfrac {f_0}2 + \tfrac\ell2)^2.
\end{equation*}
The null space of this operator can be described directly in terms of Bessel functions of an imaginary argument 
\begin{equation*}
	A \, I_{\nu}(t)  +  B \, K_\nu(t), \quad \nu = \sqrt{ \lambda + (\bN - \tfrac {f_0}2 + \tfrac\ell2)^2},\ 
	t \in \bbR^+
\end{equation*}
The functions $I_\nu$ increase exponentially with $t,$ so to stay in a (polynomially weighted) $L^2$ space, we must have $A=0.$
The functions $K_{\nu}$ decrease exponentially with $t$ as $t \to \infty,$ while
\begin{equation*}
	K_{\nu}(t) \sim \begin{cases} t^{-|\nu|} & \Mif \nu \neq 0 \\
	-\log t & \Mif \nu =0 \end{cases}
	\quad \text{ as } t \to 0.
\end{equation*}
We are interested in avoiding $K_\nu \in t^{a-\tfrac12} L^2(\; dt),$ which means we need to have
\begin{equation*}
\begin{gathered}
	1 \leq |\nu|+a = a +\sqrt{\lambda + (\bN - \tfrac {f_0}2 + \tfrac\ell2)^2},
	\text{ for all } a>0 \\
	\text{ hence } 
	1 \leq \lambda + (\bN - \tfrac {f_0}2 + \tfrac\ell2)^2
\end{gathered}
\end{equation*}

If $\lambda \neq 0,$ then our assumption is that $\lambda \geq 1,$ so this is automatic.
If $\lambda=0$ then we are looking for elements in the null space of $\cK_{\ell}$ that are also in the null space of $\Delta^Z,$ so this corresponds to $\alpha_3$ and $\beta_3.$
From \eqref{FromR1'} we see that $\alpha_3 =0$ if and only if $\beta_3 =0:$ indeed if $\alpha_3 =0$ then $\beta_3 = s^{\bN - \tfrac {f_0}2 - \tfrac12}F$ with $F$ independent of $s,$ but this is never in a polynomially weighted $L^2$ in $s,$ and similarly if $\beta_3=0.$ The same reasoning shows that $\eth^{\bbR^b}_{\dR}\alpha_3 = 0$ if and only if $\beta_3 =0$ and viceversa.
Thus, since $\alpha_3$ is in the null space of $\cK_1$ and $\beta_3$ is in the null space of $\cK_{-1},$ to avoid elements of the null space with $\lambda =0$ we need to have either
\begin{equation*}
	1 \leq |\bN - \tfrac {f_0}2 + \tfrac12|
	\Mor
	1 \leq |\bN - \tfrac {f_0}2 - \tfrac12|.
\end{equation*}
This is automatic unless $\bN = \tfrac {f_0}2,$ but this case does not happen since by assumption there are no middle degree harmonic forms on $Z.$

This implies, from \eqref{FromR2}, that $\alpha_2,$ $\alpha_3,$ $\beta_1,$ and $\beta_3$ do not contribute to the null space of $N_q(P_a)$ for $a>0,$ and we only need
to rule out $\alpha_1$ and $\beta_2.$
First note that if $\alpha_1 =0,$ then from \eqref{FromR2b} $\cK_{-3}\beta_2=0,$ but since $\cK_{-3}$ does not have non-zero null space in $s^{-\frac{1}2 +a}L^2_{\iie}(ds dy dz, \Iie\Omega),$ this implies $\beta_2 =0.$ Similarly $\beta_2=0$ implies $\alpha_1 =0.$

Next, substituting \eqref{FromR2b} into the second equation of \eqref{FromR1} we have
\begin{equation*}
	\cK_{-3}\beta_2 + 2s\eth_{\dR}^{\bbR^b}\alpha_2 -2s\pa_s\beta_2
	+ 2(\bN - \tfrac {f_0}2 - \tfrac32)\beta_2 = 0
\end{equation*}
Applying $\cK_{-1}s^{-1}$ kills the second term by \eqref{FromR2}, so
\begin{equation*}
	\cK_{-1}s^{-1}(\cK_{-3} -2s\pa_s + 2(\bN-\tfrac {f_0}2 -\tfrac32) )\beta_2 =0,
\end{equation*}
but 
\begin{equation*}
\begin{gathered}
	\cK_{-3} -2s\pa_s + 2(\bN-\tfrac {f_0}2 -\tfrac32) \\
	= \Delta^Z + s^2\Delta^{\bbR^b} -(s\pa_s)^2 + (\bN-\tfrac {f_0}2 - \tfrac32)^2
	-2s\pa_s + 2(\bN-\tfrac {f_0}2 -\tfrac32)
	= s^{-1} \cK_{-1} s,
\end{gathered}
\end{equation*}
so this says $\cK_{-1} (s^{-2} \cK_{-1} s)\beta_2 =0.$
Since we know that $\cK_{-1}$ does not have non-zero null space in
$s^{-\frac{1}2 +a}L^2_{\iie}(ds dy dz, \Iie\Omega),$ we must have
\begin{equation*}
	s^{-2} \cK_{-1} s \beta_2 =0.
\end{equation*}
Similarly substituting \eqref{FromR2a} into the third equation of \eqref{FromR1} and then applying $\cK_1s^{-1}$ yields $ \cK_1( s^{-2} \cK_1 s \alpha_1)=0$ and hence
\begin{equation*}
	s^{-2} \cK_1 s \alpha_1 = 0.
\end{equation*}

By the reasoning above, the projection onto the $\lambda$ eigenspace of $\Delta^Z$ of $\beta_2$ is (after changing variables to $t$ and $\hat\eta$) of the form 
\begin{equation}
	\cP_{\lambda}\hat \beta_2 
	= C \frac{|\eta|}t K_{\sqrt{\lambda + (\bN-\tfrac {f_0}2 - \tfrac 12)^2}}(t)
\label{Temp1}\end{equation}
and the corresponding projection of $\alpha_1$ is of the form
\begin{equation}
	\cP_{\lambda}\hat \alpha_1 
	= C' \frac{|\eta|}t K_{\sqrt{\lambda + (\bN-\tfrac {f_0}2 + \tfrac 12)^2}}(t).
\label{Temp2}\end{equation}
Thus to avoid elements of the null space we need to have either
\begin{equation*}
	1 \leq 1+a+\sqrt{\lambda + (\bN-\tfrac {f_0}2 + \tfrac 12)^2}
	\Mor
	1 \leq 1+a+\sqrt{\lambda + (\bN-\tfrac {f_0}2 - \tfrac 12)^2}
\end{equation*}
and these are automatic for all $a \geq -1.$

\end{proof}

\subsection{Indicial roots}$ $\newline
Another model operator of $P_a$ is its {\bf indicial family}, defined using the action of $P_a$ on polyhomogeneous expansions.
The indicial family is a one parameter family of operators on $Y$, $I(P_a;\zeta)$ defined by
\begin{equation*}
	P_a ( x^\zeta f) = x^\zeta I(P_a; \zeta)f\rest{x=0} + \cO(x^{\zeta+1}).
\end{equation*}
The base variables at the boundary enter into the indicial family as parameters, so we can speak of the indicial family at the point $q \in Y$ by restricting not just to $x =0$ but to the fibre over $q$. This refinement of the indicial family is denoted by $I_q(P_a;\zeta)$; from \eqref{DNearBdy} it is given by
\begin{equation*}
	I_q(P_a; \zeta) 
	= \begin{pmatrix}
		\eth_{\dR}^{Z}   & -\zeta - (f_0-\bN+a+1/2) \\
		\zeta + \bN + a +1/2 & - \eth_{\dR}^{Z}
	\end{pmatrix}, 
\end{equation*}
which coincides with the indicial family of the normal operator at $q \in Y$.
The values of $\zeta$ for which $I_q(P_a;\zeta)$ fails to be invertible (on $L^2_{\iie}(Z)$) are known as the {\bf indicial roots} of $P_a$ at $q$, or the 
boundary spectrum of $P_a$ at $q$, 
\begin{equation*}
	\spec_b(N_q( P_a )).
\end{equation*}
As we show below, this set depends on $\spec \eth_{\dR}^Z$, and hence relies on the inductive hypothesis on $Z$. 

An equivalent model of $P_a$ is the {\bf indicial operator}: 
\begin{equation*}
	I_q(P_a) 
	= \begin{pmatrix}
		\eth_{\dR}^{Z}   & -t\pa_t - (f_0-\bN+a+1/2) \\
		t\pa_t + \bN + a +1/2 & - \eth_{\dR}^{Z}.
	\end{pmatrix}
\end{equation*}
It is related to the indicial family by the Mellin transform,
\begin{equation*}
	\cM( I_q(P_a) u)(\zeta)
	= I_q(P_a; -i\zeta) \cM(u)(\zeta).
\end{equation*}
Recall that this transform is defined, e.g., for $u \in \CIc( \bbR^+ )$ by
\begin{equation}\label{DefMellin}
	\cM u(\zeta) = \int_0^\infty u(x) x^{i\zeta - 1} \; dx,
\end{equation}
and extends to an isomorphism between weighted spaces
\begin{equation}\label{MellinIso}
	x^\alpha L^2 \left(\bbR^+, \frac{dx}x \right) \xrightarrow{\cong} L^2\lrpar{ \{ \eta = \alpha \}; d\xi }
\end{equation}
where $\eta = \Im\zeta$ and $\xi = \Re\zeta$.
The inverse of the Mellin transform as a map \eqref{MellinIso} is given by
\begin{equation*}
	\cM^{-1}(v)(x) = \frac1{2\pi} \int_{\eta = \alpha} v(\zeta) x^{-i\zeta} \; d\xi.
\end{equation*}

\begin{lemma}\label{lem:Indicial}
The indicial roots of $P_a$ are contained in the union of
\begin{equation}\label{IndSets}\tag{6.16}
\begin{gathered}
	\bigcup_{k \neq \tfrac {f_0}2}
	 \lrbrac{ - \tfrac {f_0}2 - a \pm \abs{ k -\tfrac {f_0}2 \pm \tfrac 12 } }, 
	\bigcup_{\lambda_k \neq 0}
	\lrbrac{ -  \tfrac {f_0}2 -a \pm \sqrt{\lambda_k+(k-\tfrac {f_0}2 + \tfrac {\ell}2)^2 } }, \\
	\text{ and }
	\bigcup_{\lambda_k \neq 0}
	\lrbrac{ 1 -  \tfrac {f_0}2 -a \pm \sqrt{\lambda_k+(k-\tfrac {f_0}2 + \tfrac {\ell'}2)^2 } } 
\end{gathered}
\end{equation}
where $k \in \{0, \ldots, f_0 \},$  $\lambda_k$ is in the spectrum of $\Delta^Z$ acting on $k$-forms, $\ell \in \{\pm1, \pm3\}$ and $\ell' \in \{ \pm 1\}.$


The indicial operator of $P_a$ has a bounded inverse on the space $t^{-1/2}L^2_{\iie}(Z \times \bbR^+_t)$ for all $a \in (0,1)$. 
\end{lemma}

\begin{proof}
An analysis similar to -- but simpler than -- that above applies to the equation $I_q(P_a)u=0$. Indeed, it suffices to replace $(t\pa_t)^2 - t^2$ in the `equations to solve' by $(t\pa_t)^2$. 
Since the solutions to $(t\pa_t)^2u = \nu^2u$ are linear combinations of $t^{\nu}$ and $t^{-\nu}$, the solutions of $I_q(P_a)u=0$ are obtained from the solutions to $N_q(P_a)u =0$ by replacing each $I_v(t)$ by $t^\nu$ and each $K_\nu(t)$ by $t^{-\nu}$.
Both of these contribute indicial roots, since for the indicial family we do not impose growth restrictions.

For the indicial operator, we are imposing growth restrictions and, as before, asking for solutions to be in $t^{-1/2}L^2(t^f \; dt)$ excludes 
those involving $t^\nu$, hence conditions ($a$) and ($b$) show that there are no solutions involving $t^{-\nu}$ for $a>0$.
Thus the proof of Lemma \ref{lem:Bessel} shows that the indicial operator $I_q(P_a)$ is injective on $t^{-1/2}L^2(Z \times \bbR^+)$ as long as $a>0$.

Similarly, the proof of Lemma \ref{lem:Bessel} shows that if there is a non-zero solution to
$I_q(P_a;\zeta)u=0$ then $\zeta$ must be in one of the sets in \eqref{IndSets}.
An advantage of the indicial family is that we can bring to bear our inductive hypotheses about $\eth_{dR}^Z$.
Indeed, decompose $I_q(P_a)$ as 
\begin{multline*}
	I_q( P_a )(\zeta)
	= \begin{pmatrix}
		\eth_{\dR}^{Z}   & -\zeta - (f_0-\bN+a+1/2) \\
		\zeta + \bN + a+1/2 & - \eth_{\dR}^{Z}
	\end{pmatrix} \\
	= \eth_{\dR}^{Z}
	\begin{pmatrix} \Id & 0 \\ 
	0 & - \Id \end{pmatrix}
	+ \begin{pmatrix}
		0   & -\zeta - (f_0-\bN+a+1/2) \\
		\zeta + \bN + a+1/2 & 0
	\end{pmatrix}
	= A + B.
\end{multline*}
Inductively we know that $A$ is essentially self-adjoint, has closed range, and its domain, $\cD(A)$, includes compactly into $L^2_{\iie}(Z)$.
It follows that the operator $B : \cD(A) \to L^2_{\iie}(Z)$ (where $\cD(A)$ is endowed with the graph norm) is compact, i.e., $B$ is relatively compact with respect to $A$, and so
$I_q(P_a;\zeta)$ has a unique closed extension, has closed range, and its domain is also $\cD(A)$.

Since $\eth_{\dR}^Z$ is essentially self-adjoint, the adjoint of $I_q( P_a )(\zeta)$ on $L^2(Z)$ is
\begin{equation*}
\begin{split}
	I_q( P_a; \zeta)^*
	&= \begin{pmatrix}
		\eth_{\dR}^{Z}   & \zeta + \bN + a+1/2 \\
		-\zeta - f + \bN -a-1/2) & - \eth_{\dR}^{Z}
	\end{pmatrix} \\
	&= I_q(P_a; -(\zeta + f + 2a+1) )
\end{split}
\end{equation*}
Notice that $\zeta$ is in one of the sets in \eqref{IndSets} if and only if $-(\zeta + f + 2a +1)$ is.
Thus we see that if $\zeta$ is not in one of these sets, then $I_q(P_a; \zeta)$ is in fact invertible with bounded inverse. In fact, since the domain of $I_q(P_a; \zeta)$ is $\cD(A)$, its inverse is a compact operator.
This proves that \eqref{IndSets} contains the indicial roots of $N_q(P_a)$.
Denote the inverse of $I_q(P_a; \zeta)$ by 
\begin{equation*}
	Q(\zeta) : L^2_{\iie}(Z) \to \cD(A) \hookrightarrow L^2_{\iie}(Z).
\end{equation*}
We obtain an inverse for $I_q(P_a)$ as an operator on $t^{-1/2}L^2_{\iie}(R^+ \times Z)$ by applying the inverse Mellin transform to $Q(\zeta)$ along the line $\eta = -\frac f2 -1$, which we can do as long as $-\frac f2 -1$ is not an indicial root.
If ($a$) and ($b$) hold, then this is true for all $a \in (0,1)$.
\end{proof}

\subsection{Bijectivity of $N_q( P_a)$} $ $\newline 
We now show that the normal operator $N_q( P_a)$ is a bijection between its maximal domain in $s^{-1/2}L^2_{\iie}(Z \times T_q Y^+)$
and $s^{-1/2}L^2_{\iie}(Z \times T_q Y^+)$ when $0 < a < 1$. 

Observe first that by Lemma~\ref{lem:Bessel}, assuming conditions $a)$ and $b)$, 
this mapping is injective for these values of $a$. Therefore, a simple 
duality argument shows that it suffices to show that it has closed range.  Indeed, 
since $\eth_{\dR}$ is symmetric on $L^2_{\iie}(X; \Iie\Lambda^* X)$, the operator
\begin{equation*}
	D_0 : x^{-1/2}L^2_{\iie}(X;\Iie\Lambda^*X) \to x^{1/2}L^2_{\iie}(X; \Iie\Lambda^*X)
\end{equation*}
coincides with its formal adjoint. It is then straightforward that the formal adjoint of $P_0$ is
\begin{equation*}
	(P_0)^* = x^{-1/2} \eth_{\dR} x^{3/2} : 
	x^{-1/2}L^2_{\iie} (X;\Iie\Lambda^*X) 
	\to 
	x^{-1/2}L^2_{\iie} (X;\Iie\Lambda^*X), 
\end{equation*}
and similarly, 
\begin{equation}\label{PaTranspose}
	(P_a)^*
	= (x^{1/2-a} \eth_{\dR} x^{1/2+a})^*
	= x^{-1/2+a} \eth_{\dR} x^{3/2-a}
	=P_{1-a}.
\end{equation}

\begin{lemma} \label{lem:NClosedRange}
The normal operator $N_q(P_a)$ is bijective as an operator on $s^{-1/2}L^2_{\iie}(Z \times T_qY^+)$ 
acting on its maximal domain, for all $a \in (0,1)$.
\end{lemma}
\begin{proof}
For the duration of this section we write $L^2_{\iie}$ simply as $L^2$ and also omit the bundle $\Iie\Lambda^* (Z \times T_qY^+)$ 
to simplify notation. 

Following the proof of Lemma \eqref{lem:Bessel}, we pass to the Fourier transform in the 
horizontal variables, introducing the variable $\eta$ dual to $y$, and then rescale by setting $t = s|\eta|$, 
$\hat \eta = \eta/ |\eta|$. This leads to the family of operators 
\begin{equation*}
	\wt N_q(P_a, \hat \eta)
	= \begin{pmatrix}
		\eth_{\dR}^{Z} + t\cl{\hat\eta}   & -t\pa_t - (f_0-\bN+a+1/2) \\
		t\pa_t + \bN + a+1/2 & - \eth_{\dR}^{Z} - t\cl{\hat\eta}
	\end{pmatrix}
\end{equation*}
where $ \cl{\hat\eta} = \df i_{\hat{\eta}} + \df e_{\hat{\eta}}$ is Clifford multiplication and $\hat\eta$ lies in the unit sphere $S^{b-1}$.  
Notice that
\begin{equation}
\wt{N}_q(P_a, \hat\eta) = I(P_a) + t A(\hat\eta)
\label{relindnorm}
\end{equation}
where $A$ is a bounded matrix.  These operations are all reversible, so it enough to study this simper family of operators, and in 
particular to show that it is a bijection from its maximal domain in $t^{-1/2}L^2$ to $t^{-1/2}L^2$.  We have already shown in 
Lemma~\ref{lem:Bessel} that this operator is injective, and by duality, i.e.\ using injectivity for $\wt{N}_q(P_a, \hat\eta)^* = 
\wt{N}_q(P_{1-a}, -\hat\eta)$, it also has dense range. Thus it suffices to show that it has closed range, and to prove this we 
follow a standard procedure by constructing local parametrices for $\wt{N}_q(P_a,\hat\eta)$ in the two regions $(0,2T) 
\times Z$ and $(T,\infty) \times Z$ for any fixed $T$. Notice that we only need to construct a right parametrix
for $\wt{N}_q(P_a, \hat\eta)$, since a left parametrix is obtained as the dual of a right parametrix for $\wt{N}_q(P_{1-a}, -\hat\eta)$. 

First consider the region $t < 2T$.  We have indicated in \S 5.5 that $I_q(P_a)$ has an inverse $H_0 = I_q(P_a)^{-1}$ on $t^{-1/2}L^2$, and hence
\[
\wt{N}_q(P_a, \hat\eta) \circ H_0 = \mbox{Id} + t A(\hat\eta) H_0
\]
Since $H_0$ maps into the domain of $I_q(P_a)$, and the restriction of this domain to forms with bounded
support in $t$ includes compactly in $t^{-1/2}L^2$, we see that the second term on the right is a compact operator
on this subspace. 

For forms supported in $t > T$, as in  \cite[Lemma 5.5]{Mazzeo:edge}, consider the partial symbol 
\begin{equation*}
	\wt\sigma(\wt N_q(P_a, \hat\eta)^2)
	=
	\begin{pmatrix}
	\Delta^{Z} + t^2 + t^2\tau^2 & 0 \\ 0 & -(\Delta^{Z} + t^2 + t^2\tau^2)
	\end{pmatrix}
\end{equation*}
where $\tau$ is the variable dual to $\pa_t$. Clearly, 
\begin{equation*}
	|\ang{ \wt \sigma(N_q^2)u, u }| \geq t^2(1+\tau^2)\norm{u};
\end{equation*}
the inner product and norm are those of of $t^{-1/2}L^2$. 
The operator norm of $\wt \sigma(N_q^2)^{-1}$ is bounded by $t^{-2}(1+\tau)^{-2}$, so that 
\begin{equation*}
H_{\infty}(u) = \int e^{it\tau} \wt\sigma(N_q^2)^{-1} \hat u \; d\tau
\end{equation*}
defines a parametrix for $N_q^2(\hat \eta)$ in the large region. As before, $\wt N_q^2 \circ H_\infty = \wt N_q \circ (\wt N_q \circ H_\infty) 
= \mbox{Id} + B$ where $B$ is compact, hence $\wt N_q(P_a, \hat\eta) \circ H_\infty$ is the parametrix we seek. 

Now choose a partition of unity $\{\chi_0, \chi_\infty\}$ relative to the open cover $(0,2T) \cup (T,\infty)$, and fix smooth
functions $\tilde{\chi}_j$ such that $\tilde{\chi}_j = 1$ on the support of $\chi_j$ and which vanish outside a slightly larger
neighbourhood.  The right parametrix is then given by
\[
\wt{H} = \tilde{\chi}_0 H_0 \chi_0 + \tilde{\chi}_\infty (\wt{N}_q(P_a,\hat\eta) \circ H_\infty) \chi_\infty.
\]

The last thing we need to check is that $\wt{N}_q(P_a,\hat \eta) \circ \wt{H} = \mbox{Id} - Q$, where $Q$ is
compact. However, 
\[
Q = [ \wt{N}_q(P_a, \hat\eta), \tilde\chi_0] H_0 \chi_0 + [ \wt{N}_q(P_a, \hat\eta), \tilde\chi_\infty] (\wt{N}_q(P_a,\hat\eta \circ H_\infty)\chi_\infty 
+ \tilde{\chi}_\infty B \chi_\infty. 
\]
The two commutator terms are operators of order $0$, i.e.\ multiplication operators, with compact support, and using
the mapping properties of these two parametrices, we conclude that $Q$ is compact, as claimed. 

This proves that $\wt{N}_q(P_a,\hat\eta)$ is Fredholm, which completes the argument. 
\end{proof}

\subsection{Integration by parts identity for $N_q(P_a)$} $ $\newline
In computing the indicial roots of $P_a$, we have made strong use of the symmetries of the normal operator of $P_a$, namely the translation invariance along horizontal directions (i.e., those tangent to $Y$) and dilation invariance in $T_qY^+$. 
In this section we exploit this invariance to establish an integration by parts identity, which will ultimately allow us to show that any `extra' vanishing of $N_q(P_a)u$ at $x=0$ translates to some degree of vanishing of $u$ at $x=0$, the latter degree bounded by the indicial roots of $N_q(P_a)$.

We will need the Sobolev spaces on $Z \times T_q Y^+$ analogous to those on $X$.
\begin{definition} \label{def:partial} Let $N \in \mathbb{N}.$  We define 
$H_{\pie}^{N}(Z \times T_qY^+;\Iie\Lambda^* )$ to be the set of 
$u \in L^2_{\iie}(Z \times T_qY^+;\Iie\Lambda^* )$ such that 
for any positive integer $p \leq N, $
$$
X_1 \ldots X_p u \in L^2_{\iie}(Z \times T_qY^+;\Iie\Lambda^* )
$$ where the $X_j$ are vector fields which are either of the form $s \partial_s, s\partial_{u_j} \, ( 1\leq j \leq b_0)$ 
or of the form $X(z,s,u)=X(z)$ for each $(z,s,u) \in Z \times T_qY^+,$ where $X(z)$ is an edge vector field of the fibre $Z=Z_{q}.$ Notice that these vectors fields $s \partial_s, s\partial_{u_j} \,X(z)$ generate 
a Lie algebra.
\end{definition}

As we have already used in \S\ref{sec:MapProp},
if a function in $L^2_{\iie}(X)$ is $\cO(x^\gamma)$ near $x=0$ then we must have $2\gamma+f_0 >-1$. As the $L^2$ cut-off will be very important below we introduce the function
\begin{equation}\label{Defdelta}
	\delta_0(\gamma) = \gamma - \frac{f_0+1}2,
\end{equation}
thus a function in $\cO(x^{\gamma})$ is in $x^aL^2(x^{f_0 }\; dx )$ precisely when $\gamma>\delta_0(a)$. 

Briefly, let us abbreviate $L^2_{\iie}( Z \times T_qY^+, \Iie \Lambda^*( Z \times T_qY^+))$ by $L^2_{\iie}(q)$.
Let $C$ be a fixed number in $[-1/2,1]$ and $\eps \in (0,1).$  Let now $R$ be the unbounded operator induced by $N_q(P_0)$ on $s^{C +\eps}L^2_{\iie}( q )$ with domain $C^\infty_c$; with a small abuse of notation
we denote also by $R$ the operator induced by $N_q(P_0)$ on $s^{C -\eps}L^2_{\iie}( q )$ (acting distributionally).
We consider the natural pairing 
$\ang{\cdot, \cdot}: s^{C +\eps}L^2_{\iie}( q )\times s^{C -\eps}L^2_{\iie}( q )\to \bbC$
between these two spaces
\footnote{Recall that this pairing is given by $\langle u,v \rangle:= (u',v')_{s^C L^2}$ if $u=s^\epsilon u'$ and $v=s^{-\epsilon} v'$.}.
Let $R^t$ be the formal transpose of $R$ with respect to this pairing. $R^t$ is a differential operator and  we let it act, distributionally, 
on $s^{C-\eps} L^2_{\iie}( q )$.

We will establish that, if
\begin{equation*}
	u \in s^{C}L^2_{\iie}( q ), \quad 
	v \in s^{C-\eps}L^2_{\iie}( q ), \quad
	\Mand Ru, \;R^t v\; \in \; s^{C+\eps}L^2_{\iie}( q )
\end{equation*}
then, with respect to the natural pairing $\ang{\cdot, \cdot}$ above, we have 
$$\ang{v, Ru} = \ang{u, R^t v}.$$
Notice that, although both pairings make sense, this is not an instance of the definition of $R^t$, since both $u$ and $v$ are thought of as elements of $s^{C-\eps}L^2_{\iie}( q )$.

Assume inductively that we have shown $\cD_{\max}(\eth_{\dR}) = \cD_{\min}(\eth_{\dR})$ for stratifications of depth at most $m-1$ so that in particular
\begin{equation*}
	\ang{ \eth_{\dR}^{Z}u, v } = \ang{u, \eth_{\dR}^{Z}v}
\end{equation*}
for any two elements of $\cD_{\max}(\eth_{\dR}^{Z})$.

On the one hand we know that, for $u, v \in s^{C}L^2_{\iie}( q )$, the natural inner product is given by
\begin{equation*}
	\ang{u,v} = \int s^{-2C} u \wedge * v
\end{equation*}
and, on the other, the normal operator is given by
\begin{equation*}
	N_q = N_q(P_a) 
	= \begin{pmatrix}
		\eth_{\dR}^{Z} + s\eth_{\dR}^{\bbR^b}  & -s\pa_s +\bN - f- (a+1/2) \\
		s\pa_s + \bN + (a+1/2) & - \eth_{\dR}^{Z} - s\eth_{\dR}^{\bbR^b}
	\end{pmatrix},
\end{equation*}
so as anticipated we only have to justify integrating by parts the $s\pa_s$ and $s \eth_{\dR}^{\bbR^b}$.

We can  assume that we are working with sections compactly supported in a basic neighborhood $W.$

Our main tool is the Mellin transform \eqref{DefMellin}.
Using the inclusions $x^aL^2 \subset x^bL^2$ whenever $b<a$ it follows that the Mellin transform of a function in
$x^aL^2(\bbR^+, dx)$ is holomorphic in the half-plane $\{ \eta < a-1/2 \}$.
The Mellin transform is very useful for studying asymptotics. For instance, if $u$ is polyhomogeneous then $\cM u$ extends to a meromorphic function on the whole complex plane with poles at locations determined by the exponents occuring in the expansion of $u$.
Switching from $L^2(\bbR^+)$ to $L^2_{\iie}(X)$, assume that $\omega$ is supported in a basic neighborhood $W$ of $q \in Y$, then we have
\begin{equation*}
	\omega \in s^\alpha L^2_{\iie}(X)
	\iff \cM \omega \in L^2\lrpar{ \{ \eta = \delta_0(\alpha_0) \}, d\xi; L^2 (dy \; \dvol_{Z}) }
\end{equation*}
where $\cM$ denotes Mellin transform in $s$ (in the usual coordinates),  $d y$ denotes the Lebesgue 
measure of $\mathbb{R}^{b_0}$, and $\dvol_{Z}$ denote the volume form  associated to 
the edge iterated  metric of $Z.$
Notice that $\cM\omega$ extends to a holomorphic function on the half-plane $\{ \eta < \delta_0(\alpha_0) \}$ with values in $L^2(dy \; \dvol_{Z})$.

Elliptic regularity (via the symbolic calculus) tells us that elements in the null space of an elliptic edge-operator are in $H^\infty_{\iie }(X;\Iie \Lambda^*)$, and hence smooth in the interior of the manifold. However, the derivatives of elements in $H^{\infty}_{\iie}(X;\cE)$ will typically blow-up at the boundary, which is just to say that knowing $\rho \pa_y u \in L^2_{\iie}(X; \Iie \Lambda^*)$ tells us that $\pa_y u \in \rho^{-1} L^2_{\iie}(X;\Iie \Lambda^*)$.
Using the Mellin transform we can turn this around: 
if $u$ is in the null space of an elliptic $\ie$-operator, $A$, as a map
\begin{equation*}
	A: \rho^{\alpha}L^2_{\iie}(X; \Iie \Lambda^*) \to
	\rho^{\alpha}L^2_{\iie}(X; \Iie \Lambda^*)
\end{equation*}
then, in the absence of indicial roots, we can view $u$ as an element of a space with a stronger weight at the cost of giving up tangential regularity at the boundary. We shall concentrate directly on the normal operator of $P_a$, even though
much of what we prove could be extended to more general differential operators.

\begin{lemma} \label{lem:integration}
Let $W$ be a basic neighborhood for the point $q \in Y$. Set $R= N_q(P_a)$ and assume 
 that, for some $\alpha \in \bbR$ and $\eps \in(0 , 1)$,
\begin{equation}\label{SymReq}
	\{ \Re \zeta + \tfrac f2 + \tfrac12: \zeta \in \spec_b(  R ) \} \cap [\alpha-\eps, \alpha+ \eps] \subseteq \{ \alpha \}.
\end{equation}

\begin{enumerate}
\item \label{lem:Asympt}
Assume $v \in s^\alpha L^2_{\iie}(Z \times T_qY^+ ; \Iie \Lambda^*)$ is supported in $W$ and
$R v \in s^{\alpha+\eps}L^2_{\iie}(Z \times T_qY^+ ; \Iie \Lambda^*)$
then
\begin{equation*}
\begin{split}
	v &\in s^{\alpha+\eps} L^2( s^{f_0} \;d s\;dvol_{Z},  H^{-1}(dy) \otimes \Iie \Lambda^*)  \\
	&= \{ s^{\alpha+\eps} u: u \in \Diff^{1}(Y) L_{\iie}^2(W ; \Iie \Lambda^*) \}
\end{split}
\end{equation*}
Moreover, as a map into $L^2( \dvol_Z, H^{-1}(dy) \otimes \Iie \Lambda^* )$ the Mellin transform of $v$ is holomorphic in the half-plane 
$\{\eta < \delta_0( \alpha+\eps) \}$.

\item \label{lem:IntByParts}
Assume that 
$u \in s^\alpha L^2_{\iie}(Z \times T_qY^+ ; \Iie \Lambda^*)$ and 
$w \in s^{\alpha-\eps}H^2_{\pie}(Z \times T_qY^+ ; \Iie \Lambda^*)$ (cf Definition \ref{def:partial}) are such that
\begin{equation*}
\begin{gathered}
	\supp u \subseteq W \\
Ru, R^tw \in s^{\alpha+\eps}L^2_{\iie}(Z\times T_qY^+ ; \Iie \Lambda^*),	
\end{gathered}
\end{equation*}
then with respect to the natural pairing
\begin{equation*}
	\ang{\cdot, \cdot}: s^{\alpha-\eps}L^2_{\iie}(Z\times T_qY^+ ; \Iie \Lambda^*) \times s^{\alpha+ \eps} L^2_{\iie}(Z \times T_qY^+ ; \Iie \Lambda^*) \to \bbC
\end{equation*}
we have $\ang{ w, Ru } = \ang{ u, R^t w }$.
\end{enumerate}
\end{lemma}

\begin{proof}
{\bf (1):}
Since $v$ is supported in a normal neighborhood of $q \in Y$, we can write 
\begin{equation*}
	I_q( R ) v = H v + h
\end{equation*}
where $I_q(R)$ is the indicial operator of $R$ and $H$ contains all of the `higher order terms' at the boundary, e.g., $s^2\pa_s$, $s\pa_u$. 

Passing to the Mellin transform, and using that $I_q(R;\zeta)$ depends polynomially on $\zeta$, we have an equality
\begin{multline}\label{Hol1}
	\cM v (\zeta)= I(R; i\zeta)^{-1} \lrpar{\cM (H v + h) (\zeta)} \\
	\text{ as meromorphic functions }
	\{ \eta < \delta_0(\alpha) \} \to L^2(dy \; \dvol_{Z} ;  \Lambda^*),
\end{multline}
of course since the left hand side is holomorphic on this half-plane so is the right hand side.
On the other hand, $\cM(h)$ is a holomorphic function into this space on the half plane $\{ \eta < \delta_0(\alpha +\eps) \}$, and, reasoning as in \cite{Mazzeo:edge},
$\cM( H v)$ extends holomorphically to this half plane but we have to give up tangential regularity,
\begin{equation*}
	\cM( H v) :
	\{ \eta < \delta_0(\alpha+\eps) \} \to L^2(\dvol_{Z}; H^{-1}(dy) \otimes  \Lambda^*) 
	\text{ holomorphically}.
\end{equation*}
This gives us an extension of \eqref{Hol1} to
\begin{multline}\label{Hol2}
	\cM v (\zeta)= I( R; i\zeta)^{-1} \lrpar{\cM ( H v + h) (\zeta)} \\
	\text{ as meromorphic functions }
	\{ \eta < \delta_0(\alpha+\eps) \} \to L^2(\dvol_{Z}; H^{-1}(dy) \otimes  \Lambda^*).
\end{multline}
The possible poles occur at indicial roots of $R$, so the first possibility would occur at $\zeta = \delta_0(\alpha)$, and by hypothesis this is the only potential indicial root with real part less than or equal to $\delta_0(\alpha+\eps)$. However we know that
\begin{equation*}
	v(s,y,z) = \frac1{2\pi} \int_{\eta = \delta_0(\alpha)} \cM v(\xi,y,z) s^{-i\zeta} \; d\xi
\end{equation*}
so in particular (as $1/\xi^2$ is not integrable) $\cM v$ does not have any poles on this line. Hence
\begin{multline*}
	\cM v (\zeta)= I ( R ; i \zeta )^{-1} \lrpar{\cM ( H v + h) (\zeta)} \\
	\text{ as holomorphic functions }
	\{ \eta < \delta_0(\alpha+\eps) \} \to L^2( \dvol_{Z};H^{-1}(dy) \otimes  \Lambda^* )
\end{multline*}
and we conclude that
\begin{equation*}
	v \in s^{\eps} L^2(s^{f_0} ds \dvol_{Z}; H^{-1}(dy) \otimes  \Lambda^*).
\end{equation*}

{\bf (2):}
This follows as in \cite[Corollary 7.19]{Mazzeo:edge} by analyzing the Mellin transform. 
Without loss of generality we can arrange, by conjugating $R$ with an appropriate power of $s$,  to work with the measure $\frac1s(ds dy \dvol_Z )$. We will assume, for the duration of the proof, that this has been done without reflecting it in the notation.
This has the advantage that the Parseval formula for the Mellin transform has the form\footnote{
For the measure $s^{f_0} \; ds$ the Parseval formula for the Mellin transform takes the form
\begin{equation*}
	\int_0^\infty g_1(s) g_2(s)  s^{f_0} \; ds 
	= \int_{ \eta = C } \cM g_1(\zeta) \;\; \cM g_2( -(f_0+1)i -\zeta ) \; d\xi
\end{equation*} }
\begin{equation*}
	\int_0^\infty g_1(s) g_2(s) \; \frac{ds}s 
	= \int_{ \eta = C } \cM g_1(\zeta) \;\; \cM g_2(-\zeta ) \; d\xi
\end{equation*}
with $C$ chosen so that the integral on the right makes sense.

Notice that from knowing $u,w \in s^{-\eps}L^2_{\iie}$ and $R(u), R^t(w) \in s^\eps L^2_{\iie}$ the respective Mellin transforms are defined on the half-planes
\begin{gather*}
	\cM(w)(\zeta) \text{ on } \{ \eta \leq -\eps \},
	\quad
	\cM(Ru)( - \zeta ) \text{ on } \{ \eta \geq -\eps \} \\
	\cM(R^tw)(\zeta) \text{ on } \{ \eta \leq \eps \},
	\quad
	\cM(u)( - \zeta ) \text{ on } \{\eta \geq \eps \}
\end{gather*}
so that {\em a priori} there is in each case only one choice for the constant $C$ appearing in Parseval's formula. 
More precisely, $C=-\eps$ for the first pair and $C=\eps$ for the second pair.

Using part 1) of this Lemma  we know we can extend 
\begin{equation*}
	\cM(u)( - \zeta) \text{ to } \{ \eta \geq -\eps \}
\end{equation*}
albeit with a loss in tangential regularity.
Fortunately this loss in tangential regularity is compensated by a gain in tangential regularity in $\cM(R^t w)$ in this same region. 
Indeed, since $w \in s^{-\eps}H^2_{\pie}$,
we know that $R^t w \in s^{\eps}H^1_{\pie}$ hence we have $\pa_y R^t w \in s^{-1+\eps}L^2_{\iie}$. It follows that the Mellin transform of $\pa_y R^t w$ is a holomorphic map from $\{ \eta < -1+\eps \}$ into $L^2(dy\; \dvol_Z ; \Lambda^*)$ and hence on this same half-plane $\cM(R^t w)$ maps holomorphically into $L^2(\dvol_Z, H^1(dy) \otimes  \Lambda^*)$.
Again applying Calderon's complex interpolation method, we conclude that
\begin{equation}\label{InterpolatePW}
	\cM(R^t w)(\zeta) \in L^2(dz, H^{\eps - \eta} )
	\text{ for } \eps -1 \leq \eta \leq \eps.
\end{equation}
The same reasoning applies to $w$.

Thus if we start out with $\ang { u, R^t w }$ which we can write as
\begin{equation*}
	\int \int_{\eta =\eps } \cM(R^t w)(\zeta) \;\; \cM(u)(-\zeta) \; d\xi \; dy \; \dvol_Z,
\end{equation*}
we can deform the contour from $\{\eta =\eps\}$ to $\{\eta =-\eps\}$ and throughout this deformation the integrand stays holomorphic with the loss in tangential regularity of $\cM(u)$ exactly compensated by a gain in regularity by $\cM(R^t w)$, i.e. the integrand makes sense as a pairing throughout the deformation. 
Moreover the integrand is holomorphic in this region and so the value of the integral does not change during the deformation.
Hence we can write $\ang{ u, R^t w }$ as
\begin{equation*}
	\int \int_{\eta =-\eps } \cM(R^t w)(\zeta) \;\; \cM(u)(-\zeta) \; d\xi \; dy \; \dvol_Z.
\end{equation*}
Now integrating each term by parts we write this as
\begin{equation*}
	\int \int_{\eta =-\eps } \cM(w)(\zeta) \;\; \cM(Ru)(-\zeta) \; d\xi \; dy \; \dvol_Z,
\end{equation*}
which by another application of Parseval's formula we recognize as $\ang{ w, Ru }$.
\end{proof}

\subsection{End of induction: $\eth_{\dR}$ is essentially self-adjoint and Fredholm } $ $\newline
Our next task is to use the information gleaned in the previous section to show that elements of the maximal domain of $\eth_{\dR}$ 
as an operator on $L^2_{\iie}(X; \Iie\Lambda^*X)$ are automatically in $\rho^{\eps}L^2_{\iie,\Gamma}(X; \Iie\Lambda^*X)$.

\begin{proposition}\label{DomVanBdy} Up to rescaling suitably the metric, the following is true.
{\item 1)} Let $u$ be in the maximal domain of $\eth_{\dR}$ as an operator on $L^2_{\iie}(X;\Iie\Lambda^*X)$ then 
for any $\eps\in (0,1)$, $u \in \rho^{\eps}H^1_{\iie}(X;\Iie\Lambda^*X)$. 
{\item 2)} The maximal domain $\cD_{\max}(\eth_{\dR})  $ is compactly embedded in $L^2_{\iie}.$  
\end{proposition}
\begin{proof}
We can immediately localize and assume that $u$ has support in a locally trivialized 
neighborhood $\calU \times C(Z)$ of the highest depth stratum. 

We begin with the following intermediate result.
\begin{proposition} \label{prop:intermediate}
Let $u$ have compact support in $\calU \times C(Z)$ and lie in the maximal domain of $\eth_{\dR}$ as an operator on 
$L^2_{\iie}(X;\Iie\Lambda^*X).$ Then, for any $\eps\in (0,1)$, $u \in x^{\eps}L^2_{\iie}(X;\Iie\Lambda^*X)$.
\end{proposition}
\begin{proof} 
Fix $\eps_0 \in (0,1)$ small enough that 
\begin{equation*}
	\{ \Re \zeta + \tfrac f2 + \tfrac12: \zeta \in \spec_b(  R ) \} \cap [-1/2-\eps_0, -1/2+ \eps_0] \subseteq \{ -1/2 \}.
\end{equation*}
Let $u \in s^{-1/2}L^2_{\iie}(Z \times T_qY^+;\Iie\Lambda^*)$ satisfy $N_q(P_0)(u) \in s^{1/2}L^2_{\iie}(Z \times T_qY^+;\Iie\Lambda^*)$; 
by Proposition \ref{prop:upshot}, we only need check that $u \in s^{-1/2+\eps_0}L^2_{\iie}(Z \times T_qY^+;\Iie\Lambda^*)$. 
Fix such a $u$, and write $L^2_{\iie}(Z_ \times T_qY^+;\Iie\Lambda^*)$ in place of $L^2(q)$.

Applying Lemmas \ref{lem:Bessel} and \ref{lem:NClosedRange}, we know that $R=N_q(P_0)$ is injective and has closed range 
as a map from $s^{-1/2 + \eps_0}L^2(q)$ to itself (on its maximal domain). It follows that $R^t$ is surjective from 
$s^{-1/2-\eps_0}L^2(q)$ to itself (on its minimal domain) . 

Let $G$ be the bounded generalized inverse of $R^t$; $G$ is a bounded map from $s^{-1/2-\eps_0}L^2(q)$ to itself, with 
image contained in the domain of $R^t$, and satisfies
\begin{equation*}
	R^t G = \Id_{s^{-1/2-\eps_0}L^2(q)}.
\end{equation*}

Let $\phi$ be any element of $s^{-1/2+\eps_0}H_{\pie}^{1}(Z \times T_qY^+;\Iie\Lambda^* )$. Then $v = G\phi$ satisfies
\begin{equation*}
	v \in s^{-1/2-\eps_0} L^2(a), \quad
	R^t v = R^t G\phi = \phi  \in s^{-1/2+\eps_0} L^2(q),
\end{equation*} 
the latter statement and elliptic regularity allows us to strengthen the former to $v \in s^{-1/2-\eps_0} H_{\pie}^{2}(Z \times T_qY^+;\Iie\Lambda^* )$.
On the other hand, we know that $R u \in s^{1/2} L^2(q)\subset s^{-1/2+\eps_0} L^2(q)$, 
so by part 2) of Lemma \ref{lem:integration} (with $\alpha=-1/2$) we conclude that
\begin{equation*}
	\ang{R u,v} = \ang{R^t v, u}.
\end{equation*}

But then we also have
\begin{equation} \label{eq:1}
	\ang{R u, v} = \ang{ R u, G \phi} = \ang{G^t R u, \phi}
	\end{equation} 
	where we recall that $R u\in s^{1/2}L^2(q)  \subset s^{-1/2+\eps_0}L^2(q)$, $G\phi\in s^{-1/2-\eps_0}L^2(q)$
and where $G^t$ denotes the functional analytic transpose of the {\it bounded} operator $G$;
$G^t$ acts continuously on $s^{-1/2+\eps_0}L^2(q)$, so in fact $G^t R u \in s^{-1/2+\eps_0}L^2(q).$ 

Moreover, we have:
\begin{equation} \label{eq:2}
\ang{R u, v} =\ang{R^t v, u} = \ang{ R^t G \phi, u} = \ang{ \phi , u}.
\end{equation}	
	By comparing the last terms of  
	\eqref{eq:1}  and  \eqref{eq:2} we see that $\ang{u-G^t R u, \phi}=0$ and since $\phi$
	was arbitrary we finally get: $u= G^t R u.$
Therefore $u\in s^{-1/2+\eps_0}L^2(q)$.

Next, taking $\eps_1 \in (0,1)$ small enough that
\begin{equation*}
	\{ \Re \zeta + \tfrac f2 + \tfrac12: \zeta \in \spec_b(  R ) \} \cap [-1/2+\eps_0-\eps_1, -1/2+\eps_0+ \eps_1] = \emptyset \subseteq \{ -1/2+\eps_0 \},
\end{equation*}
we can repeat the argument above and conclude $u \in s^{-1/2 + \eps_0 + \eps_1}L^2(q)$; continuing in this way we conclude that
$u \in s^{-1/2 + \eps}L^2(q)$ for any $\eps \in (0,1)$ as required. \end{proof}

\noindent {\it Proof of Proposition \ref{DomVanBdy}.}
1) Proceed by induction on depth.  For depth zero, there is nothing to prove. Let $k > 0$ and assume that the result is true for 
any Witt space of depth less than $k$.  If $u \in \cD_{\max}(\eth_{\dR})$ has support in a locally trivialized neighbourhood
$\calU \times C(Z)$ at the highest depth stratum, then Proposition \ref{prop:intermediate} gives the stated decay and regularity 
in the final radial variable. Since the link $Z$ has depth $k-1$, we already know the result for it.


2) This follows since $\rho^{\eps}H^1_{\iie}(X;\Iie\Lambda^*X)$ is compactly embedded in $L^2_{\iie}$
\end{proof}

We now know that elements of the maximal domain have some `extra' degree of vanishing, and we can then 
apply an argument of Gil-Mendoza \cite{Gil-Mendoza}.

\begin{proposition}[Gil-Mendoza] \label{GilMendoza}
If $\cD_{\max}(\eth_{\dR}) \subseteq \rho^C L^2_{\iie}(X;\Iie\Lambda^*X)$ for some $C >0$, then, as an operator on $L^2_{\iie}(X;\Iie\Lambda^*X)$,
\begin{equation*}
	\cD_{\max}(\eth_{\dR}) \cap \bigcap_{\eps>0} \rho^{1-\eps} L^2_{\iie}(X; \Iie\Lambda^*X) \subseteq \cD_{\min}(\eth_{\dR})
\end{equation*}
\end{proposition}

{\bf Remark.}
Since we have actually shown not only that 
\begin{equation*}
	\cD_{\max}(\eth_{\dR}) \subseteq \rho^C H^1_{\iie}(X;\Iie\Lambda^*X)
\end{equation*}
but in fact
\begin{equation*}
	\cD_{\max}(\eth_{\dR}) \subseteq \bigcap_{\eps>0} \rho^{1-\eps} H^1_{\iie}(X;\Iie\Lambda^*X),
\end{equation*}
this proposition implies $\cD_{\max}(\eth_{\dR}) = \cD_{\min}(\eth_{\dR})$.

\begin{proof}
We point out the following simple consequence of the formal self-adjointness of $\eth_{\dR}$ and the definitions of the minimal/maximal domains and weak derivatives:
\begin{lemma}
An element $u \in \cD_{\max}(\eth_{\dR})$ is in $\cD_{\min}(\eth_{\dR})$ if and only if
\begin{equation}\label{MinDomChar}
	(\eth_{\dR}u, v) = (u, \eth_{\dR}v), \Mforevery v \in \cD_{\max}(\eth_{\dR}).
\end{equation}
\end{lemma}
\begin{proof}
 For any operator $D$ with formal adjoint $D^*$ one has, 
\begin{gather*}
	u \in \cD(D_{\min}) \iff
	u \in \cD\lrpar{ \lrpar{ (D^*)_{\max} }^* } \\
	\iff
	\ang{Du, v} = \ang{u, D^*v} \text{ for every } v \in \cD( (D^*)_{\max} )
\end{gather*}
If $D$ is symmetric so that $D^* = D$, then this is \eqref{MinDomChar}. 
\end{proof} 
Let $u \in \cD_{\max}(\eth_{\dR}) \cap \bigcap_{\eps>0} \rho^{1-\eps} L^2_{\iie}(X; \Iie\Lambda^*X),$ so 
\begin{equation*}
	u \in  \bigcap_{\eps>0} \rho^{1-\eps} H^1_{\iie}(X; \Iie\Lambda^*X).
\end{equation*}
Set $u_n = \rho^{1/n}u$ for $n \in \bbN$, so that for each $n,$ 
$u_n \subseteq \rho H_{\iie}^1(X;\Iie\Lambda^*),$ and, for every $\eps \in (0,1)$,
\begin{equation}\label{FirstConv}
	u_n \to u \text{ in } \rho^{1-\eps}H_{\iie}^1(X;\Iie\Lambda^*)
	\Mand
	\eth_{\dR} u_n \to \eth_{\dR} u \text{ in } \rho^{-\eps} L^2_{\iie}(X;\Iie\Lambda^*).
\end{equation}
Let $\eps \in (0,1)$  so that $\cD_{\max}(\eth_{\dR}) \subseteq \rho^\eps H_{iie}^1(X;\Iie\Lambda^*)$. 
Then, for any  $v \in \cD_{\max}(\eth_{\dR})$, \eqref{FirstConv} implies
\begin{gather*}
	(\eth_{\dR}u_n, v)_{L^2} 
	= (\rho^\eps \eth_{\dR} u_n, \rho^{-\eps}v)_{L^2}
	\to (\rho^\eps \eth_{\dR} u, \rho^{-\eps}v)_{L^2}
	= (\eth_{\dR}u, v)_{L^2}, \\
	\Mand 
	(u_n, \eth_{\dR}v) \to (u, \eth_{\dR}v).
\end{gather*}
Moreover, by the previous Lemma,  $u_n \in \cD_{\min}(\eth_{\dR})$ implies $(\eth_{\dR}u_n, v) = (u_n, \eth_{\dR}v)$.

It follows that $(\eth_{\dR}u, v) = (u, \eth_{\dR}v)$ for every $v \in \cD_{\max}(\eth_{\dR})$ and hence $u \in \cD_{\min}(\eth_{\dR})$.
\end{proof}

Altogether, we have now proved Theorem~\ref{MT1}. We summarize for the benefit of the reader.


\begin{proof}  Parts 1) and 2) are direct consequences of the last Proposition. Let us show that $\eth_{\dR}$ is self-adjoint 
on its maximal domain. Denote by  $\eth_{\dR,{\rm max}}$ the operator $\eth_{\dR}$ on its maximal domain. 
If $v$ is in the domain of  $\eth_{\dR,{\rm max}}$ then integration by parts, which is allowed because of the extra vanishing, implies 
that $v$ is in the domain of  $(\eth_{\dR,{\rm max}})^*$ and that  $\eth_{\dR,{\rm max}}v=   (\eth_{\dR,{\rm max}})^* v$.
Conversely, let $v$ lie in the domain of  $(\eth_{\dR,{\rm max}})^*$. Observe that $\forall u\in C^\infty_c$, 
$\langle \eth_{\dR} u,v \rangle = \langle u,\eth_{\dR} v \rangle$, with $\eth_{\dR}$ acting as a distribution on $v$. 
From the definition of adjointness we also know that $\langle \eth_{\dR} u,v \rangle = \langle u, (\eth_{\dR,{\rm max}})^*v \rangle$ 
and since this is true for all $u\in C^\infty_c$ we infer that $\eth_{\dR} v$ is in $L^2_{\iie}$. Indeed, by definition,
$(\eth_{\dR,{\rm max}})^*v\in L^2_{\iie}$. Thus $v$ is in the domain of $\eth_{\dR,{\rm max}}$ and $\eth_{\dR,{\rm max}} v= (\eth_{\dR,{\rm max}})^*v$. 
This proves that $\eth_{\dR,{\rm max}}$ is self-adjoint. 

To prove 3), since $\eth_{\dR}$ is self-adjoint, $(i \,{\rm Id} + \eth_{\dR})$ is invertible. Since 
$\cD_{\max}(\eth_{\dR})$ is compactly  embedded into $L^2_{\iie}(X;\Iie\Lambda^*X)$, $(i\, {\rm Id} + \eth_{\dR})^{-1}$  
defines a parametrix for $\eth_{\dR}$ acting on $\cD_{\max}(\eth_{\dR})$  with compact reminder.

Finally, for 4), since $\eth_{\dR}$ is Fredholm, there exists $\epsilon >0$ such that 
$(\epsilon\,{\rm Id} + \eth_{\dR})$ is invertible. Since the maximal domain is compactly embedded in 
$L^2_{\iie}$, $(\epsilon\,{\rm Id} + \eth_{\dR})^{-1}$  is compact and self-adjoint. Thus, 
the spectrum of $(\epsilon\,{\rm Id} + \eth_{\dR})^{-1}$ is discrete with finite multiplicity. Therefore, the spectrum 
of $\eth_{\dR}$ is discrete and has finite multiplicity.
\end{proof}

\section{The signature operator on Witt spaces} \label{sec:Sign} $\;$

We now turn from the de Rham operator to the signature operator, first on forms with scalar 
coefficients and then with $C^*$-algebra coefficients. We show first that these are Fredholm 
operators, but more importantly, that they define classes in the groups $K_*(\hat X)$ 
and $K_*(C^*\Gamma)$, respectively. The index of these operators is independent of the choice of
metric and defines a topological invariant.  We will show later that this class enjoys even stronger
properties: it is 
a Witt bordism invariant, a stratified homotopy invariant and  it is equal, rationally,
to a topologically defined invariant, the symmetric signature. 

\subsection{The signature operator $\eth_{\sign}$}  $ $\newline
If $X$ is even-dimensional, the Hodge star induces a natural involution on the differential forms on $X$,
\begin{equation*}
	\cI: \Omega^*(X) \to \Omega^*(X), \quad \cI^2 = \Id
\end{equation*}
whose $+1$, $-1$ eigenspaces are known as the set of self-dual, respectively anti-self dual, forms and are denoted $\Omega^*_+$, $\Omega^*_-$.
The involution $\cI$ extends naturally to $\Iie \Omega^*(X)$ and with respect to the splitting 
$\Iie\Omega^*(X) = \Iie \Omega^*_+ \oplus \Iie \Omega^*_-$, the de Rham operator decomposes
$$
\eth_{\dR} 
	= \begin{pmatrix}
	0&  \eth_{\sign}^-  \\
	 \eth_{\sign}^+  & 0
	\end{pmatrix}
$$
where
$$
\eth_{\sign}^+=d+\delta: \Iie \Omega^*_+(X) \to \Iie\Omega^*_-(X),\;  \eth_{\sign}^-=( \eth_{\sign}^+)^*.
$$
If instead the manifold $X$ is odd-dimensional, the signature operator of an (adapted) edge iterated  metric is
\begin{equation*}
	\eth_{\sign} = -i(d \cI + \cI d)
	= -i\cI (d - \delta) = -i (d - \delta) \cI.
\end{equation*}
We point out for later use that in either case, given a continuous map $r: \hat X \to B\Gamma$, 
we also obtain a twisted Mishchenko-Fomenko 
signature operator $ \wt \eth_{\sign}$ acting on sections of the bundle $\Iie\Lambda_\Gamma^*(X)$.

\begin{theorem}  \label{thm:Fred}
Up to rescaling suitably the metric the following is true. If $\hat{X}$ satisfies \eqref{Ass2} for all strata, then 
the  iterated  incomplete edge signature operator $\eth_{\sign}$ is essentially self-adjoint  with maximal domain contained in 
$$
\bigcap_{\eps>0} \rho^{1-\eps} H^1_{\iie}(X;\Iie\Lambda^*X).
$$ 
Its  unique self-adjoint extension is Fredholm on its maximal domain endowed with the 
graph-norm; moreover it has discrete $L^2$-spectrum of finite multiplicity.
\end{theorem}
  
\begin{proof}
If $X$ is even-dimensional, it is immediate to see that 
\begin{gather*}
	\cD_{\min}(\eth_{\sign}^+) = \cD_{\min}(\eth_{\dR}) \cap L^2_{\iie}(X;\Iie \Lambda^*_+(X)), \\
	\cD_{\max}(\eth_{\sign}^+) = \cD_{\max}(\eth_{\dR}) \cap L^2_{\iie}(X;\Iie \Lambda^*_+(X))
\end{gather*}
so the result follows from the corresponding results for $\eth_{\dR}$.

For $X$ odd-dimensional, we point out that one can characterize the maximal domain of $d-\delta$ through the same analysis used for $d+\delta$.
Alternately, we can use the result for $d+\delta$ to deduce it for $d-\delta$ as follows.
As explained above, a byproduct of our results is the existence of a strong Kodaira decomposition
\begin{equation*}
	L^2_{\Iie} \Omega^* = 
	L^2\cH \oplus \Image d \oplus \Image \delta
\end{equation*}
where $L^2\cH$ is the intersection of the null spaces of $d$ and $\delta$.
The de Rham operator $d+\delta$ decomposes into
\begin{equation*}
	\lrpar{ d: \Image \delta \to \Image d } \oplus
	\lrpar{ \delta: \Image d \to \Image \delta },
\end{equation*}
hence $d$ and $\delta$ individually have closed range and 
\begin{gather*}
	\cD_{\max}(\eth_{\dR}) \cap \Image \delta = \cD_{\max}(d) \cap \Image \delta \\
	\cD_{\max}(\eth_{\dR}) \cap \Image d = \cD_{\max}(\delta) \cap \Image d 
\end{gather*}
hence $i(d-\delta)$ has closed range with domain contained in (hence, by symmetry, equal to) $\cD_{\max}(\eth_{\dR})$.
Applying Proposition \ref{GilMendoza} to $i(d-\delta)$ then shows that it too is essentially self-adjoint.
\end{proof}

\subsection{The K-homology class $ [\eth_{\sign}]\in K_* (\hat{X})$} $ $\newline
The results proved so far for the signature operator $\eth_{\sign}$ on a Witt space $\widehat{X}$ allow one to define the 
K-homology class  $[\eth_{\sign}]\in K_* (\widehat{X})=KK_* (C(\widehat{X}),\bbC)$. The K-homology signature class
already appears in the work of Moscovici-Wu \cite{Mosc-Wu-Witt}; the definition there is based
on the results of Cheeger.

Recall that an {\it even} unbounded Fredholm module for the $C^*$-algebra $C(\widehat{X})$ is a pair $(H,D)$
such that:
\begin{itemize}
\item $H$ is a Hilbert space endowed with a unitary $*$-representation of $C(\widehat{X})$; $D$ is a self-adjoint unbounded linear operator on $H$;
\item there is a dense $*$-subalgebra $\mathcal A\subset C(\widehat{X})$ such that $\forall a \in \mathcal A$ the domain
of $D$ is invariant by $a$ and $[D,a]$ extends to a bounded operator on $H$;
\item $(1+D^2)^{-1}$ is a compact operator on $H$;
\item $H$ is equipped with a grading $\tau=\tau^*$, $\tau^2=I$, such that $\tau f = f \tau$ and $\tau D= -D \tau$.
\end{itemize}
An {\it odd} unbounded Fredholm module is defined omitting the last condition.

An unbounded Fredhom module defines a Kasparov $(C(\widehat{X}),\bbC)$-bimodule and thus
an element in  $KK_* (C(\widehat{X}),\bbC)$. We refer to  \cite{Baaj-Julg} \cite{Blackadar} for more
on this foundational material. 

Recall that adapted edge iterated metrics were defined in Proposition \ref{prop:rescaling}.
The following Theorem already appears in \cite{Mosc-Wu-Witt}, where it is proved using
Cheeger's results. Here we give a proof using our approach.

\begin{theorem}\label{theo:k-homology}
The signature operator $\eth_{\sign}$ associated to a Witt space 
$\widehat{X}$ endowed with an adapted edge iterated  metric $g$ defines an unbounded Fredholm module
for $C(\widehat{X})$ and thus a class $[\eth_{\sign}]\in KK_* (C(\widehat{X}),\bbC)$, $* \equiv\dim X \,{\rm mod} \,2$.
Moreover, the class  $[\eth_{\sign}]$ does not depend on the choice of the adapted edge iterated metric on $\widehat{X}$.
\end{theorem}

\begin{proof}
We take $H=L^2_{\iie}(X;\Iie\Lambda^*X)$, endowed with the natural representation of $C(\widehat{X})$
by multiplication operators. We take $D$ as the unique closed self-adjoint extension of $\eth_{\sign}$. 
These data depend of course on the choice of the adapted edge iterated  metric. We take
$\mathcal{A}$ equal to the space of Lipschitz functions on $\widehat{X}$ with respect to $g$;
$\mathcal{A}$ does not depend on the choice of $g$, since two adapted edge iterated  metrics are quasi-isometric.
Finally, in the even dimensional case we take the involution defined by $\mathcal{I}$.
All the conditions defining an unbounded Kasparov module are easily proved using the results
of the previous section: indeed, if $f$ is Lipschitz then it is elementary to see that 
multiplication by $f$ sends the maximal domain of $\eth_{\sign}$
into itself; moreover $[f,\eth_{\sign}]$ is Clifford multiplication by $df$ which exists almost everywhere and is an element
in $L^\infty (\widehat{X})$; in particular $[f,\eth_{\sign}]$ 
extends to a bounded operator on $H$; finally we know that $(1+D^2)^{-1}$ is a compact operator (indeed, we proved 
this is true for $(\pm i+D)^{-1}$).
Thus there is a well defined class $ KK_* (C(\widehat{X}),\bbC)$ which we denote simply by $[\eth_{\sign}]$;
this class depends a priori on the choice of the metric $g$. 
Recall however  that two adapted edge iterated  metric $g_0$ and $g_1$ are joined by a path of adapted edge iterated  metrics
$g_t$.
Let $\eth^0_{\sign}$ and $\eth^1_{\sign}$ the corresponding signature operators, with domains in $H_0$
and $H_1$.
Proceeding as in the work of Hilsum on Lipschitz manifolds \cite{Hilsum-LNM} one can prove that the 1-parameter family
$(H_t, \eth^t_{\sign})$ defines an {\it unbounded} operatorial  homotopy; using the homotopy invariance of
$KK$-theory one obtains
  $$[\eth^0_{\sign}]=[\eth^1_{\sign}] \;\,\text{ in }\;\;  KK_* (C(\widehat{X}),\bbC)\,.$$
  We omit the details since they are a repetition of the ones given in  \cite{Hilsum-LNM}. 
\end{proof}

\subsection{The index class of  the twisted signature operator $ \wt \eth_{\sign}$} $ $\newline
Let $\widehat{X}$ be a Witt space endowed with an adapted edge iterated  metric. 
Assume now that we are also given a continuous map $r:\widehat{X}\to B\Gamma$ and let
$\Gamma \to \widehat{X}^\prime \to \widehat{X}$ the Galois $\Gamma$-cover induced by $E\Gamma\to B\Gamma$. 
We consider the Mishchenko bundle
\begin{equation*}
	\wt{C^*_r}\Gamma: = C^*_r\Gamma \btimes_\Gamma \widehat{X}^\prime.
\end{equation*} 
and the signature operator with values in the restriction of $\wt{C^*_r}\Gamma$ to $X$, which we denote
by $ \wt \eth_{\sign}$. 

\begin{proposition}  \label{prop:esa-gamma}
The twisted signature operator $ \wt \eth_{\sign}$ is essentially self-adjoint as an operator on 
$L^2_{\iie, \Gamma}(X;\Iie\Lambda_\Gamma^*X)$, with maximal domain contained 
in \\$\cap_{\eps>0} \rho^{1-\eps} H^1_{\iie,\Gamma}(X;\Iie\Lambda_\Gamma^*X)$ which is in turn 
$C^*_r\Gamma$-compactly 
included in the Hilbert $C^*_r\Gamma$-module $L^2_{\iie, \Gamma}(X;\Iie\Lambda_\Gamma^*X)$. 
\end{proposition}

\begin{proof}
We briefly point out how the proof given for $ \eth_{\dR}$  and $\eth_{\sign}$ extends to the case of  $\wt \eth_{\dR}$
and $\wt \eth_{\sign}$.
Recall that a $C^*_r \Gamma-$distribution on $X= \mbox{reg}\,(\widehat{X})$ is a $\mathbb{C}-$linear form 
$$
T: C^\infty_c ( \mbox{reg}\,(\widehat{X}), \Iie\Lambda_\Gamma^*X) \rightarrow C^*_r \Gamma
$$ satisfying the following 
property. For any compact $K \subset \mbox{reg}\,(\widehat{X}),$ there exists a finite set $S$
of elements of ${\rm Diff}^*_{ie, \Gamma}$ such that:
$$
\forall u \in C^\infty_K ( \mbox{reg}\,(\widehat{X}), \Iie\Lambda_\Gamma^*X), \;
\lVert \langle  T ; u \rangle \rVert_{C^*_r \Gamma} \leq \sup_{  Q \in S}  \lVert (Q u) \rVert_{L^2_{\iie,\Gamma}}.
$$ Of course, any element of  $L^2_{\iie, \Gamma}(X;\Iie\Lambda_\Gamma^*X)$ defines
a $C^*_r \Gamma-$distribution on $\mbox{reg}\,(\widehat{X}).$
It is clear that $\wt  \eth_{\dR}$ sends $L^2_{\iie, \Gamma}(X;\Iie\Lambda_\Gamma^*X)$ into
the space of $C^*_r \Gamma-$distributions. Therefore, the notion of maximal domain 
for $\wt  \eth_{\dR}$ is  defined. The notion of minimal domain is also well defined (this is simply the closure
of $C^\infty_c$ with respect to the norm $\lVert u \rVert + \lVert \wt  \eth_{\dR} u \rVert$ ). Notice that these two extensions
are closed.
Our first task is to show that these two extensions coincide. 
To this end we shall make use of the fundamental hypothesis that the reference map $r: X \rightarrow B \Gamma$ extends continously 
to the whole singular  space $\widehat X$. Therefore, for any distinguished neighborhood $W \simeq \bbR^b \times C(Z)$, 
the induced  $\Gamma$-coverings over  $W$ and  over $Z$ are trivial. 
This implies  that for any $q \in Y,$ 
$N_q(\wt  \eth_{\dR} ) $ is conjugate to  $N_q(\eth_{\dR}) \otimes {\rm Id}_{\wt C^*_r \Gamma}.$
Once this has been observed we have, immediately, that  Proposition \ref{prop:upshot} and Lemma \ref{lem:integration} extend 
to the case of $\wt\eth _{\dR}.$  Then, Proposition \ref{prop:intermediate} also extends easily to the present case 
 showing that the maximal domain of  $ \wt \eth_{\sign}$ is included in
$\cap_{\eps>0} \rho^{1-\eps} H^1_{\iie,\Gamma}(X;\Iie\Lambda_\Gamma^*X)$. 
Once the extra vanishing is obtained, we can apply the argument give in the proof of
Theorem \ref{MT1} in order to show that the maximal extension is in fact self-adjoint.
The argument of Gil-Mendoza can also be extended, showing the equality of the maximal and the minimal domain. 
The details of all this are easy and for the sake of brevity we omit them. Finally, proceeding as in \cite{LP-SMF},
one can prove that $\rho^\eps H^1_{\iie,\Gamma}(X;\Iie\Lambda_\Gamma^*X)$ is $C^*_r \Gamma$-compactly included into 
$ L^2_{\iie,\Gamma}(X;\Iie\Lambda_\Gamma^*X)$. The Proposition is proved for $\wt \eth_{\dR}$. The extra step needed
for the signature operator is proved as in Theorem \ref{thm:Fred}.
\end{proof}

From now on we shall only consider the closed unbounded self-adjoint  $C^*_r\Gamma$-operator  of Proposition \ref{prop:esa-gamma}
and with common abuse of notation we continue to denote it by $\wt \eth_{\sign}.$

We now proceed to show the following fundamental 

\begin{proposition}\label{prop:regularity}
The operator   $ \wt \eth_{\sign}$  is a {\it regular} operator. Consequently 
 $(i\pm \wt \eth_{\sign})$ and  $(1+\wt \eth_{\sign}^2)$ are  invertible.
\end{proposition}

\begin{proof}
Recall that a closed unbounded self-adjoint operator $D$
on a Hilbert $C^*_r\Gamma$-module is
said to be regular if $1+D^2$ is surjective. One can show, see \cite{Lance}, that $D$ is regular if and only if
$1+D^2$ has dense image if and only if $(i\pm D)$ has dense image if and only if $(i\pm D)$ is surjective.
Moreover, if $D$ is regular then both  $(i\pm D)$  and $1+D^2$ have an inverse.
For a simple example of an unbounded self-ajoint operator on a Hilbert module  such that  $(i+ D)$ and  $(i- D)$ are
not invertible see \cite[page 415]{Hilsum-K}.

We shall prove that our operator is regular by employing unpublished ideas of George Skandalis, explained in detail
in work of Rosenberg-Weinberger \cite{ros-weinberger}. We have seen in the previous subsection that $\eth_{\sign}$ defines an
unbounded Kasparov $(C(\widehat{X}),\bbC)$-bimodule and thus a  class
$[\eth_{\sign}]\in KK_* (C(\widehat{X}),\bbC)$. Consider now 
$$
\mathcal{E}:= L^2_{\iie}(X;\Iie\Lambda^*X)\otimes_{\bbC} C^*_r \Gamma \,;
$$
tensoring $\eth_{\sign}$ with $\Id_{C^*_r\Gamma}$ we obtain in an obvious way an unbounded 
Kasparov $(C(\widehat{X})\otimes C^*_r\Gamma, C^*_r\Gamma)$-bimodule that we will denote  by $(\mathcal{E},\mathcal{D})$.
For later use we denote the corresponding KK-class as 
\begin{equation}\label{tensor-class}
[[ \eth_{\sign} ]]\in KK_*(C(\widehat{X})\otimes C^*_r\Gamma, C^*_r\Gamma).
\end{equation}
Consider $A:= C(\widehat{X})\otimes C^*_r\Gamma$ and set
$$
\mathcal{A}:=\{a\in A : a(\Dom \mathcal{D})\subset \Dom \mathcal{D} \text{ and } [a,\mathcal{D}] \ 
\text{extends to an element of }\; \mathcal{L}( \mathcal{E}) \}.
$$
It is a non-trivial result (\cite{ros-weinberger}) that $\mathcal{A}$ is a dense *-subalgebra of $A$ stable under holomorphic functional calculus.
Consider now the Mishchenko bundle $\wt{C^*_r}\Gamma$ and its continuous sections
$C^0 ( \widehat{X}; \wt{C^*_r}\Gamma)=:P$. It is obvious that $P$ is a finitely generated projective right $A$-module. 
The result cited above, together with Karoubi density theorem (\cite{Karoubi} exercice II.6.5), implies that there exists a finitely generated projective 
right $\mathcal{A}$-module $\mathcal{P}$ such that $P=\mathcal{P}\otimes_\mathcal{A} A$. Consider for $\xi\in P$ the operator 
$T_\xi: \mathcal{E} \to P\hat{\otimes}_A \mathcal{E}$ defined by $T_\xi (\eta):= \xi\otimes \eta$. $T_\xi$ is a bounded
operator of $C^*_r\Gamma$  Hilbert modules with adjoint $T_\xi^*$. Recall now, 
see \cite{ros-weinberger},  that a 
$\mathcal{D}$-connection in the present context is a  symmetric  $C^*_r\Gamma$-linear operator $\widetilde{\mathcal{D}}$
$$\widetilde{\mathcal{D}}: \mathcal{P}\otimes_\mathcal{A} \Dom (\mathcal{D}) \longrightarrow P\hat{\otimes}_A \mathcal{E}$$
such that $\forall \xi\in \mathcal{P}$ the following commutator, defined initially on $(\Dom (\mathcal{D}) )\oplus 
\mathcal{P}\otimes_\mathcal{A} \Dom (\mathcal{D}) $,  extends to a bounded operator on $ \mathcal{E}\oplus 
P\hat{\otimes}_A \mathcal{E}$: 
$$\left[
 \begin{pmatrix}
	\mathcal{D} &  0 \\
	0  & \widetilde{\mathcal{D}}
	\end{pmatrix},\begin{pmatrix}
	0 &  T_\xi^* \\
	T_\xi & 0
	\end{pmatrix} \right]
$$
Rosenberg and Weinberger have proved \cite{ros-weinberger} that every $\mathcal{D}$-connection is a 
self-adjoint regular operator. We can end the proof of the present Proposition as follows:
first we  observe that as  $C^*_r\Gamma$ Hilbert
modules  $P\hat{\otimes}_A \mathcal{E}= L^2_{\iie,\Gamma}(X;\Iie\Lambda_\Gamma^*X)$; next we consider 
$\wt \eth_{\sign}$ and prove the following.
\begin{lemma} The operator $\wt \eth_{\sign}$ defines a $\mathcal{D}$-connection. 
\end{lemma}
\begin{proof} 
It will suffice to prove the following. Let $U$ be an open
subset of ${X}$ over which $\widetilde{C^*_r\Gamma}$ is trivial. Then the restriction 
of $\xi \in \mathcal{P}$ to $U$ is a finite sum of terms of the form $\theta \otimes u$ where 
$\theta$ is a flat section and $u$ is a $C^1-$function. So we shall assume that 
$\xi= \theta \otimes u.$
Then for any $\eta \in L^2_{\iie}(U;\Iie\Lambda^*X_{| U}) \otimes C^*_r\Gamma $, one has:
$$
( \wt \eth_{\sign} \circ  T_\xi -T_\xi \circ \mathcal{D}) (\eta)=
\theta \otimes c(d u) \eta + \theta  \otimes u (\wt \eth_{\sign} - \mathcal{D})(\eta),
$$  where $\cl{d u}$ is Clifford multiplication. Recall that the restrictions to $U$ of $\wt \eth_{\sign}$ and   $\mathcal{D}$ 
are differential operators of order one having the same principal symbol. 
Therefore, $( \wt \eth_{\sign} \circ  T_\xi -T_\xi \circ \mathcal{D}) $ is bounded 
on $ L^2_{\iie}(U;\Iie\Lambda^*X_{| U}) \otimes C^*_r\Gamma. $
One then gets immediately the Lemma by using a partition of unity.
\end{proof}
Finally, we check easily that 
$\mathcal{P}\otimes_\mathcal{A} \Dom (\mathcal{D})\subset \Dom_{{\rm max}} (\wt \eth_{\sign})$. Since 
$(i+\wt \eth_{\sign})$ has dense image  with domain  $\mathcal{P}\otimes_\mathcal{A} \Dom (\mathcal{D})$,
we see that, a fortiori,
the image of $(i+\wt \eth_{\sign})$ with domain $ \Dom_{{\rm max}} (\wt \eth_{\sign})$ must also be dense.
\end{proof}

These two Propositions yield at once the following 

\begin{theorem}\label{theo:kk}
The twisted signature operator $\wt \eth_{\sign}$ and the $C^*_r\Gamma$-Hilbert
module $L^2_{\iie,\Gamma}(X;\Iie\Lambda_\Gamma^*X)$ define an unbounded 
Kasparov $(\bbC,C^*_r\Gamma)$-bimodule and thus a class in $KK_* (\bbC, C^*_r \Gamma)=K_* (C^*_r\Gamma)$.
We call this the index class associated to $\wt \eth_{\sign}$ and denote it by $\Ind (\wt \eth_{\sign})\in K_* (C^*_r\Gamma)$.\\
Moreover, if  as in \eqref{tensor-class} we denote by $[[\eth_{\sign}]]\in KK_*(C(\widehat{X})\otimes C^*_r\Gamma, C^*_r\Gamma)$ 
the class obtained from $[\eth_{\sign}]\in KK_*(C(\widehat{X}),\bbC)$ by tensoring with  $C^*_r\Gamma$, then
$\Ind (\wt \eth_{\sign})$ is equal to the Kasparov product of the class defined by Mishchenko bundle 
$[\widetilde{C^*_r}\Gamma]\in  KK_0(\bbC,C(\widehat{X})\otimes C^*_r\Gamma)$ with  $[[\eth_{\sign}]]$:
\begin{equation}\label{tensor}
\Ind (\wt \eth_{\sign})= [\widetilde{C^*_r}\Gamma]\otimes [[\eth_{\sign}]]
\end{equation}
In particular, the index class $\Ind (\wt \eth_{\sign})$ does not depend on the choice of the adapted edge iterated  metric.
\end{theorem}

\begin{proof}
We already know that $\wt \eth_{\sign}$ is self-adjoint regular and $\bbZ_2$-graded in the even dimensional case.
It remains to show that the inverse of $(1+\wt \eth_{\sign}^2)$ is a $C^*_r \Gamma$-compact operator.
However, the domain of $\wt \eth_{\sign}$ is compactly included in  $L^2_{\iie,\Gamma}(X;\Iie\Lambda_\Gamma^*X)$;
thus $(i+ \wt \eth_{\sign})^{-1}$ and $(-i+\wt \eth_{\sign})^{-1}$ are both compacts. It follows that  $(1+\wt \eth_{\sign}^2)^{-1}$ is
compact. Thus  $(\wt \eth_{\sign}, L^2_{\iie,\Gamma}(X;\Iie\Lambda_\Gamma^*X))$ defines an unbounded 
Kasparov $(\bbC,C^*_r\Gamma)$-bimodule as required. The equality $\Ind (\wt \eth_{\sign})= [\widetilde{C^*_r}\Gamma]\otimes [[\eth_{\sign}]]$
is in fact part of the theorem, attributed to  Skandalis in \cite{ros-weinberger}, on $\mathcal{D}$-connections. Finally, since we have proved that 
$ [\eth_{\sign}]$, and thus $ [[\eth_{\sign}]]$, is metric independent, and since $ [\widetilde{C^*_r}\Gamma]$
is obviously metric independent, we conclude that $\Ind (\wt \eth_{\sign})$ has this property too. The Theorem is proved.
\end{proof}

\begin{corollary}\label{cor:assembly}
Let $\beta:K_* (B\Gamma)\to K_* (C^*_r\Gamma)$ be the assembly map; let $r_* [\eth_{\sign}]\in K_* (B\Gamma)$
the push-forward of the signature K-homology class. Then 
\begin{equation}\label{assembly}
\beta(r_* [\eth_{\sign}])=\Ind (\wt \eth_{\sign}) \text{ in }  K_* (C^*_r\Gamma)
\end{equation}
\end{corollary}

\begin{proof}
Since $\Ind (\wt \eth_{\sign})= [\widetilde{C^*_r}\Gamma]\otimes [[\eth_{\sign}]]$, this follows immediately
from the very definition of the assembly map. See  \cite{Kasparov-inventiones} 
\cite{Kasparov-contemporary}.
\end{proof}

\section{Witt bordism invariance}

Let $\hat Y$ be an oriented odd dimensional Witt space  with boundary $\partial \hat Y= \hat X.$  
We assume that $\hat  Y$ is a smoothly stratified space having a product structure 
near its boundary.
We endow $\hat Y$ with an edge iterated metric having a product structure near $\partial \hat Y= \hat X$  
and inducing an adapted edge iterated metric metric $g$ (Proposition \ref{prop:rescaling}) on $ \hat X.$ Consider 
a reference map $r: \hat Y \rightarrow B \Gamma,$ its restriction to $\hat X$ and $g$ induce a  $C^*_r \Gamma-$linear signature operator on 
$\hat X$. In this section only we shall be very precise and denote this operator by $\wt \eth_{\sign}( \hat X)$.

\begin{theorem} \label{thm:bordism} We have $ \Ind \wt \eth_{\sign} (\hat X) =0 $ in $K_0 (C^*_r \Gamma) \otimes_\bbZ \bbQ.$
\end{theorem}
\begin{proof} We follow \cite[Section 4.3]{LP-Cut} and Higson \cite[Theorem 5.1]{Higson}. 
Denote by ${\hat Y}^\prime \rightarrow \hat Y$ and ${ \hat X}^\prime \rightarrow \hat X$ the two $\Gamma-$coverings
associated to the reference map $r: \hat Y \rightarrow B \Gamma.$

The analysis of  Section 7
shows that the operator $\wt \eth_{\sign}( \hat X) $ induces a class $[\wt \eth_{\sign}( \hat X) ]$ in the  
Kasparov group $KK^0 (C_0( \partial {\hat Y}), C_r^* \Gamma).$ 
In terms of the constant map $\pi^{\partial {\hat Y}}: \partial \hat Y \rightarrow \{ {\rm pt }\},$ one has:
$$
{ \Ind} \, \wt \eth_{\sign} (\hat X ) = \pi^{\partial \hat Y}_*( [\wt \eth_{\sign} ( \hat X )])  \in 
KK^0( \CC , C_r^* \Gamma) \simeq K_0( C^*_r \Gamma).
$$
Now let $C_{ {\partial \hat Y}}({\hat Y}) \subset C({\hat Y})$ denote 
the ideal of continuous functions on $\hat Y$ vanishing on the boundary. Let $i: \partial {\hat Y} \rightarrow {\hat Y}$
denote the inclusion and consider the long exact sequence in $KK (\cdot , C_r^* \Gamma)$ associated 
to the semisplit short exact sequence:

\begin{equation}
0 \rightarrow C_{ \partial {\hat Y}}({\hat Y}) \overset{j}{\rightarrow} C({\hat Y}) \overset{q}{\rightarrow} C(\partial {\hat Y}) \rightarrow 0
\end{equation}
(see Blackadar \cite[page 197]{Blackadar}). In particular, we have the exactness of
$$
KK^1 (C_{ \partial {\hat Y}}({\hat Y}) , C^*_r \Gamma) \overset{\delta}{\rightarrow} KK^0 ( C(\partial {\hat Y}), C^*_r \Gamma) \overset{i_*}{\rightarrow} KK^0 ( C( { \hat Y}), C^*_r \Gamma)
$$ and thus $i_* \circ \delta=0.$
Recall that the conic iterated metric on $\hat Y$ (with product structure near $\partial \hat Y= \hat X$) allows us 
to define a $C^*_r \Gamma-$linear twisted signature operator $\wt \eth_{\sign}$ on $\hat Y$ with 
coefficients in the bundle ${\hat Y}^\prime \times_\Gamma C^*_r \Gamma \rightarrow \hat Y.$ This twisted signature operator
allows us to define a class $[\wt \eth_{\sign}] \in KK^1(C_{\partial {\hat Y}}({\hat Y}), C_r^*\Gamma).$
\begin{lemma} One has $\delta  [\wt \eth_{\sign}]= [2 \wt \eth_{\sign} (\hat X ) ].$
\end{lemma}
\begin{proof} We are using  the proof
of Theorem 5.1 of Higson \cite{Higson}. We can replace ${\hat Y}$ by a 
 collar neighborhood $W$ ($\simeq [0,1[ \times \partial {\hat Y}$).
Consider the differential operator $d:$
$$
d=\begin{pmatrix}
	0&  - i \frac{d}{d x }  \\
	- i \frac{d}{d x }  & 0
	\end{pmatrix}
$$ acting on $[0,1].$ It defines a class in $KK^1 ( C_0 (0,1) , C^*_r \Gamma).$  Recall that the Kasparov product
$[d]\otimes \cdot $ induces an isomorphism:
$$
[d]\otimes \cdot : KK^0(C(\partial {\hat Y}), C^*_r \Gamma) \rightarrow 
KK^1(C_{\partial {\hat Y}}(W), C^*_r \Gamma).
$$
As in \cite{Higson}, the connecting map $\delta:$
$$
KK^1(C_{\partial {\hat Y}}(W), C^*_r \Gamma) \overset{\delta}{\rightarrow} KK^0(C(\partial {\hat Y}), C^*_r \Gamma)
$$ is given by the inverse of $[d]\otimes \cdot $. Denote by $\mathcal{D}_{W}$ the restriction 
of $\wt \eth_{\sign}$ to $W$ and recall that $\hat X= \partial \hat Y$ is even dimensional. Then one checks (using \cite{Higson} and \cite[page 296]{PRW})   that the $KK-$class 
$[\mathcal{D}_{W}]$ is equal to $[d]\otimes 2 [\wt \eth_{\sign} (\hat X) ],$ and one finds that 
$ \delta [\mathcal{D}_{W} ] = 2 [\wt \eth_{\sign} (\hat X) ]$ which proves the result.
\end{proof}
Let $\pi^{ \hat Y}: \hat Y \rightarrow \{ {\rm pt }\}$ denote the constant map.  By functoriality, one has:
$$
\pi_*^{\partial \hat Y} = \pi_*^{\hat Y}  \circ i_*.
$$ Since $i_* \circ \delta=0$, the previous Lemma implies that:
$$
2 \Ind \wt \eth_{\sign} (\hat X)= \pi_*^{\partial {\hat Y} }(  [2 \wt \eth_{\sign} (\hat X) ] )=
\pi_*^{ \partial \hat Y}  \circ \delta ( [\mathcal{D}_{W} ]  )= 
\pi_*^{\hat Y} \circ  i_* \circ \delta ( [\mathcal{D}_{W} ]  )=0.
$$
Therefore, Theorem \ref{thm:bordism} is proved.
\end{proof}

\bigskip

We shall denote by $\Omega^{{\rm Witt, s}}_* (B\Gamma)$ the bordism group in the category of 
{\it smoothly} stratified oriented Witt spaces. This group is
generated by the elements  of the form
$[\hat X,r: \hat X\to B\Gamma]$ where $[\hat X,r: \hat X\to B\Gamma]$ is equivalent to the zero element  
if $\hat X$ is the boundary of a smoothly stratified Witt oriented space $\hat Y$  (as in Theorem \ref{thm:bordism}) such that the map $r$
extends continuously to $\hat Y.$
It follows that the index map
\begin{equation}\label{bordsim-inv}
\Omega^{{\rm Witt,s}}_* (B\Gamma)\to K_*(C^*_r \Gamma)\otimes \bbQ,
\end{equation}
sending $[\hat X,r: \hat X\to B\Gamma]\in \Omega^{{\rm Witt,s}}_* (B\Gamma)$ to the higher index class 
$\Ind (\wt \eth_{\sign})$ (for the twisting bundle $ r^* E\Gamma\times_\Gamma C^*_r \Gamma$),
is well defined. As in the closed case, see \cite{ros-weinberget-at-2}, it might be possible to refine this result 
and show that the index map actually defines a group homomorphism $\Omega^{{\rm Witt,s}}_* (B\Gamma)\to K_*(C^*_r \Gamma)$

Recall that Siegel's Witt-bordism groups $\Omega^{{\rm Witt}}_* (B\Gamma)$ are given in terms of equivalence classes
of pairs $(\hat X,u: \hat X\to B\Gamma)$, 
with $\hat X$ a Witt space  which is not necessarily {\it smoothly} stratified.

We also 
recall that, working with PL spaces, Sullivan \cite{Sullivan} has defined 
the notion of connected KO-Homology $ko_*$ (see also \cite[page 1069]{Se}).
Siegel \cite[Chapter 4]{Se}, building on work of Sullivan and Conner-Floyd,
 has shown that the natural map 
$\Omega^{{\rm SO}}_*(B \Gamma) \otimes_\bbZ \bbQ \rightarrow  
\Omega^{{\rm Witt}}_* (B\Gamma) 
\otimes_\bbZ \bbQ $ is surjective by showing that the natural map $\Omega^{{\rm SO}}_* (B \Gamma) \otimes_\bbZ \bbQ \rightarrow ko_* (B \Gamma) \otimes_\bbZ \bbQ$ is surjective and the natural map (\cite{Se})  $\Omega^{{\rm Witt}}_* (B\Gamma) \otimes_\bbZ \bbQ  
\rightarrow ko_* (B \Gamma) \otimes_\bbZ \bbQ$ is an isomorphism.
We need to extend these results  for the corresponding groups associated with the category of smoothly stratified spaces. 
\begin{proposition} \label{prop:ko}
The natural map $\Omega^{{\rm SO}}_* (B \Gamma) \otimes_\bbZ \bbQ \rightarrow \Omega^{{\rm Witt,s}}_* (B\Gamma) 
\otimes_\bbZ \bbQ$ is surjective.
\end{proposition}
\begin{proof} 
Theorem 4.4 of \cite{Se} is still valid (by inspection) if one works in the category 
of smoothly stratified oriented Witt spaces. Namely, if $\hat X$ is an irreducible smoothly stratified Witt space 
of even dimension $>0$ such that  $w(\hat X)=0$, with $w(\hat X)\in W(\bbQ)$,
 then $\hat X$ is Witt cobordant to 
zero in the category of smoothly stratified Witt spaces. The arguments of \cite[Chapter 4]{Se} 
show that Siegel's map:
$$
\Omega^{{\rm Witt,s}}_* (B\Gamma) \otimes_\bbZ \bbQ  
\rightarrow ko_* (B \Gamma) \otimes_\bbZ \bbQ
$$ is an isomorphism and, using the surjectivity of the  map
$$\Omega^{{\rm SO}}_* (B \Gamma) \otimes_\bbZ \bbQ \rightarrow ko_* (B\Gamma) 
\otimes_\bbZ \bbQ,$$
 one gets the Proposition.
\end{proof}

\section{The homology L-class of a Witt space. Higher signatures.} \label{sec:Lclass}

The homology $L$-class $L_*(\hat X)\in H_* (\hat X,\bbQ)$ of a Witt space
$\hat X$ was defined independently by Goresky and MacPherson \cite{GM},
following ideas of Thom \cite{Thom}, and by Cheeger \cite{Ch}. See also Siegel \cite{Se}.
In this paper we shall adopt the approach of Goresky and
MacPherson.
%
We briefly recall the definition: if $\hat X$ has dimension
$n$, $k\in \bbN$ is such that $2k-1>n$, and $\cN$ denotes the `north pole' of $\bbS^k$, one can show that the map
$\sigma: \pi^k (\hat X)\to \bbZ$ that associates to $[f:\hat X\to S^k]$ the Witt-signature 
of $f^{-1} (\cN)$ is well defined and a group homomorphism. Now, by Serre's theorem,
the Hurewicz map $\pi^k (\hat X)\otimes \bbQ \to H^k (\hat X,\bbQ)$ is an isomorphism 
for $2k-1>n$ and we can thus view the above homomorphism, $\sigma\otimes \Id_{\bbQ}$, as a linear functional in
$\mathrm{Hom} (H^k (\hat X),\bbQ)\simeq H_k (\hat X,\bbQ)$. This defines $L_k (\hat X)\in 
H_k (\hat X,\bbQ)$. The restriction $2k-1>n$ can be removed by crossing with a high dimensional
sphere in the following way. Choose a positive integer $\ell$ such that 
$2(k+\ell) -1 > n +\ell$ and $k + \ell > n.$
Then by the above construction, $L_{k + \ell} (\hat X\times S^\ell)$ is well defined in $H_{k +\ell} (\hat X\times S^\ell , \bbQ).$ 
Since  $k + \ell > n,$ the K\"unneth Theorem shows that there is a natural isomorphism $ I : H_{k +\ell} (\hat X\times S^\ell , \bbQ) 
\rightarrow H_{k } (\hat X , \bbQ).$ One then defines:
$L_k (\hat X):= I ( L_{k+l} (\hat X\times S^\ell) )$.

Once we have a homology $L$-class we can define the higher signatures as follows. 

\begin{definition}\label{def:higher-signatures}
Let $\hat X$ be a Witt space and $\Gamma:=\pi_1 (\hat X)$.
Let $r: \hat X\to B\Gamma$ be a classifying map for the universal cover.
The (Witt-Novikov) higher signatures of $\hat X$ are the collection of rational numbers:
\begin{equation}\label{higher-signatures}
\{ \langle \alpha,r_* L_*( \hat X) \rangle \,,\alpha\in H^* (B\Gamma,\bbQ)\}
\end{equation}
We set $\sigma_\alpha (\hat X):= \langle \alpha,r_* L_*( \hat X) \rangle$.
\end{definition}

The Witt-Novikov higher signatures have already been studied, see for example \cite{curran}.
If $X$ is an oriented  closed compact manifold and $r:X\to B\pi_1 (X)$ is the classifying map, it is not difficult
to show that 
$$ \langle \alpha,r_* L_*(X)\rangle = \langle L(X)\cup r^* \alpha,[X]\rangle \equiv \int L(X)\cup r^* \alpha\,.$$
Thus the above definition is consistent with the usual definition of Novikov higher signatures in the closed case.

The Novikov conjecture in the closed case  is the statement
that all the higher signatures $\{ \langle L(X)\cup r^* \alpha,[X] \rangle \,,\,\alpha\in H^* (B\Gamma,\bbQ)\}$ are homotopy invariants.

The Novikov conjecture in the Witt case 
 is the statement that the Witt-Novikov
 higher signatures  $ \{ \langle \alpha,r_* L_*(\hat X) \rangle \,,\,\alpha\in H^* (B\Gamma,\bbQ)\}$
  are {\bf stratified} homotopy
invariants. Notice that intersection homology is not a homotopy invariant theory;
however, it is a stratified homotopy-invariant theory, see \cite{friedman}.

We shall need to relate the homology $L$-class of Goresky-MacPherson to the signature
class $[\eth_{\sign} ] \in K_*(\hat X)$.

\begin{theorem}\label{theo:cmwu} (Cheeger/Moscovici-Wu)
The topological homology L-class $L_*(\hat X)\in H_* (\hat X,\bbQ)$ is the
image, under the rationalized homology Chern character, of  
the signature K-homology class $[\eth_{\sign} ]_{\bbQ} \in K_*(\hat X)\otimes \bbQ$; in formul\ae
\begin{equation}\label{chern}
\ch_* [\eth_{\sign} ]_{\bbQ} = L_* (\hat X) \quad \text{in} \quad H_* (\hat X,\bbQ).
\end{equation}
\end{theorem}

This result is due to Cheeger, who proved it for
piecewise flat metric of conic type, and to Moscovici-Wu, who gave an alternative argument
valid also for any metric quasi-isometric to such a metric \cite{Ch}, \cite{Mosc-Wu-Witt}.
It is worth pointing out here that our metrics do belong to the class considered in \cite{Mosc-Wu-Witt}.
Notice that Moscovici-Wu prove that the straight Chern
character of $[\eth_{\sign} ]_{\bbQ} \in K_*(\hat X)\otimes \bbQ$ is equal to $L_*(\hat X)\in H_* (\hat X,\bbQ)$;
the straight Chern character has values in Alexander-Spanier homology; the equality with 
$L_*(\hat X)\in H_* (\hat X,\bbQ)$ is obtained using the isomorphism between Alexander-Spanier and singular homology \cite{Mosc-Wu-Witt}.

\section{Stratified homotopy invariance of the index class: the analytic approach}\label{sect:shi}

One key point in all the index theoretic proofs of the Novikov conjecture for closed oriented
manifolds is the one stating the homotopy invariance of the signature index class in $K_* (C^*_r \Gamma)$.
By this we mean that if $r:X\to B\Gamma$ as above, $f: X^\prime\to X$ is a smooth homotopy equivalence and 
$r^\prime:= r\circ f:X^\prime\to B\Gamma$,
 then the index class, in $K_* (C^*_r\Gamma)$, associated to  $\widetilde{\eth}_{\sign}$ (i.e., associated to
 the signature operator on $X$, $\eth_{\sign}$,
twisted by $r^* E\Gamma\times_\Gamma C^*_r\Gamma$) is equal to the one associated to $\widetilde{\eth}_{\sign}^\prime$
(i.e., associated to
 the signature operator on $X^\prime$, $\eth_{\sign}^\prime$,
twisted by $(r^\prime)^* E\Gamma\times_\Gamma C^*_r\Gamma$).
There are two approaches to this fundamental result:
\begin{enumerate}
\item one proves {\it analytically} that $\Ind (\widetilde{\eth}_{\sign} )=\Ind (\widetilde{\eth}_{\sign}^\prime)$ in $K_* (C^*_r \Gamma)$;
\item one proves that the index class is equal to an a priori homotopy invariant, the Mishchenko ($C^*$-algebraic) symmetric signature.
\end{enumerate}

In this section we pursue the first of these approaches. We shall thus
establish the stratified homotopy invariance of the index class on Witt spaces by following ideas from 
  Hilsum-Skandalis \cite{H-S}, where this property is proved for closed compact manifolds. See also \cite{PS}.

\subsection{Hilsum-Skandalis replacement of $f$} $\;$ \newline
If $X$ and $Y$ are closed Riemannian manifolds, and $f:X \to Y$ is a homotopy equivalence, it need not be the 
case that pull-back by $f$ induces a bounded operator in $L^2$. Indeed, suppose $f$ is an embedding and 
$\phi_\eps$ is a function which equals $1$ on the $\eps$ tubular neighborhood of the image of X.
The $L^2$-norm of $\phi_{\eps}$ is bounded by $C \eps^{\mathrm{codim}_YX}$ and hence tends to zero, while $f^* \phi_{\eps} \equiv 1$ on $X$ and so its $L^2$ norm is constant.
Thus the closure of the graph of $f^*$, say over piecewise constant functions, contains an element of the form $(0,1)$, and is not itself the graph of an operator. 

On the other hand, if $f$ is a submersion, and the metric on $X$ is a submersion metric, then $f^*$ clearly does induce a bounded operator on $L^2$. Since the latter property is a quasi-isometry invariant, and any two metrics on $X$ are quasi-isometric, it follows that pull-back by a submersion always induces a bounded operator in $L^2$.

As one is often presented with a homotopy equivalence $f$ and interested in properties of $L^2$ spaces, it is useful to follow Hilsum and Skandalis \cite{H-S} and replace pull-back by $f$ by an operator that is bounded in $L^2$. We refer to this operator as the Hilsum-Skandalis replacement of $f^*$ and denote it $HS(f)$.

Such a map is constructed as follows. Consider a disk bundle $\pi_Y:\bbD_Y \to Y$ and the associated pulled back bundle  $f^* \bbD_Y$ 
by the map $f: X \rightarrow Y.$ Denote by
 $\pi_X: f^*\bbD_Y \to X$ the induced projection. Then   $f$ admits a natural lift $\bbD(f)$ such that
\begin{equation*}
	\xymatrix{ f^*\bbD_Y \ar[r]^{\bbD(f)} \ar[d] & \bbD_Y \ar[d] \\ X \ar[r]^{f} & Y }
\end{equation*}
commutes. Moreover, we consider  a (smooth) map $e:\bbD_Y \to Y$ such that $p=e \circ \bbD(f): f^*  \bbD_Y \to Y$ is a submersion, and a choice of Thom form $\cT$ for $\pi_X$.
The Hilsum-Skandalis replacement of $f^*$ is then the map
\begin{equation*}
	\xymatrix @R=1pt @C=1pt
	{ HS(f) = HS_{\cT, f^*\bbD_Y, \bbD_Y, e}(f):& \CI(Y; \Lambda^*) \ar[r] & \CI(X; \Lambda^*) \\
	& u \ar@{|->}[r] & (\pi_X)_* (\cT \wedge p^*u) }
\end{equation*}
Notice that $HS(f)$ induces a bounded map in $L^2$ because $p^*=(e\circ \bbD(f) )^*$ does.

For example, as in \cite{H-S}, one can start with an embedding $j: Y \to \bbR^N$ and a tubular neighborhood $\cU$ of $j(Y)$ such that $j(\zeta) + \bbD \subseteq \cU$,
and then take $\bbD_X = X \times \bbD$, $\bbD_Y = Y \times \bbD$, $\bbD(f) = f \times \id$, and $e(\zeta, v) = \tau(\zeta + v)$ where $\tau: \cU \to Y$ is the projection.
Alternately, one can take $\bbD_Y$ to be the unit ball subbundle of $TY$  and $e(\zeta, v) = \exp_{f(\zeta)}(v)$. 
We will extend the latter approach to stratified manifolds.

In any case, one can show that $HS(f)$ is a suitable replacement for $f^*$. 
Significantly, using $HS(f)$ we will see that the K-theory classes induced by the signature operators of homotopic stratified manifolds coincide.
\subsection{Stratified homotopy equivalences} $\;$ \newline
Let $\hat X$ and $\hat Y$ denote stratified spaces, $X$ and $Y$ their regular parts, and $\cS(X)$ and $\cS(Y)$ the corresponding sets of strata.
Following \cite{friedman} and \cite[Def. 4.8.1 ff]{Kirwan-Woolf} we say that a map $f:\hat X \to \hat Y$ is 
{\em stratum preserving} if
\begin{equation*}
	S \in \cS(\hat Y) \implies f^{-1}(S) \text{ is a union of strata of $X$}
\end{equation*}
and {\em codimension preserving} if also
\begin{equation*}
	\mathrm{codim} f^{-1}(S) = \mathrm{codim} S.
\end{equation*}
We will say that a map is {\em strongly stratum preserving} if it is both stratum and codimension preserving.

In these references, a {\em stratum-preserving homotopy equivalence} between stratified spaces is a strongly stratum preserving map $f:\hat X \to \hat Y$ such that there exists another strongly stratum preserving map $g: \hat Y \to \hat X$ with both $f \circ g$ and $g \circ f$ homotopic to the appropriate identity maps through strongly stratum preserving maps. It is shown that stratum-preserving homotopy equivalences induce isomorphisms in intersection cohomology.

Notice that the existence of a homotopy equivalence between {\em closed} manifolds implies that the manifolds have the same dimension, so it is natural to impose a condition like strong stratum preserving on stratified homotopy equivalences. We shall also assume
that $f$ is a {\it smooth} strongly stratified map, see Definition \ref{correct-smooth},
 and that it is a {\it smooth }
strongly stratified homotopy equivalence. (Once again, in the index-theoretic approach
to the Novikov conjecture on closed manifolds, this additional  hypothesis 
of smoothness is also made.)

{\it We shall often omit the reference to the smoothness of $f$, given that our methods
are obviously suited for these kind of maps only.}

A smooth strongly stratified map lifts, 
according to Definition \ref{correct-smooth},
to a smooth map between the resolutions of the stratified spaces $\wt f: \wt X \to \wt Y$ 
preserving the iterated boundary fibration structures.
In particular, $\wt f$ is a $b$-map and the differential of $\wt f$ sends tangent vectors to the boundary fibrations
of $\wt X$ 
to tangent vector  to the boundary fibrations of  $ \wt Y$

%
This implies that there exist linear maps 
\begin{equation*}
	f^*: \CI(Y; \Ie \Lambda^*) \to \CI(X; \Ie \Lambda^*), \Mand
	f^*: \CI(Y; \Iie \Lambda^*) \to \CI(X; \Iie \Lambda^*),
\end{equation*}
though, as on a closed manifold, these do not necessarily induce bounded maps in $L^2$.


\subsection{Hilsum-Skandalis replacement on complete edge manifolds}$\;$ \newline
Suppose ${\wt X}$ and ${\wt Y}$ are both manifolds with boundary and boundary fibrations
\begin{equation*}
	\phi_{\wt X}: \pa {\wt X} \to H_{\wt X}, \quad
	\phi_{\wt Y}: \pa {\wt Y} \to H_{\wt Y}.
\end{equation*}
Let $X$ and $Y$ denote the interiors of $\wt X$ and $\wt Y$ respectively. 

Endow ${\wt Y}$ with a complete edge metric $\tilde{g}=\rho^{-2} g$ \eqref{def:completemet} such that $g$ 
is adapted in the sense of Proposition \ref{prop:rescaling}.
Let $\bbD_{Y} \subseteq {}^eT{Y}$ be the edge vector fields on ${Y}$ with pointwise length bounded by one, and let $\exp: \bbD_{Y} \to {Y}$ be the exponential map on ${Y}$ with respect to the edge metric.
The space $\bbD_{Y}$ is itself an (open) edge manifold with boundary fibration $\phi_{\bbD_{Y}}: \pa \bbD_{Y} \to \pa {\wt Y} \to H_{\wt Y}$.
Notice that $\exp$ extends to a b-map that sends fibers of $\phi_{\bbD_{\wt Y}}$ to fibers of $\phi_{\wt Y}$ and hence induces a map
\begin{equation*}
	\exp_*:{}^eT\bbD_{Y} \to {}^eT{ Y}
\end{equation*}
which is seen to be surjective.

Let $f: {\wt X} \to {\wt Y}$ be a smooth b-map that sends fibers of $\phi_{\wt X}$ to fibers of $\phi_{\wt Y}$.
Pulling-back the bundle $\bbD_{Y} \to Y$ to $X$ gives a commutative diagram
\begin{equation}\label{PullBackDiag}
	\xymatrix{
	f^*\bbD_{Y} \ar[r]^{\bar f} \ar[d]^{\pi_{X}} & \bbD_{Y} \ar[d]^{\pi_{Y}} \\ 
	X \ar[r]^f & Y }
\end{equation}
which we use to construct the Hilsum-Skandalis replacement for pull-back by $f$.
Namely, define  $e = \exp: \bbD_{Y} \to {Y}$, let $\cT$ the pull-back by $\bar f$ of a Thom form for $ \bbD_{Y}$, and let 
\begin{equation}\label{def-of-HS}
	HS(f) = (\pi_{X})_* ( \cT \wedge p^* ): 
	\CI( {Y}; \Ie \Lambda^*) \to \CI({X}; \Ie\Lambda^*)
\end{equation}
with $p = e \circ \bbD(f)$. Observe that $p$ is a proper submersion and hence a fibration.
Then, as above, $HS(f)$ induces a map between the corresponding $L^2$ spaces.

The generalization to manifolds with corners and iterated fibrations structures is straightforward: we just replace the edge tangent bundle with the iterated edge tangent bundle. 
Indeed, it is immediate that if $\bbD_{Y} \subseteq \Ie T Y$ is the set of iterated edge vector fields on $Y$ with pointwise length bounded by one the exponential map $\exp: \bbD_{Y} \to Y$ with respect to a (complete) iterated edge metric induces a map $\exp_*: \Ie T\bbD_{Y} \to \Ie T Y$. That this map is surjective can be checked locally and follows by a simple induction.
Then given a smooth b-map $f: \wt X \to \wt Y$ with the property that, whenever $H \in \cM_1(\wt X)$ is sent to $K \in \cM_1(\wt Y)$, the fibers of the fibration on $H$ are sent to the fibers of the fibration on $K$,
we end up with a map
\begin{equation*}
	HS(f): \CI(\wt Y, \Ie \Lambda^* ) \to \CI(\wt X, \Ie \Lambda^*)
\end{equation*}
that induces a bounded map between the corresponding $L^2_{\ie}$ spaces.

Next, recall that 
\begin{equation*}
	\CI(\wt Y; \Iie \Lambda^1) = \rho_{\wt Y} \CI(\wt Y; \Ie \Lambda^1) 
\end{equation*}
where $\rho_{\wt Y}$ is a total boundary defining function for $\partial \wt Y$.
Hence, if $f: \wt X \to \wt Y$ induces $f^*: \CI(\wt Y; \Ie \Lambda^1) \to \CI(\wt X; \Ie \Lambda^1)$, it will also induce a map
\begin{equation*}
	f^*: \CI(\wt Y; \Iie \Lambda^1) \to \CI(\wt X; \Iie \Lambda^1)
\end{equation*}
if $f^*(\rho_{\wt Y})$ is divisible by $\rho_{\wt X}$. That is, we want $f$ to map the boundary of $\wt X$ to the boundary of $\wt Y$ ({\em a priori}, it could map a boundary face of $\wt X$ onto all of $\wt Y$).
For maps $f$ coming from pre-stratified maps, this condition holds and hence the map $HS(f)$ induces a bounded map between iterated incomplete edge $L^2$ spaces.
Of course, once $f^*$ induces a map on $\Iie \Lambda^1$, it extends to a map on $\Iie \Lambda^*$.

\subsection{Stratified homotopy invariance of the analytic signature class}$\;$ \newline
Suppose we have a stratum-preserving smooth homotopy equivalence between stratified spaces $f: \hat X \to \hat Y$.
Recall that $X$ and $Y$ denote the regular parts of $\hat X$ and $\hat Y$, respectively. Recall the map $r: \hat Y \to B\Gamma$ 
and the 
 flat bundle $\cV '$ of finitely generated $C^*_r\Gamma$-modules over $\hat Y:$ 
 $$
 \cV '\,=\, C^*_r \Gamma \times_\Gamma r^*(E \Gamma).
 $$ Notice that using the blowdown map $\wt Y \rightarrow \hat Y,$  $\cV '$ induces a flat bundle, still denoted $\cV '$ on $\wt Y.$
Consider  $\cV = f^*\cV'$ the corresponding flat bundle over $\hat X$.  We have a flat connection on $\cV'$, $\nabla_{\cV'}$, over $Y$ (and $\wt Y$) and associated differential $d_{\cV'}$, and
 corresponding connection $\nabla_{\cV}$ and differential $d_{\cV}$ on $X$ (and $\wt X$).
It is straightforward to see that the Hilsum-Skandalis replacement of $f$ constructed above extends to 
\begin{equation*}
	HS(f): \CI( Y; \Iie \Lambda^* \otimes \cV') \to \CI( X; \Iie \Lambda^* \otimes \cV)
\end{equation*}
and induces a bounded operator between the corresponding $L^2$ spaces.

We now explain how the rest of the argument of Hilsum-Skandalis extends to this context.

Suppose $(f_t)_{0\leq t \leq 1}: \hat X \to \hat Y$ is a homotopy of stratum-preserving smooth homotopy equivalences, let $\bbD_{Y}$ be as above.
Assume that $(e_s)_{0\leq s \leq 1}: \bbD_{Y} \to Y$ is a homotopy of smooth maps such that, for any $s \in [0,1]$,  $p_s = e_s \circ \bbD(f_s): f_s^*\bbD_{Y} \to Y$ induces a surjective map on $\iie$ vector fields.
Choose a smooth family of bundle isomorphisms (over $X$) $A_s: f_s^* \bbD_{Y} \longrightarrow f_0^* \bbD_{Y},$ 
$( 0 \leq s \leq 1),$
such that $A_0= \Id.$ Set $\cT_s = A_s^* \cT_0$ where $\cT_0$ is a Thom form for
the bundle $f_0^* \bbD_{Y} \rightarrow X.$
Consider $\nabla$ a flat unitary connection on $\cV '$. It induces an exterior derivative $d_{\cV '}$ on the bundle $\wedge^* T^* \wt Y \otimes \cV '.$
Choose a smooth family of $C^*_r \Gamma-$bundle isomorphism $U_s$ from 
the bundle $(p_s \circ A^{-1}_s)^* \cV ' \rightarrow f_0^* \bbD_{Y}$ onto the bundle 
$p_0^*\cV ' \rightarrow f_0^* \bbD_{Y}$ such that $U_0 =Id.$ Implicit in the statement of the next Lemma is the fact that,  for each $s \in [0, 1]$, $p_s \circ A^{-1}_s$ induces a morphism  from the space of sections of the bundle $ \cV ' \rightarrow Y$ on
the space of sections of the bundle $(p_s \circ A^{-1}_s)^* \cV ' \rightarrow f_0^* \bbD_{ Y}.$

\begin{lemma}\label{lem:Htpy}
Under the above hypotheses and notation,
there exists a bounded operator $\Upsilon: L^2_{\iie}(Y; \Iie \Lambda^* \otimes \cV ') \to L^2_{\iie}(f_0^*\bbD_{Y}; \Iie \Lambda^* \otimes p_0^*\cV ')$ such that 
\begin{equation*}
(  {\rm Id} \otimes U_1) \circ  (\,  \cT_0 \wedge (p_1 \circ A_1^{-1})^*\,) - (\cT_0 \wedge p_0^*) = p_0^*( d_{\cV '}) \Upsilon + \Upsilon d_{\cV '}.
\end{equation*}
\end{lemma}

\begin{proof}
We follow  Hilsum-Skandalis. Consider the map
$$
H: f_0^* \bbD_{ Y} \times [0,1] \rightarrow Y
$$
$$
(x,s) \mapsto H(x,s)= p_s \circ A_s^{-1}(x).
$$
Then the required map $\Upsilon$ is defined by, $ \forall \omega \in  L^2_{\iie}(Y; \Iie \Lambda^* \otimes \cV '),$
 $$
 \Upsilon (\omega) = \int_0^1 i_{\frac{\partial}{\partial t}} \bigl( \,
  U_t \circ (p_t \circ A^{-1}_t)_F^* \otimes ( \cT_0 \wedge H^* \omega) \, \bigr) d t.
 $$
\end{proof}

We need to see how this construction handles composition.
Recall that given $f:\wt X \to \wt Y$ we are taking $\bbD_{Y}$ to be the $\ie$ vectors over $Y$ with length bounded by one,  $\bbD(f): f^*\bbD_{Y} \to \bbD_{Y}$ the natural map \eqref{PullBackDiag}, $e: \bbD_{Y} \to {Y}$ the exponential map, $p= e\circ \bbD(f)$, and $\cT$ a Thom form on $f^* \bbD_{Y}$, and then
\begin{equation*}
	HS(f)u = (\pi_X)_* (\cT \wedge p^*u).
\end{equation*}

Now suppose $\wt X$, $\wt Y$, and $\wt Z$ are manifolds with corners and iterated fibration structures, and 
\begin{equation*}
	\wt X \xrightarrow{h} \wt  Y \xrightarrow{f} \wt Z
\end{equation*}
are smooth b-maps that send boundary hypersurfaces to boundary hypersurfaces and the fibers of boundary fibrations to the fibers of boundary fibrations. Assume that the map  $r: \hat X \rightarrow B \Gamma$ above is of the form 
$r=r_1 \circ f$ for a suitable map $r_1: \hat Z \rightarrow B \Gamma.$ We then get a flat $C^*_r\Gamma-$bundle 
$\cV ''$ over $\hat Z$ (and $\wt Z$) such that $\cV ' = f^* \cV ''.$
Denoting the various $\pi_{\cdot}$'s by $\tau$'s, we have the following diagram
\begin{equation*}
	\xymatrix{ (f \circ p')^* \bbD_{Z} \ar[d]^{\tau_1} \ar[rd]^{\wt p} \ar@/_2pc/[dd]_{\tau} \ar@/^2pc/[rrdd]^{p''} \\
	h^*\bbD_{Y} \ar[d]^{\tau_2} \ar[rd]^{p'} & f^*\bbD_{Z} \ar[d]^{\tau_0} \ar[rd]^p \\
	X \ar[r]^h & Y \ar[r]^f & Z }
\end{equation*}
where $X,$ $Y,$ $Z$ are the interiors of $\wt X,$ $\wt Y,$ $\wt Z,$ $\wt p(\zeta, \xi, \eta) = (p'(\zeta, \xi), \eta)$, and, with $\cT$ standing for a Thom form, we define
\begin{equation*}
\begin{gathered}
	HS(f): \CI(Z; \Iie \Lambda^1 \otimes \cV '') \to \CI(Y; \Iie \Lambda^1 \otimes \cV '), \\
	HS(h): \CI(Y; \Iie \Lambda^1 \otimes \cV ' ) \to \CI(X; \Iie \Lambda^1 \otimes \cV )\\
	HS(f,h): \CI(Z; \Iie \Lambda^1 \otimes \cV '') \to \CI(X; \Iie \Lambda^1\otimes \cV), \\ 
	HS(f)(u) = (\tau_0)_* (\cT_{\tau_0} \wedge p^*u), \quad HS(h)(u) = (\tau_2)_* (\cT_{\tau_2} \wedge (p')^*u) \\
	HS(f,h)(u) = \tau_* (\cT_{\tau_2} \wedge (\wt p)^* \cT_{\tau_0} \wedge (p'')^*u)
\end{gathered}
\end{equation*}

\begin{lemma}\label{lem:Comp}
 $HS(f,h) = HS(h) \circ HS(f)$ and $HS(f, h) - HS( f \circ h) = d_{\cV}\Upsilon + \Upsilon d_{\cV ''}$ for some bounded operator $\Upsilon$.
\end{lemma}

\begin{proof} For simplicity, we give the proof only in the case $\Gamma =\{1\}.$
Using the specific definitions of $\tau_1, \wt p, p', \tau_0$ one checks easily that $(\tau_1)_* \wt p^* = (p')^* (\tau_0)_*.$ 
Therefore, $(\wt p)^* \cT_{\tau_0}$ is indeed a Thom form associated with $\tau_1.$
Since $p''= p \circ \wt p,$ one gets:
$$
	HS(f,h)
	= (\tau_2)_* (\tau_1)_* (\cT_{\tau_2}  \wedge \wt p^*( \cT_{\tau_0} \wedge p^*) ) 
	$$ Then replacing  $(\tau_1)_* \wt p^*$ by $(p')^* (\tau_0)_*$  one gets:
	
	$$HS(f,g)= (\tau_2)_* (\cT_{\tau_2}  \wedge (p')^* ( (\tau_0)_*  (\cT_{\tau_0} \wedge (p)^*) ) )
	= HS(h) \circ HS(f).
	$$

Next, notice that the maps
\begin{equation*}
	(t; \zeta, \xi, \eta) \mapsto \exp_{f ( \exp_{h(\zeta)} (t\xi) )} (\eta)
\end{equation*}
are a homotopy between $p'': (f \circ p')^*\bbD_{Z} \to Z$ and $\hat p: (f \circ h)^*\bbD_{Z} \to Z$  
within submersions. Hence we can use the previous lemma to guarantee the existence of $\Upsilon$. 
\end{proof}

Instead of the usual $L^2$ inner product, we will consider the quadratic form
\begin{equation*}
\begin{gathered}
	Q_{X}: \CI( X ; \Iie \Lambda^* \otimes \cV ) \times \CI( X ; \Iie \Lambda^* \otimes \cV ) \to C^*_r\Gamma \\
	Q_{X}(u,v) = \int_{X} u \wedge v^*
\end{gathered}
\end{equation*}
and also the analogous $Q_{Y}$, $Q_{\bbD_{Y}}$, $Q_{f^*\bbD_{Y}} $. Recall that any element of $\CI(X ; \Iie \Lambda^* \otimes \cV )$ vanishes at the boundary of $X$ so that $Q_{X}$ is indeed well defined. (We point out that the corresponding quadratic form in  Hilsum-Skandalis \cite[page 87]{H-S} is given by $ i^{ |u| (n- |u|)} Q_{X}(u,v).$)
We denote the adjoint of an operator $T$ with respect to $Q_{X}$ (or $Q_{Y}$) by $T'$. Thus, for instance, $d_{\cV}' = - d_{\cV}$.

From Theorem \ref{theo:kk}, we know that the signature data on $\hat X$ defines an element of $K_{\dim X}(C^*_r\Gamma)$ and similarly for the data on $\hat Y$.
Hilsum and Skandalis gave a criterion for proving that two classes are the same which we now employ.

\begin{proposition}\label{prop:h-ska} Consider a stratum-preserving homotopy equivalence 
$f: \hat X \rightarrow \hat Y,$ where $\dim \hat X=n $ is even.  Denote still by $f$ the induced map $\wt X \rightarrow \wt Y.$
The bounded operator 
\begin{equation*}
	HS(f):L^2_{\iie}(Y; \Iie \Lambda^* \otimes \cV') 
	\to L^2_{\iie}(X; \Iie \Lambda^* \otimes \cV)
\end{equation*}
satisfies the following properties:
\begin{itemize}
 \item [a)] $HS(f) d_{\cV'} = d_{\cV} HS(f)$ and $HS(f)(\Dom d_{\cV'}) \subseteq \Dom d_{\cV}$
 \item [b)] $HS(f)$ induces an isomorphism $HS(f): \ker d_{\cV'}/ \Im d_{\cV'} \to \ker d_{\cV} / \Im d_{\cV}$
 \item [c)] There is a bounded operator $\Upsilon$ on a Hilbert module associated to $Y$ 
 such that $\Upsilon(\Dom d_{\cV'}) \subseteq \Dom d_{\cV'}$ and
	$\Id - HS(f)'HS(f) = d_{\cV'} \Upsilon + \Upsilon d_{\cV'}$
 \item [d)]
	There is a bounded self-adjoint involution $\eps$ on $Y$ such that $\eps(\Dom d_{\cV'}) \subseteq \Dom d_{\cV'}$, 
	which commutes with $\Id - HS(f)'HS(f)$ and anti-commutes with $d_{\cV'}$.
\end{itemize}
Hence the signature data on $\hat X$ and $\hat Y$ define the same element of $K_{0}(C^*_r\Gamma)$. 
\end{proposition}

\begin{proof}
The final sentence follows from ($a$)-($d$) and Lemma 2.1 in Hilsum-Skandalis \cite{H-S}.

In Section 7 we showed that the signature operator has a unique closed extension, it follows that so do $d_{\cV}$ and $d_{\cV'}$ (see, e.g., \cite[Proposition 11]{Hunsicker-Mazzeo}). Since this domain is the minimal domain, as soon as we know that an operator is bounded in $L^2_{\iie}$ and commutes or anticommutes with these operators, we know that it preserves their domains.

a) Since $HS(f)$ is made up of pull-back, push-forward, and exterior multiplication by a closed form, $HS(f)d_{\cV'} = d_{\cV}HS(f)$. 

b) From ($a$) we know that $HS(f)$ induces a map $\ker d_{\cV'}/ \Im d_{\cV'} \to \ker d_{\cV} / \Im d_{\cV}$.
Let $h$ denote a homotopy inverse of $f$ and consider
\begin{equation*}
	HS(h): L^2_{\iie}(X; \Iie \Lambda^* \otimes \cV) 
	\to L^2_{\iie}(Y; \Iie \Lambda^* \otimes \cV').
\end{equation*}
We know from Lemma \ref{lem:Comp} that $HS(f \circ h)$ and $HS(h) \circ HS(f)$ induce the same map in cohomology and, from Lemma \ref{lem:Htpy} that $HS(f \circ h)$ induces the same map as the identity. Since the same is true for $HS(f \circ h)$ we conclude that $HS(h)$ and $HS(f)$ are inverse maps in cohomology and hence each is an isomorphism.

c) Recall that $p: f^* \bbD_{Y} \rightarrow Y$, being a proper submersion, is a fibration. Choose a Thom form 
$\wt \cT$ for the fibration $\pi_{Y}: \bbD_{Y} \rightarrow Y$ so that $\bbD_{Y} (f)^* \wt  \cT$ defines a Thom form 
for the fibration $\pi_{X}: f^* \bbD_{Y} \rightarrow X.$
These two facts allow us to carry out the following computation, where $u \in C^\infty (\wt Y; \Iie \Lambda^* \otimes \cV')$ and 
$v \in C^\infty (X; \Iie \Lambda^* \otimes \cV).$
\begin{equation*}
\begin{split}
	Q_{X}(HS(f)u, v) 
	&= Q_{X}\lrpar{ (\pi_{X})_*( \bbD_{Y} (f)^* \wt \cT \wedge p^*u), v } \\
	&= Q_{f^*\bbD_{Y}} (\bbD_{Y} (f)^* \wt \cT \wedge p^*u, \pi_{X}^*v ) \\
	&= (-1)^{ n(n-| v |)} Q_{f^*\bbD_{Y}} ( p^*u, \bbD_{Y} (f)^* \wt \cT \wedge \pi_{X}^*v ) \\
	&= (-1)^{ n(n-| v |)}Q_{Y} ( u, p_* (\bbD_{Y} (f)^* \wt \cT \wedge \pi_{X}^*v ) ).
\end{split}
\end{equation*}
Since $n$ is even this shows that 
$HS(f)' v =  p_* (\bbD_{Y} (f)^* \wt \cT \wedge \pi_{X}^*v )$
and hence
\begin{equation*}
	HS(f)'HS(f)u = p_* ( \bbD_{Y} (f)^* \wt \cT \wedge \pi_{X}^*  (\pi_{X})_*((\bbD_{Y} (f)^* \wt \cT \wedge p^*u) ) ).
\end{equation*}
Next one checks easily that, for any differential form $ \omega $ on $\bbD_{Y}$,
$$ \bbD_{Y} (f)^*  \pi^*_{Y}   (\pi_{Y})_* \omega =  \pi^*_{X}  (\pi_{X})_* \bbD_{Y} (f)^* \omega.
$$ 
and so, from the identity  $p^* = \bbD_{Y} (f)^* e^*,$
\begin{equation*}
	HS(f)'HS(f)u = (e \circ \bbD_{Y} (f))_* ( \bbD_{Y} (f)^*(\,  \wt \cT \wedge \pi_{Y}^*  (\pi_{Y})_*( \wt \cT \wedge e^*u) ) ).
\end{equation*}

Now observe that  $\bbD_{Y} (f): f^* \bbD_{Y} \rightarrow  \bbD_{Y},$ being a homotopy equivalence of manifolds with corners,
sends the relative fundamental class of $f^* \bbD_{Y}$ to the relative fundamental class of $ \bbD_{Y}$ and so
$$
Q_{f^* \bbD_{Y}} ( \bbD_{Y} (f)^*\alpha  , \bbD_{Y} (f)^* \beta ) =Q_{ \bbD_{Y}} (\alpha  ,  \beta ).
$$ From this identity, the previous equation, and the fact that $e$ induces a fibration, one checks easily that
$$
 Q_{Y} ( HS(f)'HS(f)u , w)= Q_{Y} ( e_* (\, \wt \cT \wedge \pi_{Y}^*  (\pi_{Y})_*( \wt \cT \wedge e^*u) ) , w)
$$
and hence
$$
HS(f)'HS(f)u = e_* (\,\wt \cT \wedge  \pi_{Y}^*  (\pi_{Y})_*( \wt \cT \wedge e^*u) ).
$$
Finally, $e$ is homotopic to $\pi_{Y}$, and since
$$
(\pi_{Y})_* (\, \wt \cT \wedge \pi_{Y}^*  (\pi_{Y})_*( \wt \cT \wedge \pi_{Y}^*u) )= (\pi_{Y})_* (\wt \cT \wedge \pi_{Y}^*u) = u,
$$
Lemma \ref{lem:Htpy}, $\Id - HS(f)'HS(f) = d_{\cV'} \Upsilon + \Upsilon d_{\cV'}$ as required.

d) It suffices to take $\eps u = (-1)^{|u|} u$.
\end{proof}

\noindent {\bf Remark.} Consider now the case of   an odd dimensional Witt space $ \hat{X} $ endowed with an edge adapted iterated metric $g$ and a reference
map $r: \hat{X} \rightarrow B \Gamma.$ We have defined in Section 7  the higher signature index class 
$ \Ind\, (\wt \eth_{\sign}) \in KK_1(\bbC, C^*_r \Gamma)\simeq K_1(C^*_r \Gamma)  $ associated to the twisted signature 
operator 
defined by the data $ (\hat{X} , g, r) $. Recall that there is a suspension  isomorphism $\Sigma: K_1(C^*_r \Gamma) \leftrightarrow \wt K _0(C^*_r \Gamma \otimes C(S^1))$ which is induced by taking the
Kasparov product with the Dirac operator of $S^1$. Consider 
the even dimensional Witt space $\hat{X} \times S^1$ endowed with the obvious stratification and with  the reference map 
$$
 r \times \Id_{S^1}: \hat{X} \times S^1 \rightarrow B ( \Gamma \times \mathbb{Z}) \simeq B \Gamma \times S^1.
$$ 
As explained in \cite[p. 624]{LLP}, \cite[\S 3.2]{LPJFA}, the suspension of  the odd index class 
$ \Ind\, (\wt \eth_{\sign}) \in KK_1(\bbC, C^*_r \Gamma)\simeq K_1(C^*_r \Gamma) $
is equal to the even signature index class associated to the data
$ (\hat{X} \times S^1, g \times (d \theta)^2 , r  \times \Id_{S^1} ) .$ 
If now $f:  \hat{X}\to   \hat{Y} $  is  a stratified homotopy equivalence of odd dimensional
Witt spaces, then $f$ induces a stratified homotopy equivalence from $\hat{X}\times S^1$ to    
$\hat{Y}\times S^1$. By the previous Proposition the signature index classes of
$\hat{X}\times S^1$ and     
$\hat{Y}\times S^1$ are the same. Then using the suspension isomorphism $\Sigma$,  we deduce finally that the odd 
signature index classes associated to $\hat{X}$ and $\hat{Y} $ are the same.
Thus, the (smooth) stratified homotopy invariance of the signature index class is established
for Witt spaces of arbitrary dimension.

\section{Assembly map and stratified homotopy invariance of  higher signatures}\label{section:novikov}
Consider  the assembly map $\beta: K_* (B\Gamma) \to K_*(C^*_r \Gamma)$.
The rationally injectivity  of this map is known as the  strong Novikov conjecture for $\Gamma$. In the closed case
it implies that the Novikov higher signatures are oriented homotopy invariants. The rational injectivity
of the assembly map is still unsettled in general, although it is known to hold for large classes
of discrete groups; for closed manifolds having these fundamental  groups the higher signatures are thus homotopy invariants.
The following  is the main topological result of this paper:

\begin{theorem}\label{theo:main}
Let $\hat X$ be an oriented Witt space,  $r: \hat X\to B\pi_1 (\hat X)$ the classifying map
for the universal cover, and let $\Gamma:= \pi_1 (\hat X)$.
If the assembly map $K_* (B\Gamma) \to K_*(C^*_r \Gamma)$
is rationally injective, then the Witt-Novikov  higher signatures
$$\{ \langle \alpha, r_* L_* (\hat X) \rangle, \alpha\in H^* (B\Gamma,\bbQ) \}$$
are stratified homotopy invariants.
\end{theorem}

 \begin{proof}
The proof proceeds in four steps and is directly inspired by Kasparov's proof in the closed case, see
for example \cite{Kasparov-contemporary}  and the references therein:

   \begin{enumerate}
  \item Consider  $({\hat X}', r':{\hat X}'\to B\Gamma)$
 and $(\hat X,r: \hat X\to B\Gamma)$, with $r=r'\circ f$ and  $f: \hat X\to {\hat X}'$
 a stratified homotopy equivalence between (smoothly stratified) oriented Witt spaces. Denote by $ \wt \eth_{\sign}' $ the twisted signature operator associated to
 $({\hat X}', r':{\hat X}'\to B\Gamma)$.
 We have proved that
 $$ \Ind (\wt \eth_{\sign}) = \Ind (\wt \eth_{\sign}')  \quad \text{in}\quad K_* (C^*_r \Gamma)\otimes \bbQ  \,.$$
 \item We know that 
 the assembly map sends $r_* [\eth_{\sign}]\in K_* (B\Gamma)$ to the Witt index class $\Ind (\wt \eth_{\sign}).$
 More explicitly:
 $$\beta (r_* [\eth_{\sign}])= \Ind (\wt \eth_{\sign})\quad\text{in}\quad K_*(C^*_r \Gamma)\otimes \bbQ$$
  \item We deduce from the assumed rational injectivity of the assembly map  that
 $$ r_* [\eth_{\sign}] = (r')_* [\eth_{\sign}']\quad \text{in}\quad K_* (B\Gamma)\otimes \bbQ.$$
 \item Since we know from Cheeger/Moscovici-Wu that $\Ch_* (r_* [\eth_{\sign}])=r_* (L_* (\hat X))$
 in $ H_* (B\Gamma,\bbQ)$ we finally get that 
 $$ r_* (L_* (\hat X)) = (r')_* (L_* ({\hat X}'))\quad \text{in}\quad H_* (B\Gamma,\bbQ)$$
 which obviously implies the stratified homotopy invariance of  the higher signatures
 $\{ <\alpha, r_* L_* (\hat X)>, \alpha\in H^* (B\Gamma,\bbQ) \}$.  
\end{enumerate}
\end{proof}

Examples of discrete groups for which the assembly map is rational injective include:
amenable groups,
discrete subgroups of  Lie groups with a finite number of connected components, Gromov hyperbolic groups, discrete groups acting properly
on bolic spaces, 
countable subgroups of $GL(\bbK)$ for $\bbK$ a field.

\section{The symmetric signature on Witt spaces}\label{section:symmetric}

\subsection{The symmetric signature in the closed case} $\;$ \newline
Let $X$ be a closed orientable manifold and let $r:X\to B\Gamma$ be a classifying map
for the universal cover.
The symmetric signature of Mishchenko, $\sigma (X,r)$,
is a purely topological object \cite{Mish1}. In its most
sophisticated presentation, it  is an element in  the $L$-theory groups $L^* (\bbZ\Gamma)$.
 In general one can define the symmetric signature of any algebraic Poincar\'e complex, i.e.,
a cochain complex of finitely generated $\bbZ\Gamma$-modules satisfying a kind
of Poincar\'e duality. The Mishchenko symmetric signature corresponds to the choice
of the Poincar\'e complex defined by the cochains on the universal cover.
In the treatment of the Novikov conjecture one is in fact interested in a 
less sophisticated invariant, namely the image of $\sigma (X,r)\in L^* (\bbZ\Gamma)$
 under the natural map $\beta_{\bbZ}: L^* (\bbZ\Gamma)\to
L^* (C^*_r\Gamma)$. Recall also that there is a natural isomorphism $\nu: L^* (C^*_r\Gamma)\rightarrow
K_* (C^*_r \Gamma)$ (which is in fact valid for any $C^*$-algebra). The $C^*$-algebraic symmetric
signature is, by definition, the element $\sigma_{_{C^*_r \Gamma}}(X,r) := \nu (\beta_{\bbZ} (\sigma (X,r))$;
thus $\sigma_{_{C^*_r \Gamma}}(X,r)\in K_* (C^*_r \Gamma)$.
The following result, due to Mishchenko and Kasparov, generalizes the equality between the numeric 
index of the signature operator and the topological signature. With the usual notation:
\begin{equation}\label{equality}
\Ind (\wt \eth_{\sign})=\sigma_{_{C^*_r \Gamma}}(X,r)\in K_* (C^*_r \Gamma)
\end{equation}
As a corollary we see that the signature index class 
is a homotopy invariant; this is the topological approach
to the homotopy invariance of the signature index class that we have mentioned in
the introductory remarks in  Section \ref{sect:shi}. The equality of the $C^*$-algebraic symmetric signature with the 
signature index class (formula (\ref{equality}) above) can be
restated as saying that the following diagram is commutative
\begin{equation}
  \label{diagram}
  \begin{CD}
   \Omega^{{\rm SO}}_* (B\Gamma) @>{{\rm Index}}>> K_i (C^*_r \Gamma)\\
    @VV{\sigma}V  @V{\nu^{-1}}VV\\
  L^* (\bbZ\Gamma)  @>{\beta_{\bbZ}}>> L^*(C^*_r \Gamma).
  \end{CD}
\end{equation}
where $i\equiv *$ mod 2.

\subsection{The symmetric signature on Witt spaces.} $\;$ \newline
The middle perversity intersection homology groups of a Witt space do satisfy
Poincar\'e duality over the rationals. Thus, it is natural to expect that for  a Witt space
$\hat X$ endowed with a reference map $r:\hat X\to B\Gamma$ it should be possible
to define a symmetric signature $\sigma^{{\rm Witt}}_{\bbQ \Gamma}(X,r) \in L^* (\bbQ \Gamma).$
And indeed,  the definition of symmetric signature
in the Witt context, together with its expected properties, such as Witt bordism invariance,
does appear in the literature, see for example  \cite{Wei},  \cite{Chang},  \cite{Wei-higher-rho}.
However, no rigorous account of this definition was given in these references,
which is unfortunate, given that things are certainly more complicated 
 than in the smooth
case and for diverse reasons that for the sake of brevity we shall not go into.

Fortunately, in a recent paper Markus Banagl \cite{banagl-msri} has given a rigorous  definition
of the symmetric signature on Witt spaces\footnote{Banagl actually concentrates on the more restrictive class of IP spaces, for which an integral symmetric signature,
i.e. an element in $L^* (\bbZ\Gamma)$, exists; it is easy to realize that his construction can be given for the larger
class of Witt spaces, producing, however, an element in $L^* (\bbQ \Gamma)$.} using  surgery techniques as well as previous results
of Eppelmann \cite{Eppelmann}. Banagl's symmetric signature is an element $\sigma^{\rm Witt}_{\bbQ\Gamma} (\hat X,r)\in L^* (\bbQ \Gamma)$;
we refer directly to Banagl's interesting article for the definition and only point out that directly from his construction we can
conclude that

\begin{itemize}
\item the symmetric signature $\sigma^{\rm Witt}_{\bbQ\Gamma}  (\hat X,r)$ is equal to (the rational) Mishchenko's symmetric signature
if $\hat X$ is a closed compact manifold;
\item the Witt symmetric signature is a Witt bordism invariant; it defines a group homomorphism
$\sigma^{\rm Witt}_{\bbQ\Gamma} : \Omega^{\rm Witt}_* (B\Gamma) \to L^* (\bbQ \Gamma)$.
\end{itemize}

On the other hand, it is not known whether Banagl's symmetric signature $\sigma^{\rm Witt}_{\bbQ\Gamma} (\hat X,r)$
 is a stratified homotopy invariant.

We define the $C^*$-algebraic Witt symmetric signature as the image of $\sigma^{\rm Witt}_{\bbQ\Gamma}  (\hat X,r)$
under the composite 
$$L^* (\bbQ \Gamma)\xrightarrow{\beta_{\bbQ}} L^* (C^*_r \Gamma)\xrightarrow{\nu} K_* (C^*_r \Gamma)\,.$$
We denote the $C^*$-algebraic Witt symmetric signature by $\sigma^{{\rm Witt}}_{_{C^*_r \Gamma}}(X,r)$.

\subsection{Rational equality of the Witt symmetric signature and of the signature index class}$\;$ \newline
Our most general goal would be to prove that
there is a commutative diagram
\begin{equation}
  \label{diagram-witt}
  \begin{CD}
   \Omega^{{\rm Witt}}_*(B\Gamma) @>{{\rm Index}}>> K_i (C^*_r \Gamma)\\
    @VV{\sigma^{\rm Witt}_{\bbQ\Gamma}}V  @V{\nu^{-1}}VV\\
  L^* (\bbQ\Gamma)  @>{\beta_{\bbQ}}>> L^*(C^*_r \Gamma).
  \end{CD}
\end{equation}
or, in formul\ae
$$\sigma^{{\rm Witt}}_{_{C^*_r \Gamma}}(X,r)= \Ind (\wt \eth_{\sign}) \quad \text{in} \quad  K_i (C^*_r \Gamma)$$
with $\Ind(\wt \eth_{\sign})$ the signature index class decribed in the previous sections.
We shall be happy with a little less, namely the rational equality.

\begin{proposition}\label{prop:magic}
Let $\sigma^{{\rm Witt}}_{C^*_r \Gamma}(X,r)_{\bbQ}$ and $ \Ind (\wt \eth_{\sign})_{\bbQ} $ be
the rational classes,  in the rationalized K-group $K_i (C^*_r \Gamma)\otimes \bbQ$, defined by the Witt symmetric signature 
and by the signature index class.
Then 
\begin{equation}\label{rational}
\sigma^{{\rm Witt}}_{C^*_r \Gamma}(X,r)_{\bbQ}= \Ind (\wt \eth_{\sign})_{\bbQ}\quad\text{in}\quad K_i (C^*_r \Gamma)\otimes \bbQ
\end{equation}
\end{proposition}
\begin{proof}
We already  know from \cite{banagl-msri} that the rationalized symmetric signature 
defines  a homomorphism from $(\Omega^{{\rm Witt}}_*(B\Gamma))_{\bbQ}$ to $K_i (C^*_r \Gamma)\otimes \bbQ$.
However, it also clearly defines a homomorphism $(\Omega^{{\rm Witt},s}_*(B\Gamma))_{\bbQ}\to  K_i (C^*_r \Gamma)\otimes \bbQ$,
exactly as the signature index class.
For notational convenience, let $\mathcal{I}: (\Omega^{{\rm Witt},s}_*(B\Gamma))_{\bbQ}\to  K_i (C^*_r \Gamma)\otimes \bbQ$
be the (Witt) signature  index morphism; let $\mathcal{I}': (\Omega^{{\rm Witt},s}_*(B\Gamma))_{\bbQ}\to  K_i (C^*_r \Gamma)\otimes \bbQ$
be the (Witt) symmetric signature morphism. We want to show that 
$$\mathcal{I} = \mathcal{I} '\,.$$
We know from Proposition \ref{prop:ko}  that the natural map 
$\Omega^{{\rm SO}}_*(B\Gamma) \rightarrow  \Omega^{{\rm Witt},s}_*(B\Gamma)$ induces   a rational surjection
$$  s:  ( \Omega^{{\rm SO}}_*(B\Gamma) )_{\bbQ}\rightarrow  (\Omega^{{\rm Witt},s}_*(B\Gamma))_{\bbQ}. $$
In words, a  smoothly stratified Witt space $X$ with reference map $r:X\to B\Gamma$ is smoothly stratified
Witt bordant to $k$-copies of a closed oriented
compact manifold $M$ with reference map $\rho: M\to B\Gamma$.
Moreover, we remark that the Witt  index classes and the Witt symmetric signature of an oriented  closed compact manifold
coincide with the classic signature index class and the Mishchenko symmetric signature. 
Then
$$\mathcal{I} ([X,r])= \mathcal{I} (k[M,\rho])=\mathcal{I}' (k[M,\rho]) = \mathcal{I}' ([X,r])\,$$
with the first and third equality following from the above remark and the second equality obtained using 
the fundamental result of Kasparov and Mishchenko on closed manifolds.
The proof is complete.
\end{proof}

The above Proposition  together with Proposition \ref{prop:h-ska} implies at once the following result:

\begin{corollary}\label{cor:shi-of-banagl}
The $C^*$-algebraic symmetric signature defined by Banagl is a rational stratified 
homotopy invariant.
\end{corollary}

This Corollary does not seem to be obvious from a  purely topological point of view.
We add that very recently Friedman and McClure have given an alternative definition
 of symmetric signature on Witt spaces; while its relationship with Banagl's definition
 is for the time being unclear, we point out that the symmetric signature of 
 Friedman and McClure is a stratified homotopy invariant; moreover, with the same proof
 given above, its image in $K_* (C^*_r \Gamma)$ is rationally equal to our signature index class.

\section{Epilogue}\label{section:epilogue}

Let $\hat X$ be an orientable Witt pseudomanifold with fundamental group $\Gamma$. 
We endow the regular part of $\hat X$ with an adapted iterated edge metric $g$ (Proposition \ref{prop:rescaling}).
Let $\hat X ^\prime$ 
be a Galois $\Gamma$-covering and $r: \hat X\to B\Gamma$ a classifying map for $ \hat X ^\prime$. 
We now restate once more  the signature package for the pair $(\hat X,r:\hat X\to B\Gamma)$ indicating precisely where 
the individual items have been established in this paper.

\begin{list}
 {(1)} \item The signature operator defined by the edge (adapted) iterated  metric $g$
with values in the Mishchenko bundle $r^* E\Gamma \times_\Gamma C^*_r\Gamma$
defines  a signature index class  $\Ind (\wt \eth_{\sign})\in K_* (C^*_r \Gamma)$, $* \equiv \dim X \;{\rm mod}\; 2$.
{\it Established in Theorem \ref{theo:kk}.}
 \end{list}

\begin{list}
 {(2)} \item The signature index class is a (smooth) Witt bordism invariant; more precisely  it defines a group homomorphism 
$\Omega^{{\rm Witt},s}_* (B\Gamma) \to K_* (C^*_r \Gamma)\otimes \bbQ$.
{\it This is Theorem \ref{thm:bordism}, together with \eqref{bordsim-inv}.}
\end{list}

\begin{list}
 {(3)}\item The signature index class is a stratified homotopy invariant.
{\it Proposition \ref{prop:h-ska}.}
\end{list}

\begin{list}
 {(4)}
\item  There is a  K-homology signature class $[\eth_{\sign}]\in K_* (X)$ whose Chern
 character is, rationally, the  homology L-Class of Goresky-MacPherson.
  {\it Theorem \ref{theo:k-homology} and Theorem \ref{theo:cmwu}.}
  \end{list}
 
 \begin{list}
 {(5)}
\item The assembly map $\beta: K_* (B\Gamma)\to K_* (C^*_r\Gamma)$ sends the
class $ r_* [\eth_{\sign}]$ into $\Ind (\wt \eth_{\sign})$. 
{\it Corollary \ref{cor:assembly}.}
\end{list}

\begin{list}
 {(6)}
\item If  the assembly map  is rationally injective one  can deduce from the above results the homotopy invariance of the
 Witt-Novikov higher signatures. {\it Theorem \ref{theo:main}.}
\end{list}

 \begin{list}
 {(7)} \item There is a topologically defined $C^*$-algebraic symmetric signature 
 $\sigma^{{\rm Witt}}_{C^*_r \Gamma}(X,r)$\\$\in  K_* (C^*_r \Gamma)$ which is equal to the analytic index class
 $ \Ind (\wt \eth_{\sign})$ rationally. 
{\it  This is Banagl's construction together with our Proposition \ref{prop:magic}.}
   \end{list}


\end{document}